\documentclass[preprint,3p]{elsarticle}

\usepackage{amssymb}
\usepackage{amsmath}
\usepackage{amsthm}
\usepackage{bbm}

\usepackage{graphicx}
\usepackage{xcolor}

\usepackage{nameref,hyperref}
\newtheorem{theorem}{Theorem}[section]
\newtheorem{lemma}[theorem]{Lemma}
\newtheorem{corollary}[theorem]{Corollary}

\newtheorem{definition}[theorem]{Definition}

\newtheorem{remark}[theorem]{Remark}
\newcommand{\rmd}{\mathrm{d}}

\newcommand{\mA}{\mathcal{A}}
\newcommand{\mB}{\mathcal{B}}
\newcommand{\mF}{\mathcal{F}}
\newcommand{\mL}{\mathcal{L}}

\newcommand{\mFR}{\mathcal{F}_{\mathbb{R}^3}}
\newcommand{\mFS}{\mathcal{F}_{S^2}}

\newcommand{\mG}{\mathcal{G}}
\newcommand{\mM}{\mathcal{M}}
\newcommand{\mS}{\mathcal{S}}
\newcommand{\tU}{\tilde{U}}
\newcommand{\tW}{\tilde{W}}

\newcommand{\tK}{\tilde{K}}
\newcommand{\tQ}{\tilde{Q}}
\newcommand{\tlambda}{\tilde{\lambda}}

\newcommand{\tUps}{\tilde{\Upsilon}}
\newcommand{\bc}{\mathbf{c}}
\newcommand{\bx}{\mathbf{x}}
\newcommand{\bR}{\mathbf{R}}
\newcommand{\be}{\mathbf{e}}

\newcommand{\by}{\mathbf{y}}
\newcommand{\bn}{\mathbf{n}}
\newcommand{\ba}{\mathbf{a}}
\newcommand{\bd}{\mathbf{d}}

\newcommand{\bD}{\mathbf{D}}
\newcommand{\bN}{\mathbf{N}}
\newcommand{\bY}{\mathbf{Y}}
\newcommand{\bM}{\mathbf{M}}
\newcommand{\bu}{\mathbf{u}}

\newcommand{\bhw}{\mathbf{\hat{w}}}
\newcommand{\hW}{\hat{W}}
\newcommand{\hU}{\hat{U}}
\newcommand{\hK}{\hat{K}}
\newcommand{\hG}{\hat{G}}

\newcommand{\ul}{\mathbf}
\newcommand{\OmegaBF}{\mbox{\boldmath$\Omega$}}
\newcommand{\sumlm}{\sum_{l=0}^{\infty} \sum_{m=-l}^{l}}
\newcommand{\summl}{\sum_{m=-\infty}^{\infty} \sum_{l=|m|}^{\infty}}
\newcommand{\Philmom}{\Phi_{l,m}^{\bomega}}
\newcommand{\PhilmomE}{\Phi_{l,m}^{E,\bomega}}
\newcommand{\Psilmom}{\Psi_{l,m}^{\bomega}}
\newcommand{\Psilpmpom}{\Psi_{l',m'}^{\bomega}}

\newcommand{\lambdalmom}{\lambda^{l,m}_{\bomega}}

\newcommand{\tlambdalmrho}{\tilde{\lambda}^{l,m}_{\rho}}
\newcommand{\lambdalmr}{\lambda^{l,m}_{r}}
\newcommand{\lambdalpmpr}{\lambda^{l',m'}_{r}}
\newcommand{\ylmom}{Y_{\bomega}^{l,m}}
\newcommand{\quot}{\mathbb{R}^3\rtimes S^2}
\newcommand{\QDa}{Q^{\bD,\ba}}
\newcommand{\Ylm}{Y^{l,m}}
\newcommand{\Ylpm}{Y^{l',m}}
\newcommand{\Slm}{S^{l,m}}
\newcommand{\Slmrho}{S^{l,m}_\rho}

\newcommand{\bomega}{{\boldsymbol{\omega}}}
\newcommand{\R}{\mathbb{R}}

\begin{document}

\title{New Exact and Numerical Solutions of the (Convection-)Diffusion Kernels on SE(3)
}

%    Enter the address for every author, even if some are the same.

%    author one information
\author{J.M. Portegies}
\address{}
%\curraddr{}
%\email{j.m.portegies@tue.nl}
%\thanks{}
\address{Department of Mathematics and Computer Science, CASA,
              Eindhoven University of Technology, \\
              Tel.: +31-40-2478634,
              \url{J.M.Portegies@tue.nl}}
              
%    author two information
\author{R. Duits}
\address{Department of Mathematics and Computer Science, CASA,
              Eindhoven University of Technology, \\
              Tel.: +31-40-2472859,
              \url{R.Duits@tue.nl}}
%\curraddr{}
%\email{r.duits@tue.nl}
%\thanks{}

%    \subjclass is required.
%\subjclass[2000]{Primary }
%    The 2010 edition of the Mathematics Subject Classification is
%    now available.  If you are citing a classification from the
%    new scheme, use the following input coding instead.
%\subjclass[2010]{Primary }

\date{}

%\dedicatory{}

%    Abstract is required.
\begin{abstract}
We consider hypo-elliptic diffusion and convection-diffusion on $\quot$, the quotient of the Lie group of rigid body motions SE(3) in which group elements are equivalent if they are equal up to a rotation around the reference axis. We show that we can derive expressions for the convolution kernels in terms of eigenfunctions of the PDE, by extending the approach for the SE(2) case. This goes via application of the Fourier transform of the PDE in the spatial variables, yielding a second order differential operator. We show that the eigenfunctions of this operator can be expressed as (generalized) spheroidal wave functions. The same exact formulas are derived via the Fourier transform on SE(3). We solve both the evolution itself, as well as the time-integrated process that corresponds to the resolvent operator. Furthermore, we have extended  a standard numerical procedure from SE(2) to SE(3) for the computation of the solution kernels that is directly related to the exact solutions. Finally, we provide a novel analytic approximation of the kernels that we briefly compare to the exact kernels.
\end{abstract}

\maketitle

\section{Introduction}
\subsection{Background and Motivation}
The properties of the kernels of the hypo-elliptic (convection-)diffusions on Lie groups, in particular the Euclidean motion group, have been of interest in fields such as image analysis \cite{august_curve_2001,august_sketches_2003,citti_cortical_2006,van_almsick_context_2007,momayyez_3d_2009,duits_left-invariant_2010-2}, robotics \cite{chirikjian_stochastic_2011} and harmonic analysis \cite{ter_elst_weighted_1998,chirikjian_engineering_2000,agrachev_intrinsic_2009,sharma_left-invariant_2015}. In \cite{mumford_elastica_1994} Mumford posed the problem of finding solutions of the kernels for the convection-diffusion process (direction process) on the roto-translation group SE(2). Subsequently, in \cite{thornber_characterizing_1997} analytic approximations were provided. Several numerical approaches were provided in \cite{august_curve_2001,zweck_euclidean_2004,sanguinetti_invariant_2011,zhang_numerical_2016,chirikjian_engineering_2000}. Exact solutions were derived in \cite{duits_explicit_2008}.

In \cite{agrachev_intrinsic_2009}, a Fourier-based method is presented to compute the kernel of the hypo-elliptic heat equation on unimodular Lie groups and examples are provided for various Lie groups. Three approaches to derive the exact solutions for the heat kernel for SE(2) have been proposed in SE(2) in \cite{duits_line_2009,duits_left-invariant_2010-2}, of which the first is equivalent to the approach in \cite{agrachev_intrinsic_2009} and provides a solution in terms of a Fourier series. This approach will be extended in this paper to SE(3) and the connection to the Fourier transform on SE(3) will be made. The second approach of \cite{duits_line_2009,duits_left-invariant_2010-2} provides a solution in terms of rapidly decaying functions and in the third approach this series is explicitly computed in terms of Mathieu functions.

In image analysis of 2D images, both the convection-diffusion and the diffusion PDEs can be used for crossing-preserving enhancement of elongated structures, after the image is lifted from $\mathbb{R}^2$ to the group SE(2) via an invertible orientation score \cite{franken_crossing-preserving_2009}, other types of lifting operators \cite{citti_cortical_2006,citti_sub-riemannian_2015}, or using a semi-discrete variant \cite{prandi_image_2015}. Enhancement of such structures is important in for example retinal imaging to improve the detection of blood vessels, which is challenging due to the low contrast and noise in the images.

Recently, both the diffusion and the convection-diffusion PDE have become of interest for enhancement of diffusion-weighted Magnetic Resonance Imaging (dMRI). The diffusion orientation distribution function (ODF) or the fiber orientation distribution (FOD) function can be considered to be a function on a domain in $\mathbb{R}^3 \times S^2$, indicating at each position the estimated diffusion profile or distribution of fibers on the sphere $S^2$. Such functions are usually position-wise estimated from the dMRI data. The advantage of processing with these PDEs is that they induce a better alignment between local orientations and their surrounding. In \cite{momayyez_3d_2009} the 3D extension of Mumford's direction process was used enhance dMRI data with the aid of stochastic completion fields. In \cite{reisert_about_2012} it is shown that the convection-diffusion kernel can be used to obtain asymmetric, regularized FODs. The practical advantages of the diffusion process for regularization of dMRI data, in particular better fiber tractography results and improved connectivity measures, are given in \cite{prckovska_extrapolation_2010,prckovska_contextual_2015,portegies_improving_2015,meesters_fast_2016}.

Although finite-difference implementations \cite{creusen_numerical_2011} and kernel approximations \cite{duits_left-invariant_2010} of the PDEs on $\quot$ (for which the formal definition follows later) exist, so far no exact expressions are known. The derivation of these exact solutions will be one of the main results of this paper. Other contributions of this work are summarized at the end of this introduction. We first provide more details on the mathematical setting, established in previous work \cite{duits_left-invariant_2010}.

\subsection{Left-Invariant Convection-Diffusion Operators on SE(3)}\label{se:leftinvariantoperators}
Let $\tU: SE(3) \rightarrow \mathbb{R}^+$ be a square integrable function defined on the Lie group of rigid body motions $SE(3) = \mathbb{R}^3 \rtimes SO(3)$, which is the semi-direct product of the translation group $\mathbb{R}^{3}$ and the rotation group $SO(3)$. We consider $\tU$ to be a given input, that requires regularization or enhancement of certain features. We use a tilde throughout the paper to indicate functions and operators on the group $SE(3)$ (in contrast to functions/operators on a group quotient later in the paper, that we denote without a tilde). A possible approach is to use particular evolution equations that are special cases of the following general evolution process:

\begin{equation}
\begin{cases}
\dfrac{\partial \tW(g,t)}{\partial t} = \tQ \tW(g,t), & \qquad \text{ for all } g \in \text{SE(3)}, t \geq 0,\\
\tW(g,0) = \tU(g), & \qquad \text{ for all } g \in \text{SE(3)}.
\end{cases}
\end{equation}

Here $\tQ$ is the generator of the evolution on the group, where we restrict ourselves to generators such that the evolution becomes a linear, second order convection-diffusion process. Moreover, $\tQ$ should be a left-invariant operator and is composed of the left-invariant differential operators $\underline{\mA} = (\mA_1, \dots, \mA_6)$. These vector fields are obtained via a pushforward $\mA_i|_g = (L_g)_* A_i$, where $\left\{ A_i \right\}_{i=1}^6$ is a choice of basis in the tangent space $T_e(SE(3))$ at the unity element $e = (0, I)$, where $A_1, A_2, A_3$ are spatial tangent vectors and $A_4, A_5, A_6$ are rotational tangent vectors. Here $L_g$ is the left-multiplication $L_g q = gq$.
% , obtained from the push-forward of the left-multiplication of the standard differential operators $(\partial_x, \partial_y, \partial_z, \partial_{\gamma}, \partial_{\beta}, \partial_{\alpha})$, where $\alpha$, $\beta$ and $\gamma$ are Euler angle parameters of $SO(3)$.
If $\tQ$ is a generator of a linear convection-diffusion PDE, then its general form can be written as:

\begin{equation}\label{eq:generator}
\tQ^{\bD,\ba} = \sum_{i=1}^6 -a_i \mA_i + \sum_{i,j = 1}^6 \mA_i D_{ij} \mA_j.
\end{equation}
Here $\bD = [D_{ij}] \in \mathbb{R}^{6 \times 6}$ is a symmetric positive semi-definite $6 \times 6$-matrix, and $\ba = (a_i)_{i=1}^6 \in \mathbb{R}^6$. The explicit formulas for the $\mA_i$ require two charts \cite{duits_left-invariant_2010}, but these are not of crucial importance in this work, because of the particular configuration in which the $\mA_i$'s occur in the two generators (two special cases of (\ref{eq:generator})) that we will consider. We use $\tUps_t(\tU) = \tW(\cdot,t)$ to denote the solution of the above evolution equation, so $\tUps_t : \mathbb{L}_2(SE(3)) \rightarrow \mathbb{L}_2 (SE(3))$.
The evolution can then formally be solved by convolution with the integrable solution kernel or impulse response $\tK_t:SE(3) \rightarrow \mathbb{R}^+$:

\begin{equation}
\begin{aligned}
\tUps_t(\tU)(g) = (e^{t \tQ^{\bD,\ba}} \tU)(g) = (\tK_t \ast_{SE(3)} \tU)(g) = \int \limits_{SE(3)} \tK_t(h^{-1}g)\tU(h) \; \rmd h,
\end{aligned}
\end{equation}
where we use the Haar measure on SE(3), which is the product of the Lebesgue measure on $\mathbb{R}^3$ and the Haar measure on SO(3).

\subsection{The Space of Coupled Positions and Orientations \texorpdfstring{$\quot$}{R3xS2} Embedded in SE(3)}\label{sec:introse3}
From the dMRI application, we often encounter functions defined on $\mathbb{R}^3\times S^2$ instead of the functions on SE(3) as above. These functions are called Fiber Orientation Distributions (FODs) and they estimate at each position in the space $\mathbb{R}^3$ the probability of having a fiber in a certain orientation. As shown in \cite{portegies_improving_2015}, to appropriately regularize these functions, a coupling between positions and orientations is required. This can be obtained by embedding the space $\mathbb{R}^3 \times S^2$ in SE(3). To this end, we define the space of coupled positions and orientations as a Lie group quotient in SE(3):

\begin{equation}
\quot := SE(3)/(\{\mathbf{0}\} \times SO(2)).
\end{equation}
For two elements $g = (\bx, \bR), g'=(\bx',\bR') \in SE(3)$, the group product is defined as $g g' = (\bx + \bR \bx', \bR \bR')$. Because this group product influences the group action in $\quot$, we use the semi-product notation $\rtimes$ (even though this is usually reserved for the semi-direct product of groups). The group action $\odot$ of $g \in SE(3)$ onto $(\by,\bn)  \in \mathbb{R}^{3} \times S^{2}$ is defined by
\begin{equation} \label{odot}
\begin{array}{l}
g \odot (\by,\bn)=(\bx,\bR) \odot(\by,\bn):=(\bx+ \bR \by, \bR \bn).
\end{array}
\end{equation}
Within the quotient structure $\quot$ two elements $g, g'$ (as above) are  equivalent if

\begin{equation}\label{eq:equivalencerelation}
\begin{array}{l}
g' \sim  g \Leftrightarrow h:=(g')^{-1}g \in \{\mathbf{0}\} \times SO(2)  \Leftrightarrow \bx = \bx' \textrm{ and } \exists \; \alpha \in (0,2\pi]  :\; (\bR')^{-1}\bR=\bR_{\be_{z},\alpha} \in SO(2),
\end{array}
\end{equation}
with $\be_z = (0,0,1)^T$ the reference axis\footnote{Here and throughout the paper, we use the notation $\bR_{\bn,\phi}$ for 3D rotations that perform a counter-clockwise rotation about axis $\bn \in S^2$ with angle $\phi \in  [0,2\pi]$.}. 
Using the group action, equivalence relation (\ref{eq:equivalencerelation}) amounts to:
\begin{equation}
g' \sim g \Leftrightarrow g' \odot (\mathbf{0},\be_z)= g \odot (\mathbf{0},\be_z).
\end{equation}

% Within this quotient structure two rigid body motions $g=(\bx,\bR), g'=(\bx',\bR') \in SE(3)$ are  equivalent if

% \begin{equation}\label{eq:equivalencerelation}
% \begin{array}{l}
% g' \sim  g \Leftrightarrow h:=(g')^{-1}g \in \{\mathbf{0}\} \times SO(2)  \Leftrightarrow \bx = \bx' \textrm{ and } \exists \; \alpha \in \[0,2\pi\]  :\; (\bR')^{-1}\bR=\bR_{\be_{z},\alpha} \in SO(2),
% \end{array}
% \end{equation}
% with $\be_z = (0,0,1)^T$ the reference axis. Here and throughout the paper, we use the notation $\bR_{\bn,\phi}$ for 3D rotations that perform a counter-clockwise rotation about axis $\bn \in S^2$ with angle $\phi \in [0,2\pi]$. Furthermore, the action $\odot$ of $g=(\bx,\bR) \in SE(3)$ onto $(\by,\bn)  \in \mathbb{R}^{3} \times S^{2}$ is defined by
% \begin{equation} \label{odot}
% \begin{array}{l}
% g \odot (\by,\bn)=(\bx,\bR) \odot(\by,\bn):=(\bx+ \bR \by, \bR \bn).
% \end{array}
% \end{equation}
% As a result, equivalence relation (\ref{eq:equivalencerelation}) can be simplified to:
% \begin{equation}
% g' \sim g \Leftrightarrow g' \odot (\mathbf{0},\be_z)= g \odot (\mathbf{0},\be_z).
% \end{equation}

Thereby, an arbitrary element in $\mathbb{R}^{3} \rtimes S^{2}$ can be considered as the equivalence class of all rigid body motions that map reference position and orientation $(\mathbf{0},\be_z)$ onto $(\bx,\bn)$  (with arbitrary $\bx \in \mathbb{R}^3$ and $\bn \in S^2$). Similar to the common identification of $S^{2}\equiv SO(3)/SO(2)$, we denote elements of the Lie group quotient $\mathbb{R}^{3}\rtimes S^{2}$ by $(\bx,\bn)$.
% Using the common identification of $S^{2}\equiv SO(3)/SO(2)$, we denote elements of the Lie group quotient $\mathbb{R}^{3}\rtimes S^{2}$ by $(\by,\bn)$. The action $\odot$ on quotient elements can be defined as

% \begin{equation}
% (\by,\bn) \odot (\by',\bn') := (\by + \bR_{\bn} \by', \bR_{\bn} \bn').
% \end{equation}

Functions $U: \quot \rightarrow \mathbb{R}^+$ are identified with functions $\tU: SE(3) \rightarrow \mathbb{R}^+$ by the relation

\begin{equation}\label{eq:quotientgroupconstraint}
U(\by,\bn) = \tU(\by, \bR_\bn),
\end{equation}
where $\bR_{\bn}$ is \emph{any} rotation matrix such that $\bR_{\bn} \be_z = \bn$ (not to be confused with the notation $\bR_{\bn,\phi}$, that specifies the rotation by its axis $\bn$ and angle of rotation $\phi$). This means that the functions $\tU$ that we consider have the invariance $\tU(\by,\bR) = \tU(\by,\bR \bR_{\be_z,\alpha})$, for all $\alpha \in (0, 2\pi]$. 

\subsection{Legal Diffusion and Convection-Diffusion Operators}\label{se:legaldiffusion}
For functions defined on the quotient $\quot$ the same machinery as in Section \ref{se:leftinvariantoperators} can be used, although some restrictions apply. In \cite{duits_morphological_2012} it is shown that operators on $\tU$ are legal (i.e., they correspond to well-defined operators on $U$ that commute with rotations/translations) if and only if the following holds:

\begin{alignat}{2}\label{def:legaloperator}
\tUps \circ \mL_g &= \mL_g \circ \tUps, & \qquad g \in SE(3)&, \\ \nonumber
\tUps &= \mathcal{R}_h \circ \tUps, & \qquad h = (\mathbf{0},\bR_{\be_z,\alpha}) \in \{\mathbf{0}\}\times SO(2)&,
\end{alignat}
with left-regular action $(\mL_g \tU) (g') = \tU(g^{-1} g')$ and right-regular action $(\mathcal{R}_g \tU)(g') = \tU(g' g)$. For the diffusion case, an equivalent statement is that the diffusion matrix $\ul{D}$ is invariant under conjugation with $\mathbf{Z}_{\alpha}= \ul{R}_{\ul{e}_z,\alpha} \oplus \ul{R}_{\ul{e}_z,\alpha} \in SO(6)$, i.e. $\mathbf{Z}_\alpha^{-1} \bD \mathbf{Z}_\alpha = \bD$. Among the few legal, left-invariant diffusion and convection-diffusion generators on $\quot$, recall (\ref{eq:generator}), are the pure hypo-elliptic diffusion case and the hypo-elliptic convection-diffusion case, for details see \cite{duits_left-invariant_2010,duits_morphological_2012}. The first case corresponds to the forward Kolmogorov equation for hypo-elliptic Brownian motion, the second case corresponds to the forward Kolmogorov equation for the direction process (in 3D), as illustrated in Fig.~\ref{fig:kernelvis}. We denote the evolution generators of these two cases on the group and on the quotient with $\tQ_i$ and $Q_i$, $i = 1,2$, respectively. They are defined as:

\begin{equation}\label{eq:CEoperatortilde}
\tQ_1 := D_{33} \mA_3^2 + (D_{44} \mA_4^2 + D_{55} \mA_5^2),
\end{equation}
with $D_{33}>0, D_{44} = D_{55} > 0$. The generator $Q_1$ acts on sufficiently smooth functions $\tU$ and can be identified with a generator $Q_1$ acting on functions $U$:

\begin{equation}\label{eq:CEoperator}
\boxed{Q_1 := D_{33}(\bn \cdot \nabla_{\mathbb{R}^3})^2 + D_{44} \Delta_{S^2}.}
\end{equation}
Here $\nabla_{\mathbb{R}^3}$ denotes the gradient on $\mathbb{R}^3$, and $\Delta_{S^2}$ is the Laplace-Beltrami operator on the sphere $S^2$. Also the following (hypo-elliptic) convection-diffusion generator can be identified with a legal generator on $\quot$:

\begin{equation}\label{eq:CCoperatortilde}
\tQ_2 := - \mA_3 + (D_{44} \mA_4^2 + D_{55} \mA_5^2),
\end{equation}
again with $D_{44} = D_{55}>0$. The corresponding generator $Q_2$ acting on sufficiently smooth functions $U$ is defined as:

\begin{equation}\label{eq:CCoperator}
\boxed{Q_2 := -(\bn \cdot \nabla_{\mathbb{R}^3}) + D_{44} \Delta_{S^2}.}
\end{equation}

\begin{remark}\label{rem:relationgroupquotientaction}
The way $\tQ_i$ and $Q_i$ act on functions $\tW$ and $W$, respectively, is related by the following identities:

\begin{equation}
\begin{aligned}
(\bn \cdot \nabla_{\mathbb{R}^3} W) (\by,\bn) &= (\mA_3 \tW)(\by,\bR_\bn),\\
(\Delta_{S^2} W)(\by,\bn) &= (\mA_4^2 +\mA_5^2)\tW(\by,\bR_\bn).
\end{aligned}
\end{equation}
As a result, the following relation holds:

\begin{equation}
(\tQ_i \tW)(\by,\bR_{\bn}) = (Q_i W)(\by,\bn), \qquad i = 1,2,
\end{equation}
regardless of the choice of rotation for $\bR_{\bn}$ (mapping $\be_z$ onto $\bn$).
\end{remark}

By the previous remark, we are no longer concerned with the left-invariant vector fields $\mA_i$, and we focus on the following two PDE systems using (\ref{eq:CEoperator}) and (\ref{eq:CCoperator}):

\begin{equation}\label{eq:differentialQ}
\boxed{\left\{\begin{aligned}
\frac{\partial}{\partial t} W(\by,\bn,t) &= (Q_i W)(\by,\bn,t),  \\
W(\by,\bn,0) &= U(\by,\bn),
\end{aligned} \right. }
\end{equation}
with $(\by,\bn) \in \quot$, $t \geq 0$, $i = 1,2$. In particular, we want to find an exact solution and approximations for the impulse response or convolution kernel of these PDEs, i.e., the solution with initial condition $U(\by,\bn) = \delta_{(\mathbf{0},\be_z)}(\by,\bn)$. We denote the linear and bounded evolution operator that maps $U(\by,\bn)$ on $W(\by,\bn,t)$ by $\Upsilon_t : \mathbb{L}_2(\quot) \rightarrow \mathbb{L}_2(\quot)$:
\begin{equation}
W(\by,\bn,t)=(\Upsilon_t U)(\by,\bn), \qquad \text{for all } (\by,\bn) \in \R^{3} \rtimes S^{2}.
\end{equation}

\subsection{Relation with Stochastic Processes, Time-Integrated Processes and Resolvent kernels}\label{se:stopro}
Any evolution equation corresponding to a generator $Q^{\bD,\ba}$ as above, is in fact a Fokker-Planck equation for the evolution of a probability density function. The (convection)-diffusion processes generated by $Q_1$ and $Q_2$ were called contour enhancement and contour completion in \cite{duits_left-invariant_2010}, respectively, because of the stochastic processes they relate to, see Fig. \ref{fig:kernelvis}. The random walks in the figure are simulated according to the formal definition of the stochastic processes in \ref{app:stochastics}. The convolution kernels of the PDEs can also be obtained with a Monte Carlo simulation by accumulating infinitely many random walks starting at $(\mathbf{0},\be_z)$ that obey the underlying stochastic process and have traveling time $t$, as illustrated in Fig. \ref{fig:kernelvis}. The figure also displays the SE(2) equivalents of the processes, because they are easier to visualize and interpret. The SE(2) process for contour completion is better known as Mumford's direction process \cite{mumford_elastica_1994} and finds its application in computer vision. 

\begin{remark}[Glyph field visualization]\label{remark:glyphfieldvis}
In Fig. \ref{fig:kernelvis} and in later figures in the paper, we make use of so-called glyph field visualizations of probability distributions $U$ on $\quot$. A glyph at a grid point $\by \in c \, \mathbb{Z}^3$, $c>0$, is given by the surface $\{\by + \nu U(\by,\bn) \bn \; | \; \bn \in S^2 \}$, for a suitable choice of $\nu \in \mathbb{R}$. Usually we choose $\nu > 0$ depending on the maximum of $U$, such that the glyphs do not intersect.
\end{remark}

% \subsection{Time-Integrated Processes and Resolvent Kernels}\label{subsec:timeintegrated}
When we are interested in the position of a random walker regardless of the traveling time, we can condition on the traveling time and integrate. If we assume to have exponentially distributed traveling time $t \sim Exp(\alpha)$, with mean $1/\alpha$, $\alpha>0$, then the probability density function of a random particle given its initial distribution $U$ can be written as:

% then the probability density function of a random particle is given by

% \begin{equation}
% R_{\alpha}(\by,\bn) = \int_0^{\infty} p_t(\by,\bn) \; \alpha \; e^{-\alpha t} \; \rmd t.
% \end{equation}
\begin{equation}
P_{\alpha}(\by,\bn) = \alpha \int_0^{\infty} e^{-t \alpha} (e^{t(Q^{\bD,\ba})}U)(\by,\bn) \; \rmd t
= - \alpha (Q^{\bD, \ba} - \alpha I)^{-1} U (\by,\bn).
\end{equation}
This shows that the time-integrated process relates to the resolvent operator of $\QDa$, in particular for $Q_1$ and $Q_2$ in Eq. (\ref{eq:CEoperator}) and (\ref{eq:CCoperator}), respectively. Therefore, apart from the time evolutions, we are also concerned with the corresponding resolvent equation:

\begin{equation}\label{eq:resolventQi}
\boxed{\left( (Q_i - \alpha I \right) P^i_\alpha) (\by,\bn) = -\alpha U(\by,\bn)},
\end{equation}
with $(\by,\bn) \in \quot$, $\alpha > 0$, $i = 1,2$.

Spectral decomposition of generator $Q_i$ (with a purely discrete spectrum) therefore also yields the spectral decomposition of both the (compact) operator $\Upsilon_t$ and the (compact) resolvent operator \mbox{$(Q_i - \alpha I)^{-1}$}. It follows that Monte Carlo simulations of random walks with exponentially distributed traveling times can be used to approximate resolvent kernels. Furthermore, it can be shown that the resolvent operator occurs in Tikhonov regularization of the input \cite{duits_left-invariant_2010}. For the convection-diffusion case, the time-integrated process is practically even more useful than the time-dependent case.

\begin{figure}[t!]
   \centering
   \includegraphics[width=0.9\textwidth]{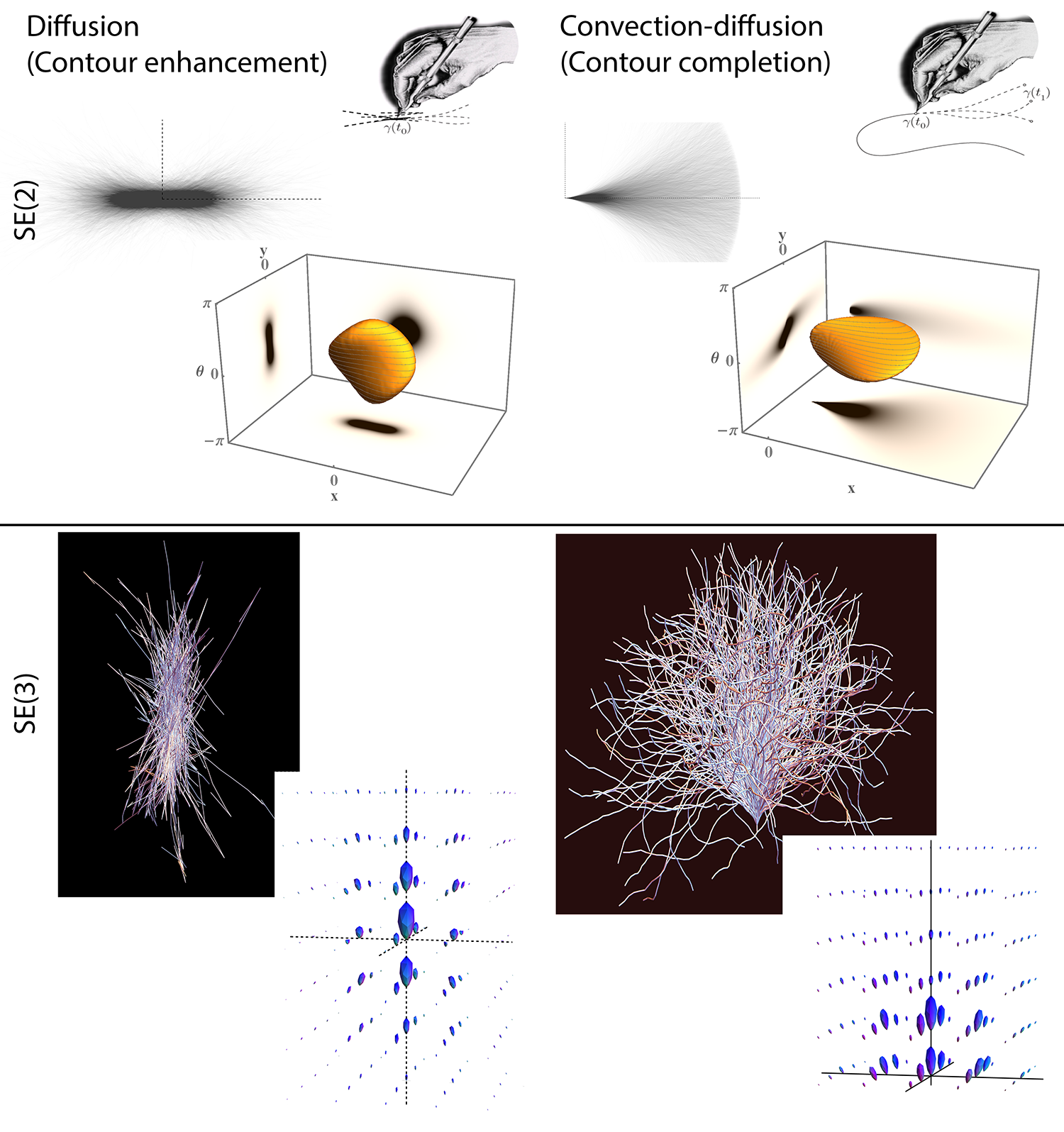}
 \caption{Various visualizations of the diffusion process (for contour enhancement, left, generated by $Q_1$) and the convection-diffusion process (for contour completion, right, generated by $Q_2$). For the SE(2) case, both the projections of random paths on $\mathbb{R}^2$ are shown, as well as an isocontour in SE(2) of the limiting distribution. For the SE(3) case, we show spatial $\R^3$-projections of random paths in SE(3), as given in \ref{app:stochastics}, and visualize the limiting distribution as a glyph field, see Remark \ref{remark:glyphfieldvis}.}\label{fig:kernelvis}
 \end{figure}

\subsection{Contributions of the Paper}
In this paper, we provide exact expressions for the kernels corresponding to the generators $Q_1$ and $Q_2$ for diffusion and convection-diffusion, in terms of eigenfunctions of the operator, after applying a Fourier transform in the spatial coordinates. As such we solve the 3D version of Mumford's direction process \cite{mumford_elastica_1994}, and the 3D version of the (2D) hypo-elliptic Brownian motion kernel, considered for orientation processing in image analysis in \cite{citti_cortical_2006,agrachev_intrinsic_2009,duits_left-invariant_2010-2}. The main results can be found in Theorems \ref{th:maintheoremCE}, \ref{th:resolventseries} and \ref{th:maintheoremCC} and in Corr. \ref{cor:ellipticcase}.

A similar approach was used in \cite{duits_explicit_2008} for the SE(2)-case, but for the SE(3)-case exact solutions are not known, to the best of our knowledge. For SE(2), solutions were expressed in \cite{duits_explicit_2008} in terms of Mathieu functions. In SE(3), we encounter (generalized) spheroidal wave functions, that can be written as a series of associated Legendre functions. We use the eigenfunctions to give expressions for both the time-dependent and the time-integrated process, associated with a resolvent equation. Using the results for the exact solutions for the time-integrated case, we derive in Thm \ref{thm:gammadist} the minimal amount of repeated convolutions of the resolvent kernels required to remove the singularities in the origin. The proof of this result is given in \ref{app:proof}.

We also provide two numerical methods to approximate the kernels. First, we extend the numerical algorithm by August \cite{august_curve_2001} from SE(2) to SE(3). This algorithm has the advantage that it is closely related to the exact solutions and can therefore be shown to yield convergence to the exact solutions. As a second approach we provide a novel approximation of the kernels, that provide correct symmetries in contrast to existing approximations provided in \cite{duits_left-invariant_2010}. This rough and local approximation is convenient as it allows for straightforward and fast implementations, but no convergence results can be obtained. A numerical approximation based on finite differences was presented in \cite{creusen_numerical_2011}, and is therefore not considered here. A full comparison between exact solutions, stochastic process limits and kernel approximations for these processes on SE(2) has been done in \cite{zhang_numerical_2016}. In the present paper, the focus will be on the derivation of the exact solutions and its connection to a numerical approximation method that generalizes the results \cite{august_curve_2001,duits_explicit_2008} from $SE(2)$ to $SE(3)$.

Finally, in order to make the connection to earlier work on exact solutions of heat kernels in the SE(2)-case \cite{agrachev_intrinsic_2009,duits_explicit_2008} via harmonic analysis and the Fourier transform on SE(2), we also derive equivalent representations of the kernels using the SE(3) Fourier transform. The details of this equivalent representation (and alternative roadmap to the solutions) is presented in \ref{ap:Fourierse3}, where we follow an  algebraic Fourier theoretic approach. We stress that the body of this article relies on classical, geometrical and functional analysis. However, the approach in \ref{ap:Fourierse3} provides the reader further insight in specific choices in the geometric analysis in the main body of the article.
% Finally, we demonstrate how such kernels can be applied on dMRI data. [some words on the experiment, add later]

\subsection{Outline of the Paper}
The paper is structured as follows. In Sections \ref{se:exactCE} and \ref{se:exactCC} we derive exact expressions for the convolution kernels of the differential equations in (\ref{eq:differentialQ}), with the main results in Theorem \ref{th:maintheoremCE} and Theorem \ref{th:maintheoremCC}. We emphasize that the roadmap and computations in these sections are very similar, only in the convergence proofs of the encountered series of eigenfunctions we need different theory for each case. In Section \ref{se:Fourierr3s2} we present a matrix representation of the evolution in a Fourier basis and provide an algorithm to numerically compute a truncation of the exact series solution. Section \ref{se:approximationkernel} shows the Gaussian kernel approximations. We summarize our findings and conclude in Section \ref{se:conclusion}. The derivation of equivalent solutions for the main PDEs using the SE(3) Fourier transform can be found in \ref{ap:Fourierse3}.

\section{Derivation of the Exact Solutions for Hypo-Elliptic Diffusion on \texorpdfstring{$\quot$}{R3xS2}}\label{se:exactCE}
In this section we derive the exact solutions for the hypo-elliptic diffusion case. We first set out the formal procedure for finding these solutions, before we present the details for this particular case and the specific eigenfunctions that we encounter. The evolution process for hypo-elliptic diffusion on $\quot$, i.e., with generator $Q_1$ as in (\ref{eq:CEoperator}), is written as follows:

\begin{equation}\label{eq:CEdiff}
\begin{cases}
\dfrac{\partial}{\partial t} W(\by,\bn,t) = Q_1 W(\by,\bn,t) = (D_{33}(\bn \cdot \nabla_{\mathbb{R}^3})^2 + D_{44} \Delta_{S^2}) W(\by,\bn,t), & \quad \by \in \mathbb{R}^3, \bn \in S^2, t \geq 0,\\
W(\by,\bn,0) = U(\by,\bn), &\quad \by \in \mathbb{R}^3, \bn \in S^2.
\end{cases}
\end{equation}
As in (\ref{eq:CEoperator}),(\ref{eq:CCoperator}), we use $\nabla_{\mathbb{R}^3}$ to indicate the gradient with respect to the spatial variables, and $\Delta_{S^2}$ to denote the Laplace-Beltrami operator on the sphere. Parameters $D_{33}, D_{44} >0$ influence the amount of spatial and angular regularization, respectively. We use both a subscript and superscript $1$ throughout this section for operators that arise from this evolution, to distinguish these from the operators corresponding to the convection-diffusion that we encounter in Section \ref{se:exactCC}.

The domain of the generator $Q_1$ equals

\begin{equation}
\mathcal{D}(Q_1) = \mathbb{H}_2(\mathbb{R}^3) \otimes \mathbb{H}_2(S^2),
\end{equation}
where we use $\mathbb{H}_2$ to denote a Sobolev space, although in $\mathcal{D}(Q_1)$ both $\mathbb{H}_2(\mathbb{R}^3)$ and $\mathbb{H}_2(S^2)$ are equipped with the usual $\mathbb{L}_2$-norm. For details on Sobolev spaces see e.g. \cite{rudin_functional_2006,wloka_partial_1987}. By linearity and left-invariance of the differential equation, the solution can be found by $\mathbb{R}^3 \rtimes S^2$-convolution with the corresponding integrable kernel $K_t^1: \mathbb{R}^3 \rtimes S^2 \rightarrow \mathbb{R}^+$:

\begin{equation}\label{eq:convolution}
\begin{aligned}
W(\by,\bn,t) &= (K^1_t \ast_{\mathbb{R}^3 \rtimes S^2} U)(\by,\bn)\\
&= \int \limits_{S^2} \int \limits_{\mathbb{R}^3} K^1_t (\bR_{\bn'}^T(\by - \by'), \bR_{\bn'}^T \bn) \; U(\by', \bn') {\rm d} \by' \; {\rm d} \sigma(\bn').
\end{aligned}
\end{equation}
The specific choice for the rotation matrix $\bR_{\bn'}$ does not matter, since the left-invariance of the PDE implies that $K_t^1(\by,\bn) = K_t^1(\bR_{\be_z,\alpha}\by,\bR_{\be_z,\alpha}\bn)$ for all $\alpha \in [0,2 \pi]$, see \cite{duits_left-invariant_2010}. This $\alpha$-invariance is important because of the equivalence relation described in \eqref{eq:equivalencerelation}. Our approach for finding the exact kernel $K_t^1$ is inspired by the approach for the SE(2) case in \cite{duits_explicit_2008}. We first apply a Fourier transform with respect to the spatial variables:

\begin{equation}
\hat{W}(\bomega,\bn,t) := \mFR(W)(\bomega,\bn,t) = \int \limits_{\mathbb{R}^3} W(\by,\bn,t) e^{-i \bomega \cdot \by} \rmd \by.
\end{equation}
The hat is used to indicate that a function has been Fourier transformed. The PDE (\ref{eq:CEdiff}) in terms of $\hW$ then becomes:

\begin{equation}\label{eq:CEdiffFourier}
\left\{
\begin{aligned}
& \frac{\partial}{\partial t} \hW(\bomega,\bn,t)= (D_{44} \Delta_{S^2} - D_{33}(\bomega \cdot \bn)^2) \hW(\bomega,\bn,t), \\
& \hW(\bomega,\bn,0) = \mFR(U)(\bomega,\bn) =: \hat{U}(\bomega,\bn).
\end{aligned}
\right.
\end{equation}
We fix $\bomega \in \mathbb{R}^3$ and we define the operator $\mB^1_{\bomega}: \mathbb{H}_2(S^2) \rightarrow \mathbb{L}_2(S^2)$ as follows:

\begin{equation}\label{eq:operatorBw}
\mB^1_{\bomega} := D_{44} \Delta_{S^2} - D_{33}(\bomega \cdot \bn)^2.
\end{equation}
We use the subscript $\bomega$ to explicitly indicate that the operator depends on the frequency vector $\bomega$, as will the eigenfunctions of this operator. When we write

\begin{equation}
(\mB \hW)(\bomega,\bn) := (\mB^1_{\bomega} \hW(\bomega, \cdot)(\bn),
\end{equation}
then the correspondence between the operator $\mB$ and the generator $Q_1$ can be written as:

\begin{equation}
\mB = (\mF_{\mathbb{R}^3} \otimes \mathbbm{1}_{\mathbb{L}_2(S^2)}) \circ Q_1 \circ (\mF_{\mathbb{R}^3}^{-1} \otimes \mathbbm{1}_{\mathbb{H}_2(S^2)}).
\end{equation}
The heat kernel $\hK_t^1 \in \mathbb{L}_2(\quot) \cap C(\quot)$ of the PDE in the Fourier domain should then satisfy:

\begin{equation}\label{eq:CEPDEkernel}
\left\{\begin{aligned}
& \frac{\partial}{\partial t} \hK_t^1(\bomega,\bn) = \mB^1_{\bomega} \hK_t^1(\bomega,\bn), \\
&\hK_0^1(\bomega,\bn) = \delta_{\be_z}(\bn),
\end{aligned}\right.
\end{equation}
with $\delta_{\be_z}$ the $\delta$-distribution on $S^2$ at $\be_z$. We show later that the $\mathbb{L}_2(S^2)$-normalized eigenfunctions of the operator $\mB^1_{\bomega}$ form an orthonormal basis for $\mathbb{L}_2(S^2)$ and that, similar to the enumeration of spherical harmonics, these functions are indexed with integers $l$ and $m$, $|m| \leq l$. For the eigenfunctions, that we denote with $\Philmom$, we have:

\begin{equation}\label{eq:eigenfunctions}
\mB^1_{\bomega} \Philmom = \lambdalmom \Philmom, \qquad \text{ with } \lambdalmom \leq 0.
\end{equation}
The kernel $\hat{K}_t^1$ can then be written in terms of these eigenfunctions as

\begin{equation}\label{eq:Khatseries}
\hat{K}_t^{1}(\bomega,\bn) = \sumlm \overline{\Philmom(\be_z)} \Philmom(\bn)  \; e^{\lambdalmom t}.
\end{equation}
The solution of the differential equation (\ref{eq:CEdiffFourier}) in the Fourier domain is given by:

\begin{equation}
\hat{W}(\bomega,\bn,t) = \sumlm (\hat{U}(\bomega,\cdot), \Philmom)_{\mathbb{L}_2(S^2)} \Philmom(\bn) \; e^{\lambdalmom t},
\end{equation}
where we rely on the following inner product convention on $\mathbb{L}_2(S^2)$:

\begin{equation}\label{eq:definnerproduct}
(f,g)_{\mathbb{L}_2(S^2)} = \int_{S^2} f(\bn) \overline{g(\bn)} \; \rmd \sigma(\bn).
\end{equation}
Thereby, the solution of Eq. (\ref{eq:CEdiff}) is given by:
\begin{equation}
W(\by,\bn,t) = \left[\mFR^{-1} \hat{W} (\cdot, \bn, t) \right] (\by).
\end{equation}
This expression should coincide with the convolution expression in Eq. (\ref{eq:convolution}). The following lemma gives us a useful identity for the eigenfunctions $\Philmom$, that allows us to connect the series expression for the solution of (\ref{eq:CEdiff}) with the convolution form in (\ref{eq:convolution}).

%and leads to Corollary \ref{cor:seriesandconvolution}, that connects the series expression for the solution of (\ref{eq:CEdiff}) with the convolution form in (\ref{eq:convolution}).

% \begin{proof}
% \begin{equation}
% \begin{aligned}
% \left[(D_{44} \Delta_{S^2} - D_{33} (\bR \bomega \cdot \bn)^2) \Phi_{l,m}^{\bomega} (\bR^T \cdot) \right](\bn) &= \left[(D_{44} \Delta_{S^2} - D_{33} (\bomega \cdot \bn')^2) \Phi_{l,m}^{\bomega} (\cdot) \right](\bn') \\
% \lambdalmom \Philmom(\bR^T \bn) = \lambdalmom \Philmom(\bn')
% \end{aligned}
% \end{equation}
% \end{proof}

\begin{lemma}\label{lemma:eigenfunctions}
For all $l \in \mathbb{N}_0,m \in \mathbb{Z}$, $|m| \leq l$, let $\Philmom(\bn)$ be an eigenfunction of $\mB^1_{\bomega}$, with eigenvalue $\lambdalmom$, and let $\bR \in SO(3)$. Then $\Philmom(\bR^T \cdot)$ is an eigenfunction of $\mB^1_{\bR \bomega}$ with eigenvalue $\lambda^{l,m}_{\bR \bomega} = \lambdalmom$.
\end{lemma}
\begin{proof}
We can write

\begin{equation}
\begin{aligned}
(\mB^1_{\bR \bomega} \Philmom(\bR^T \cdot))(\bn) &= (\Delta_{S^2} \Philmom(\bR^T \cdot))(\bn) - (\bR \bomega,\bn)^2 (\Philmom)(\bR^T \bn)\\
&= (\Delta_{S^2} \Philmom)(\bR^T \bn) - (\bomega, \bR^T \bn)^2 \Philmom(\bR^T \bn) \\
&= (\mB^1_{\bomega} \Philmom(\cdot))(\bR^T \bn)\\
&= \lambdalmom \Philmom(\bR^T \bn),
\end{aligned}
\end{equation}
from which the result follows.
\end{proof}
From the lemma we conclude that the eigenvalues $\lambdalmom$ only depend on the norm $r = ||\bomega||$ of the frequency $\bomega$, so from now on we write $\lambdalmr$. Moreover, we have the following relation between eigenfunctions:

\begin{equation}\label{eq:eigenfunctionrelation}
\Phi_{l,m}^{\bR \bomega}(\bn) = \Phi_{l,m}^{\bomega}(\bR^T \bn).
\end{equation}

We can combine the previous considerations and Lemma \ref{lemma:eigenfunctions} to give two equivalent series expressions for the kernel $\hat{K}_1$ satisfying Eq. (\ref{eq:CEPDEkernel}):

\begin{equation}\label{eq:twoidentities}
\hK^1_t(\bomega,\bn) = \sumlm \overline{\Philmom(\be_z)} \, \Philmom(\bn) \, e^{\lambdalmr t} = \sumlm \overline{\Phi_{l,m}^{\bR_{\bn'}\bomega}(\bn')}\,\Phi_{l,m}^{\bR_{\bn'}\bomega}(\bR_{\bn'}\bn)\,e^{\lambdalmr t}.
\end{equation}
Finally, the solution of \eqref{eq:CEdiff} can then be written as 
% \todo[inline]{The $d\sigma(\bn)$ should not be there right?}
\begin{equation}\label{eq:Whatseriesint}
\begin{aligned}
W(\by,\bn,t) &= (K_t^1 \ast_{\quot} U)(\by,\bn) = ((\mFR^{-1} \hK_t^1(\cdot,\cdot)) \ast_{\quot} U)(\by,\bn)\\
&=  \int \limits_{\mathbb{R}^3} \sumlm (\hU(\bomega,\cdot),\Philmom)_{\mathbb{L}_2(S^2)} \Philmom(\bn) \; e^{\lambdalmr t} e^{i \by \cdot \bomega} \rmd \bomega. 
\end{aligned}
\end{equation}

Now that we have formal expressions for the solution of the hypo-elliptic diffusion PDE, we can focus on finding the eigenfunctions $\Philmom$.

\subsection{Frequency-Dependent Choice of Variables}
So far we have not specified a choice of spherical variables for the functions to which the operator $\mB^1_{\bomega}$ is applied. This is needed in order to derive expressions for the eigenfunctions $\Philmom$. As in the case for spherical harmonics, we want to use separation of variables for each fixed spatial frequency $\bomega$. To be able to do this, we choose to parameterize the orientation of $\bn$ using angles dependent on $\bomega$. This choice should be such that the variables can be separated in both the Laplace-Beltrami operator $\Delta_{S^2}$ and in the multiplication operator $(\bomega \cdot \bn)^2$.

The Laplace-Beltrami operator on a Riemannian manifold with metric tensor $\mG$ is defined as:

\begin{equation}
\Delta_{LB} = \frac{1}{\sqrt{|\mG|}}\sum_{i,j=1}^n \frac{\partial}{\partial x_i} \left( \sqrt{|\mG|} \mG^{ij} \frac{\partial}{\partial x_j} \right).
\end{equation}
The $\mG^{ij}$ are matrix elements of the inverse $\mG^{-1}$ of $\mG$. The Laplace-Beltrami operator on the sphere $\Delta_{S^2}$ is, when we take standard spherical coordinates, given by:

\begin{equation}
\Delta_{S^2} = \frac{1}{\sin{\phi}} \frac{\partial}{\partial \phi} \left(\sin \phi \frac{\partial}{\partial \phi} \right) + \frac{1}{\sin^{2} \phi} \frac{\partial ^2}{\partial \theta^2},
\end{equation}
with $x = \sin \phi \cos \theta$, $y = \sin \phi \sin \theta$ and $z = \cos \phi$. However, in these coordinates the term $(\bomega \cdot \bn)^2$ is not separable.

To this end, we choose spherical coordinates with respect to the (normalized) frequency $r^{-1} \bomega$, with $r = ||\bomega||$, and a second axis perpendicular to $\bomega$. The specific choice for the latter axis is not important, since in our PDE only the angle between $\bomega$ and $\bn$ plays a role. For convenience we take as second axis $\frac{\bomega \times \be_z}{||\bomega \times \be_z||}$, and we let $\beta, \gamma$ denote the angles of rotation about axes $\frac{\bomega \times \be_z}{||\bomega \times \be_z||}$ and $r^{-1} \bomega$, respectively. For $r^{-1} \bomega = \be_z$, $\beta$ and $\gamma$ are just the standard spherical coordinates. Every orientation $\bn \in S^{2}$ can now be written in the form

\begin{equation}\label{def:coordinatechoice}
\bn = \bn^{\bomega}(\beta, \gamma) =  \bR_{r^{-1} \bomega,\gamma} \bR_{\frac{\bomega \times \be_z}{||\bomega \times \be_z||},\beta} (r^{-1} \bomega), \text{ with } r = ||\bomega||,
\end{equation}
see Fig. \ref{fig:figureParametrization}.

\begin{figure}[t!]
   \centering
   \includegraphics[width=0.4\textwidth]{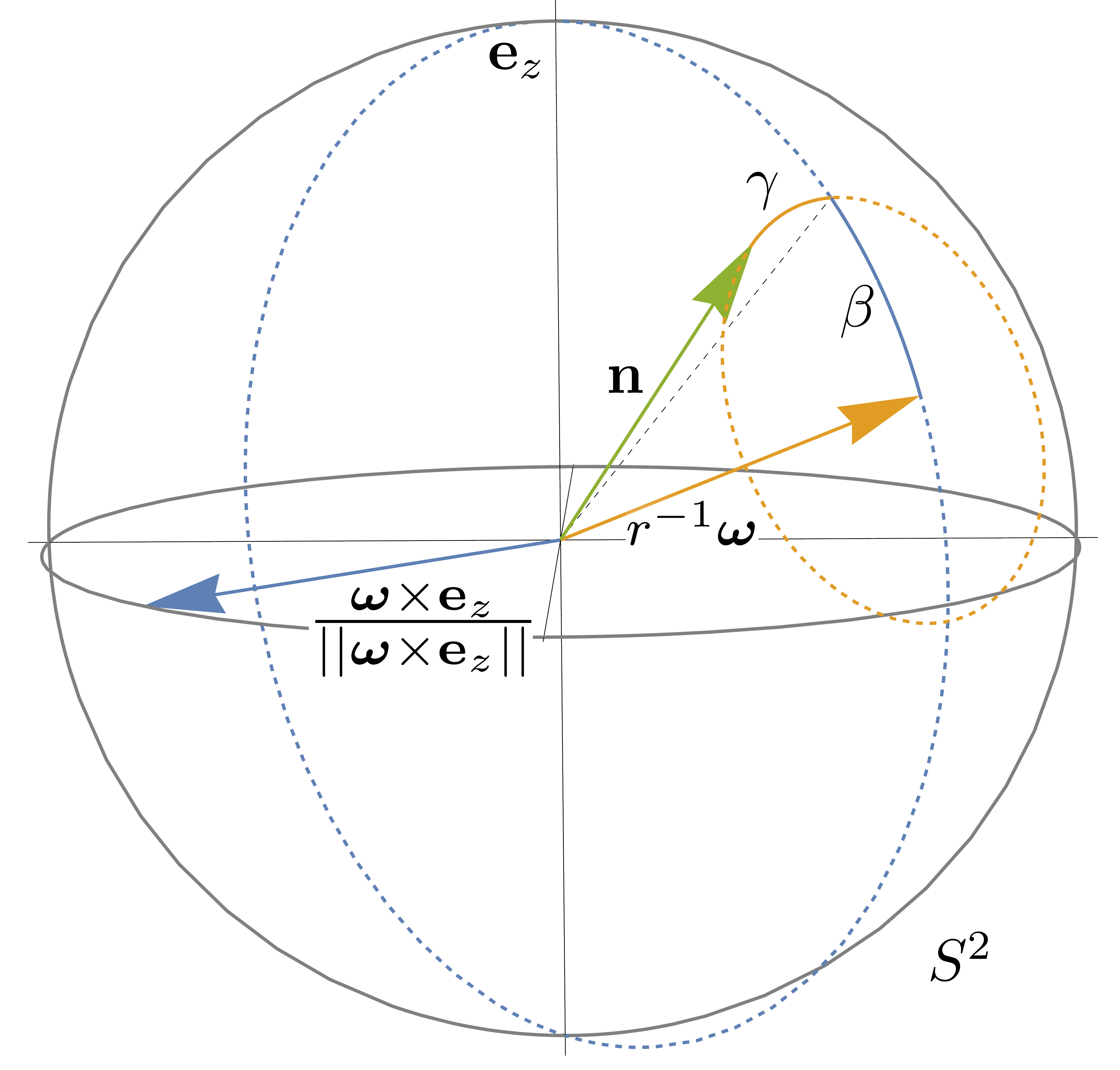}
 \caption{For $\bomega \neq \be_z$, we parameterize every orientation $\bn$ (green) by rotations around $r^{-1} \bomega$ (orange) and $\frac{\bomega \times \be_z}{||\bomega \times \be_z||}$ (blue). In other words, $\bn^{\bomega} (\beta,\gamma) = \bR_{r^{-1} \bomega,\gamma} \bR_{\frac{\bomega \times \be_z}{||\bomega \times \be_z||}, \beta} (r^{-1} \bomega )$.}\label{fig:figureParametrization}
 \end{figure}

\begin{lemma}\label{le:laplacebeltrami}
The Laplace-Beltrami operator on $S^2$ with the choice of variables as in Eq. (\ref{def:coordinatechoice}) is given by:

\begin{equation}
\Delta_{S^2} = \frac{1}{\sin{\beta}} \frac{\partial}{\partial \beta} \left(\sin \beta \frac{\partial}{\partial \beta} \right) + \frac{1}{\sin^{2} \beta} \frac{\partial ^2}{\partial \gamma^2}.
\end{equation}
The full operator $\mB^1_{\bomega}$ as defined in (\ref{eq:operatorBw}) with this choice of coordinates takes the separable form
\begin{equation}\label{eq:operatorBwbetagamma}
\mB^1_{\bomega} = \frac{D_{44}}{\sin{\beta}} \frac{\partial}{\partial \beta} \left(\sin \beta \frac{\partial}{\partial \beta} \right) + \frac{D_{44}}{\sin^{2} \beta} \frac{\partial ^2}{\partial \gamma^2} - D_{33} (r \cos \beta)^2.
\end{equation}
\begin{proof}
The result follows from direct computation of the metric tensor $\mG$ w.r.t. the coordinates in in \eqref{def:coordinatechoice}.
\end{proof}
\end{lemma}

% \begin{remark}
% The idea for this coordinate transform to bring $\mB^1_\bomega$ in separable form can be motivated from Fourier theory on SE(3), see \ref{ap:Fourierse3}.
% \end{remark}

\subsection{Separation of Variables}
Thanks to the choice of coordinates \eqref{def:coordinatechoice} and Lemma \ref{le:laplacebeltrami}, we can apply the method of separation of variables to solve the diffusion equation (\ref{eq:CEdiffFourier}). We first look for solutions of the PDE without initial conditions, and we take the solutions of the separable form $T(t)\Phi(\beta,\gamma)$:

\begin{equation}
\begin{aligned}
\frac{\partial}{\partial t} (T(t)\Phi(\beta,\gamma)) &= \mB^1_{\bomega}( T(t) \Phi)(\beta,\gamma), \\
\Phi(\beta,\gamma) \frac{\partial}{\partial t} T(t) &= T(t) (\mB^1_{\bomega} \Phi)(\beta,\gamma).
\end{aligned}
\end{equation}
It follows that we get

\begin{equation}
\frac{1}{T(t)}\frac{\partial T(t)}{\partial t} = \frac{(\mB^1_{\bomega}\Phi)(\beta,\gamma)}{\Phi(\beta,\gamma)} =: \lambda_r,
\end{equation}
where $\lambda_r$ is the separation constant. We get that $T(t) = K e^{\lambda_r t}$, $K$ constant and $\lambda_r$ an eigenvalue of $\mB^1_{\bomega}$.

For finding eigenfunctions $\Phi(\beta,\gamma)$ of $\mB^1_{\bomega}$, we assume that these functions can be written as

\begin{equation}\label{eq:separationofphi}
 \Phi(\beta,\gamma)=B(\beta)\,C(\gamma).
\end{equation}
We then find:

\begin{equation}
\left\{ \frac{\sin{\beta}}{ B(\beta)} \frac{d}{d \beta }\left(\sin \beta \frac{d B(\beta)}{d \beta} \right)  - \left(\frac{D_{33}}{D_{44}} r^2 \cos^2 \beta + \frac{\lambda_r}{D_{44}}\right) \sin ^2 \beta \right\} + \frac{1}{C(\gamma)} \frac{d^2 C(\gamma)}{d \gamma^2} = 0.
\end{equation}
The resulting equation for $C(\gamma)$ can be solved straightforwardly. Taking into account the $2\pi$-periodicity of $C$, we find that $C(\gamma)$ is a multiple of $e^{i m \gamma}$ for some $m \in \mathbb{Z}$. Normalization then gives:

\begin{equation}\label{eq:eigenfunctionC}
C(\gamma) \equiv C_m(\gamma) = \frac{1}{\sqrt{2\pi}}e^{i m \gamma}, \qquad m \in \mathbb{Z}.
\end{equation}

\subsection{Spheroidal Wave Equation}
The equation for $B(\beta)$, with separation constant $m$ as above, now becomes:

\begin{equation}
\sin{\beta} \frac{d}{d \beta }\left(\sin \beta \frac{d B(\beta)}{d \beta} \right)  + \left [ -\left(\frac{D_{33}}{D_{44}} r^2 \cos^2 \beta + \frac{\lambda_r}{D_{44}}\right) \sin ^2 \beta - m^2 \right] B(\beta) = 0.
\end{equation}
With the substitution $x = \cos \beta$, $y(x) = B(\beta)$ (which is commonly done for equations of this type), we get:

\begin{equation}\label{eq:maineq}
(1-x^2) \frac{d^2 y(x)}{dx^2} - 2x \frac{d y(x)}{dx} +  \left[ -\frac{D_{33}}{D_{44}} r^2 x^2 - \frac{\lambda_r}{D_{44}} - \frac{m^2}{1-x^2} \right]y(x) = 0, \qquad -1 \leq x \leq 1.
\end{equation}
We introduce two more parameters to bring this equation in a standard form:

\begin{equation}
\rho := \sqrt{\frac{D_{33}}{D_{44}}} r, \qquad \tlambda_\rho = - \frac{\lambda_r}{D_{44}}.
\end{equation}
We then find:

\begin{equation}\label{eq:spheroidalwave}
\frac{d}{dx} \left[ (1-x^2) \frac{dy(x)}{dx}\right] + \left[ \tlambda_\rho - \rho^2 x^2 - \frac{m^2}{1-x^2} \right] y(x) =0.
\end{equation}
This equation is known as the spheroidal wave equation (SWE) \cite[Eq. 30.2.1]{flammer_spheroidal_1957,olver_nist_2010}, for which the eigenvalues and eigenfunctions, commonly referred to as spheroidal eigenvalues and spheroidal wave functions, are known and can be found up to arbitrary accuracy. In the remainder of this article, we denote them with $\tlambda_\rho^{l,m}$ and $\Slmrho$, respectively, for $l \in \mathbb{N}_0$, $m \in \mathbb{Z}$, $|m| \leq l$. For explicit analytic representations, see \eqref{def:spheroidalwavefunctions} in \ref{ap:SWF}.
% The derivation of these eigenfunctions is discussed in \ref{ap:SWF}.

 \begin{figure}[t!]
   \centering
   \includegraphics[width=0.9\textwidth]{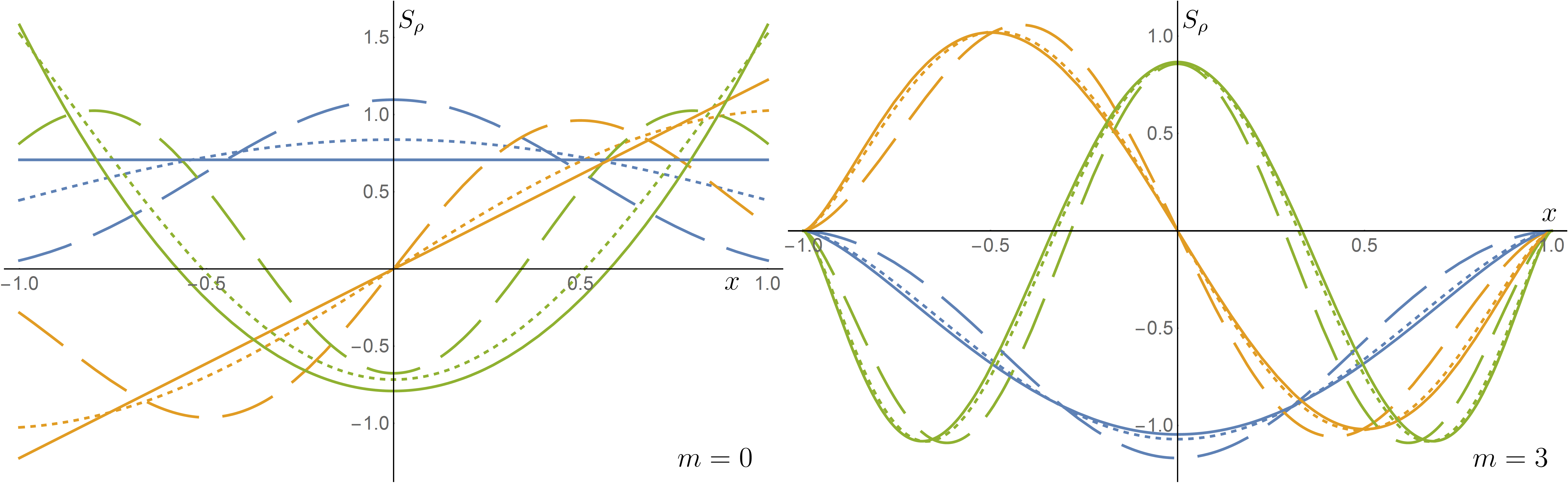}
 \caption{Plot of the spheroidal wave functions $S^{l,m}_{\rho}(x)$ for $m = 0$ (left) and $m = 3$ (right), with $\rho = 0, 2, 5$ (indicated with solid, small dashed and long dashed lines, respectively) and $l = m, m+1, m+2$ (indicated with blue, yellow and green, respectively).}\label{fig:spheroidalwavefunctions}
 \end{figure}

 \begin{figure}[t!]
   \centering
   \includegraphics[width=0.9\textwidth]{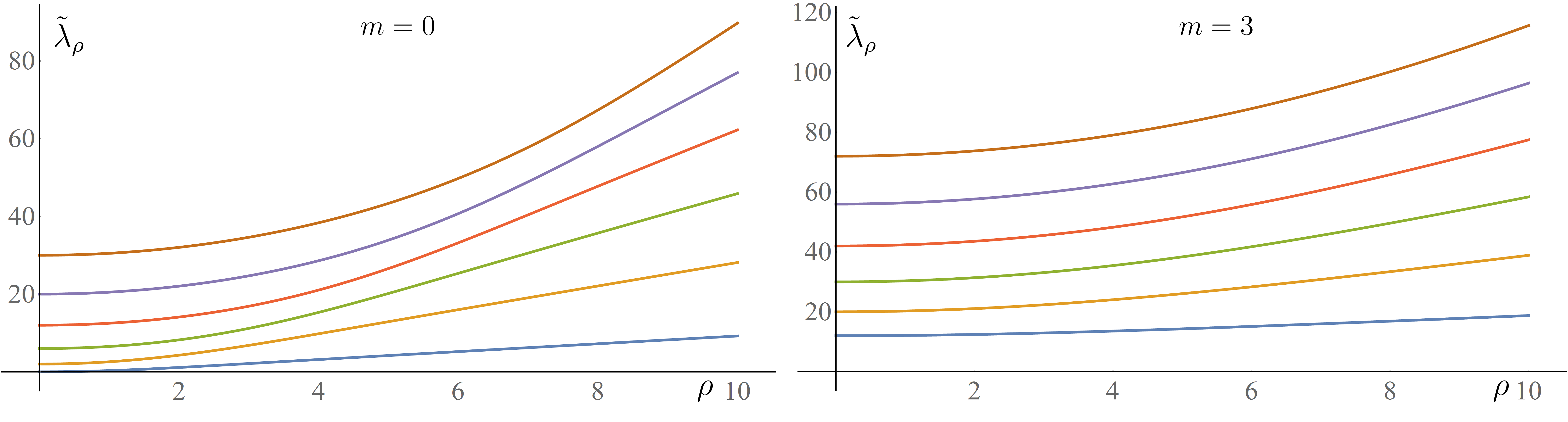}
 \caption{Plot of the spheroidal eigenvalues $\tlambda^{l,m}_\rho$ as a function of $\rho$, for $m = 0$ (left) and $m=3$ (right), and $l = m, \dots, m + 4$ (top to bottom). The eigenvalues are real for all $\rho$.}\label{fig:spheroidaleigenvalues}
 \end{figure}

 The (normalized) spheroidal wave functions and the spheroidal eigenvalues as a result of our computations are displayed in Figs. \ref{fig:spheroidalwavefunctions} and \ref{fig:spheroidaleigenvalues} for a selection of values of $l$ and $m$. It can be seen that all eigenvalues are real and all eigenfunctions are real-valued and vary continuously with parameter $\rho$. In the next section, we show that the spheroidal wave functions form a complete orthonormal basis for $\mathbb{L}_2([-1,1])$, from which it follows that the eigenfunctions $\Philmom$ form a complete orthonormal basis for $\mathbb{L}_2(S^2)$.

\subsection{Sturm-Liouville Form}
The spheroidal wave equation can be written in Sturm-Liouville form:
\begin{equation}
(L y)(x) = \frac{d}{dx}\left[ p(x) \frac{dy(x)}{dx} \right] + q(x) y(x) = -\tlambda_\rho w(x) y(x), \quad x \in [-1,1].
\end{equation}
For this we choose $p(x) = (1-x^2)$, $q(x) = -\rho^2x^2 - \frac{m^2}{1-x^2}$, and we have weight function $w(x) =1$. In this formulation, $p(x)$ vanishes at the boundary of the interval, which makes our problem a singular Sturm-Liouville problem (on a finite interval). It is sufficient to require boundedness of the solution and its derivative at the boundary points to have nonnegative, distinct, simple eigenvalues and existence of a countable, complete orthonormal basis of eigenfunctions $\{y_k\}_{k=1}^{\infty}$ \cite{margenau_mathematics_1956} for the spheroidal wave equation. Since the weight function $w(x) = 1$, orthogonality is understood in the following sense:

% The right-inverse of $L$ is given by some $K: \dots$, such that

% \begin{equation}
% (K f)(x) = \int_{-1}^1 k(x,y)f(y) dy,
% \end{equation}
% with $k = \dots$. It can be shown that $K$ is compact and has a complete basis of eigenfunctions. 
	
\begin{equation}\label{def:orthogonality}
\int_{-1}^1 y_k(x) \overline{y_l(x)}w(x) \; \rmd x = \int_{-1}^1 y_k(x) \overline{y_l(x)} \; \rmd x = \delta_{kl}.
\end{equation}
For our choice of $p(x), q(x)$ and $w(x)$, with $m$ fixed, we have $y_k(\cos \beta) = S_\rho^{l,m}(\cos \beta)$ (with $k = l^2 + l + 1 + m$) and corresponding eigenvalues $\tilde{\lambda}_\rho^{l,m}$.

\subsection{Main Theorem}
From the above considerations, we can come to the main result of this section.

\begin{theorem}\label{th:maintheoremCE}
The normalized eigenfunctions $\Philmom$ of the operator $\mathcal{B}^1_{\bomega}$ in (\ref{eq:operatorBwbetagamma}) are given by:

\begin{equation}\label{eq:phitheorem}
\Philmom(\bn^{\bomega}(\beta,\gamma)) = \Slmrho(\cos \beta) \; C_m (\gamma), \qquad l \in \mathbb{N}_0, m \in \mathbb{Z}, |m|\leq l, \quad \rho = \sqrt{\frac{D_{33}}{D_{44}}} ||\bomega||.
\end{equation}
Here $\Slmrho$ is the $\mathbb{L}_2([-1,1])$-normalized eigenfunction of Eq. (\ref{eq:spheroidalwave}), given in \eqref{def:spheroidalwavefunctions}, with corresponding eigenvalues $\lambdalmr = -D_{44}\tlambdalmrho$, with $\tlambdalmrho$ the standard eigenvalues of the SWE as in Eq. (\ref{eq:spheroidalwave}). Function $C_m$ is as in Eq. (\ref{eq:eigenfunctionC}). The solution of the hypo-elliptic diffusion on $\quot$ Eq. (\ref{eq:CEdiff}) is given by:

\begin{equation}\label{eq:solutionCEviainvFourier}
\begin{aligned}
W(\by,\bn,t) &= \left( K_t^1 \ast_{\quot} U \right)(\by,\bn), \qquad \text{with} \\
K_t^1 (\by,\bn) &= \mFR^{-1}\left( \bomega \mapsto \sumlm  \overline{\Philmom(\be_z)} \Philmom(\bn) \; e^{\lambdalmr t} \right)(\by).
\end{aligned}
\end{equation}

\end{theorem}

\begin{proof}
From the separation of variables approach above, and Eq. \eqref{def:orthogonality}, it follows that the $\Philmom$ are indeed normalized eigenfunctions. From the fact that the resolvent of operator $B^1_{\bomega}$ is compact, which follows from Sturm-Liouville theory, we obtain, using the spectral decomposition of compact self-adjoint operators, a complete orthonormal basis of eigenfunctions of $\mathbb{L}_2(\quot)$. As $\dim(S^2) = 2$, we can number the orthonormal basis with two indices $l,m$, similar to the spherical harmonics, that appear in the special case $\bomega = \mathbf{0}$. This allows us to use the expression in (\ref{eq:Khatseries}) and the result follows.
\end{proof}

In the particular case of $\bomega = 0$, the operator $\mB^1_{\bomega}$ reduces to $\mB^1_{\mathbf{0}} = D_{44} \Delta_{S^2}$, for which the spherical harmonic functions are the eigenfunctions. For the spherical harmonics, we use the following convention:

\begin{equation}\label{eq:SH}
\Delta_{S^2} \Ylm(\beta,\gamma) = -l(l+1)\Ylm(\beta,\gamma), \qquad \Ylm(\beta,\gamma) =  \frac{\varepsilon_m}{\sqrt{2\pi}}   P_l^m(\cos \beta) e^{i m \gamma},
\end{equation}
with the associated Legendre polynomials $P_l^m$ as defined in \eqref{def:plm}, and

\begin{equation}
\qquad \varepsilon_m = \left\{  \begin{matrix}
(-1)^m && \qquad m \geq 0, \\
1 && \qquad m < 0.
\end{matrix}\right.
\end{equation}
In the following corollary we write the eigenfunctions $\Philmom$ in (\ref{eq:phitheorem}) directly in terms of the spherical harmonics.

\begin{corollary}\label{corr:relationSHSWE}
The $\mathbb{L}_2(\quot)$-normalized eigenfunctions $\Philmom$ can be directly expressed in the spherical harmonics as given in Eq. (\ref{eq:SH}):

\begin{equation}
\Philmom(\bn^{\bomega}(\beta,\gamma)) = \sum_{j = 0}^{\infty} \frac{d_j^{l,m}}{||d^{l,m}||} Y^{|m|+j,m}(\beta,\gamma),
\end{equation}
with coefficients $d_j^{l,m} = d_j^{l,m}(\rho)$ depend only on $\rho = \sqrt{\frac{D_{33}}{D_{44}}} || \bomega||$, and are given by Eq. (\ref{eq:matrixeqCE}) in \ref{ap:SWF}. 

\end{corollary}

\subsection{Time-Integrated Processes and Resolvent Kernels}
Under the assumption of exponentially distributed traveling times, the probability of finding a particle at a certain position and orientation can be expressed in terms of the resolvent $(Q_1 - \alpha I)^{-1}$ of the generator $Q_1$, recall Eq. (\ref{eq:CEoperator}) and Section \ref{se:stopro}. It can be seen that for eigenfunctions of $\mB^1_{\bomega}$, we have:

\begin{equation}
(\mB^1_{\bomega} - \alpha I)^{-1} \Philmom = \frac{1}{\lambdalmr - \alpha} \Philmom.
\end{equation}
It follows that the resolvent kernel is given by

\begin{equation}
\hat{R}_\alpha^1(\bomega, \bn) := - \alpha \left( (\mB_\bomega^1 - \alpha I)^{-1} \delta_{\be_z} \right) = \alpha \sumlm \frac{1}{\alpha - \lambdalmr} \overline{\Philmom(\be_z)} \Philmom(\bn).
\end{equation}
Thereby the probability density of finding a random walker at a certain position $\by$ with orientation $\bn$, regardless of the traveling time, is given by:

\begin{equation}
\begin{aligned}
P^1_\alpha(\by,\bn) &= \left(\mF^{-1}_{\mathbb{R}^3} (\hat{R}_\alpha^1) \ast_{\quot}  U \right) (\by,\bn)  \\ &= \left(\mF^{-1}_{\mathbb{R}^3} \left(\bomega \mapsto \alpha \sumlm \frac{1}{\alpha - \lambdalmr} \overline{\Philmom(\be_z)} \Philmom(\cdot) \right) \ast_{\quot} U \right)(\by,\bn).
\end{aligned}
\end{equation}

\subsection{From the Hypo-Elliptic Diffusion Kernel to the Elliptic Diffusion Kernel}
So far in this section we have restricted our diffusion by choosing only $D_{33}$ and $D_{44}= D_{55}$ as nonzero entries in $\bD$, recall (\ref{eq:generator}), motivated from the use of this process in applications. However, even for elliptic diffusion it is still possible to obtain exact solutions, with just a simple transformation from the hypo-elliptic case. We are still required to use only legal generators, in the sense discussed in Section \ref{se:legaldiffusion}. Furthermore, for all differentiable functions $\tilde{U}:SE(3) \to \mathbb{R}$ induced by $U:\mathbb{R}^{3} \rtimes S^{2} \to \mathbb{R}$ via (\ref{eq:quotientgroupconstraint}) one has $\mathcal{A}_{6} \tilde{U}=(\mathcal{A}_{6})^2 \tilde{U}=0$.

Hereby, also the case of elliptic diffusion can be considered, i.e., $\bD = \text{diag}(D_{11},D_{11},D_{33},D_{44},D_{44},0) >0$, such that the generator $\tQ_E$  of the evolution on SE(3) and the generator $Q_E$ on $\quot$ become:

\begin{equation}
\begin{aligned}
\tQ_E := D_{11}(\mA_1^2 + \mA_2^2) + D_{33} \mA_3^2 + D_{44}(\mA_4^2 + \mA_5^2) \implies Q_E :=& D_{11} || \bn \times \nabla ||^2 + D_{33} (\bn \cdot \nabla)^2 + D_{44} \Delta_{S^2} \\ =& D_{11} || \bn \times \nabla ||^2 + Q_1.
\end{aligned}
\end{equation}
Recall Remark~\ref{rem:relationgroupquotientaction} for the relation between the group and quotient generators. The operator $\mB_\bomega$ on $\mathbb{H}_2(S^2)$, obtained as before from applying a Fourier transform in the spatial variables, changes accordingly:

\begin{equation}\label{eq:operatorBomegaE}
\mB_\bomega^E = \mB_\bomega^1 - D_{11} r^2 \sin^2 \beta = D_{44} \Delta_{S^2} - (D_{33} - D_{11})r^2 \cos^2 \beta - r^2 D_{11}.
\end{equation}

It is now fairly straightforward to obtain from Theorem~\ref{th:maintheoremCE} the following corollary:

\begin{corollary}\label{cor:ellipticcase}
Let $D_{33}>D_{11} >0$ and $t>0$. Then
the elliptic heat kernel $K_{t}^{E}=e^{t Q_{E}} \delta_{(\ul{0},\ul{e}_{z})}$ is given by
\begin{equation}\label{eq:Ell}
\begin{aligned}
K_t^E (\by,\bn) &= \mFR^{-1}\left( \bomega \mapsto \sumlm  \overline{\PhilmomE(\be_z)} \PhilmomE(\bn) \; e^{(\lambdalmr  - \|\bomega\|^2 D_{11}) t} \right)(\by).
\end{aligned}
\end{equation}
with $\PhilmomE(\ul{n}^{\bomega}(\beta,\gamma))= S_{\rho \sqrt{\frac{D_{33}-D_{11}}{D_{33}}}}(\cos \beta)\, C_{m}(\gamma)$. For $D_{11} \downarrow 0$, one recovers the hypo-elliptic diffusion kernel $K_{t}^{1}$ computed in Theorem~\ref{th:maintheoremCE}.
\end{corollary}
\begin{proof} Recall that $\rho = \sqrt{\frac{D_{33}}{D_{44}}} ||\bomega||= \sqrt{\frac{D_{33}}{D_{44}}} r$. Then the result follows by (\ref{eq:operatorBomegaE}) and the transformation
\[
r \mapsto \sqrt{\frac{D_{33}-D_{11}}{D_{33}}}r \textrm{ and } \lambda \mapsto \lambda - r^2 D_{11}.
\]
Finally, the limit $D_{11} \downarrow 0$ can be interchanged with the sum in the series, since by application of the Weiertrass criterion (and the existence of a uniform bound on all eigenfunctions $\PhilmomE(\bn)$) the series is uniformly converging for all $t>0$.
\end{proof}

% As a result, the kernel in the Fourier can domain can be adjusted by including a factor $e^{-r^2 D_{11}t}$:

% \begin{equation}
% \hK_t^e (\bomega) = e^{-r^2 D_{11}t} \hK_t^{1,D_{33}-1,D_{44}}(\bomega).
% \end{equation}
% and

% \begin{equation}
% \begin{aligned}
% \rho^e = \sqrt{\frac{D_{33}-D_{11}}{D_{44}}} r, \qquad \left(\rho =\sqrt{\frac{D_{33}}{D_{44}}}r \right),\\
% \lambda_\rho^e = \frac{-\lambda_r + r^2 D_{11}}{D_{44}}, \qquad \left( \lambda_\rho = -\frac{\lambda_r}{D_{44}} \right).
% \end{aligned}
% \end{equation}
% Thereby, we have the following theorem:

% \begin{theorem}\label{th:hypoelliptic}
% Let $D_{11}, D_{33}, D_{44}, t > 0$. All legal heat kernels on $\quot$ are given by

% \begin{equation}
% \hK_t^e(\bomega,\bn) = \sumlm \overline{\Phi_{l,m}^{\bomega,e}(\be_z)} \Phi_{l,m}^{\bomega,e}(\bn)e^{(\lambda_r^{l,m} - r^2 D_{11})t},
% \end{equation}
% with

% \begin{equation}
% \Phi_{l,m}^{\bomega,e}(\bn^\bomega(\beta,\gamma)) = S_{\rho \sqrt{\frac{D_{33}-D_{11}}{D_{44}}}}(\cos \beta)C_{m}(\gamma),
% \end{equation}
% where $D_{11} > 0$ corresponds to the elliptic case, and $D_{44} = 0$ corresponds to the hypo-elliptic case.
% \end{theorem}

% \subsection{Numerical results}
% We visualize the Green's function in Fig. ...

\section{Derivation of the Exact Solutions for Convection-Diffusion on \texorpdfstring{$\quot$}{R3xS2}}\label{se:exactCC}
In this section we consider the second central equation of the paper, Eq. (\ref{eq:differentialQ}) with $i = 2$, for the 3D direction process or convection-diffusion process. We focus here on the time-integrated process, as this has proven to be more useful in applications. We show how exact solutions for the resolvent kernel can be found. Similar to the approach in Section \ref{se:exactCE}, we derive eigenfunctions for the corresponding evolution operator in the Fourier domain. However, this operator can no longer be transformed into the standard Sturm-Liouville form. We therefore use the framework of perturbations of self-adjoint operators \cite{kato_operators_1976,makin_summability_2012} to prove important properties of the eigenvalues and to prove completeness of the eigenfunctions.

The convection-diffusion system that we consider is the following:

\begin{equation}\label{eq:CCdiff}
\begin{cases}
\dfrac{\partial}{\partial t} W(\by,\bn,t) = Q_2 W (\by,\bn,t) = (- (\bn \cdot \nabla_{\mathbb{R}^3}) + D_{44} \Delta_{S^2}) W(\by,\bn,t), & \quad \by \in \mathbb{R}^3, \bn \in S^2, t \geq 0,\\
W(\by,\bn,0) = U(\by,\bn), &\quad \by \in \mathbb{R}^3, \bn \in S^2.
\end{cases}
\end{equation}
In this section, we focus on the time-integrated, resolvent process (although the solution strategy is the same):

\begin{equation}\label{eq:resolventCC}
\left( (\bn \cdot \nabla_{\mathbb{R}^3}) - D_{44} \Delta_{S^2} - \alpha I \right) P^2_\alpha(\by,\bn) = \alpha \,  U(\by,\bn).
\end{equation}
Again we fix $\bomega \in \mathbb{R}^3$ and the operator $\mB^2_{\bomega}$ (superscript 2) corresponding to $Q_2$ now becomes:
\begin{equation}\label{eq:operatorBwCC}
\mathcal{B}^2_{\bomega} = D_{44} \Delta_{S^2} - (i \bomega \cdot \bn).
\end{equation}
When we express $\bn$ in spherical coordinates $\beta,\gamma$ with respect to $\bomega$ as was done in (\ref{def:coordinatechoice}) and Fig.~\ref{fig:figureParametrization}, the differential operator in $\beta,\gamma$ becomes:

\begin{equation}
\mB^2_{\bomega} = \frac{D_{44}}{\sin{\beta}} \frac{\partial}{\partial \beta} \left(\sin \beta \frac{\partial}{\partial \beta} \right) + \frac{D_{44}}{\sin^{2} \beta} \frac{\partial ^2}{\partial \gamma^2} -  (i r \cos \beta).
\end{equation}
Since the Laplace-Beltrami operator is symmetric and the multiplication operator has a purely imaginary symbol, the operator $\mB^2_{\bomega}$ in this case is not symmetric and not self-adjoint, but does satisfy:

\begin{equation}\label{eq:adjointproperty}
\mB_{\bomega}^{2,*} f = \overline{\mB^2_{\bomega} \overline{f}}, \qquad \text{ for all } \quad f \in \mathbb{H}_2(S^2).
\end{equation}
Here we note that the domain of the closed unbounded operator $\mB_\bomega^2$ is the Sobolev space $\mathbb{H}_2(S^2)$, equipped with the $\mathbb{L}_2(S^2)$-norm. It will turn out to be useful to regard $\mB^2_\bomega$ as the sum of a self-adjoint operator and a bounded operator $\mM: \mathbb{L}_2(S^2) \rightarrow \mathbb{L}_2(S^2)$, that just applies a multiplication, $(\mM f)(\bn^\bomega(\beta,\gamma)) = \cos \beta \cdot f(\bn^\bomega(\beta,\gamma))$:

\begin{equation}
\mB^2_\bomega = D_{44} \Delta_{S^2} - i r \cos \beta = D_{44} \Delta_{S^2} - i r \mM, \qquad r = ||\bomega||.
\end{equation}
So in particular, as before, the operator $\mB^2_{\mathbf{0}} = D_{44} \Delta_{S^2}$ has the spherical harmonics as eigenfunctions. We denote the eigenfunctions of $\mB^2_\bomega$ with $\Psilmom$ and the corresponding eigenvalues with $\lambdalmr$, even though the eigenvalues are not the same as in Section \ref{se:exactCE}. Again we assume that the eigenfunctions can be written as the product-form $\Psilmom(\bn^{\bomega}(\beta,\gamma)) = B(\beta)C(\gamma)$, leading to two ordinary differential equations. The equation with variable $\gamma$ is the same as before, with solutions $e^{i m \gamma}$. The equation for variable $\beta$ can be written as

\begin{equation}
\frac{1}{\sin \beta} \frac{d}{d \beta }\left(\sin \beta \frac{d B(\beta)}{d \beta} \right)  + \left [ -\left(\frac{1}{D_{44}} i r \cos \beta + \frac{\lambda_r}{D_{44}}\right) - \frac{m^2}{\sin^2 \beta} \right] B(\beta) = 0,
\end{equation}
i.e., as

\begin{equation}\label{eq:CCrewritten}
\frac{1}{\sin \beta} \frac{d}{d \beta }\left(\sin \beta \frac{d B(\beta)}{d \beta} \right)  + \left [ -\left( i \rho \cos \beta + \tlambda_\rho \right) - \frac{m^2}{\sin^2 \beta} \right] B(\beta) = 0,
\end{equation}
now with 

\begin{equation}\label{eq:relationrholambdarho}
\rho = \frac{r}{D_{44}}, \qquad \tlambda_\rho = \frac{\lambda_r}{D_{44}}.
\end{equation}
We define the differential operator $\mB_\rho^m$ as

\begin{equation}\label{def:operatorBrhom}
\mB_\rho^m := \left(\frac{1}{\sin \beta} \frac{d}{d \beta} \left( \frac{1}{\sin \beta} \frac{d}{d \beta} \right) \right) - \frac{m^2}{\sin^2 \beta} - i \rho \cos \beta = \mB_0^m - i \rho \mM, 
\end{equation}
with slight abuse of notation, since now $\mM: \mathbb{L}_2([0,\pi]) \rightarrow \mathbb{L}_2([0,\pi])$. Then Eq. (\ref{eq:CCrewritten}) can be rewritten as:

\begin{equation}
\mB^m_\rho B(\beta) = \tlambda_\rho B(\beta), \qquad \beta \in [0,\pi].
\end{equation}
We explicitly denote the dependence on $\rho$ and the separation constant $m$, as we need it later in our spectral analysis of the operator.
% \begin{remark}
% Insert a remark here, similar to remark 2.4, explaining why in this case we can still expect a complete system of eigenfunctions.
% \end{remark}

\subsection{The Generalized Spheroidal Wave Equation}
% The equation for $B(\beta)$ now has the form of a specific case of the generalized spheroidal wave equation (GSWE) \cite[Sec. 30.12]{leaver_solutions_1986,olver_nist_2010}. We write the eigenvalue problem once more in full detail:
% \begin{align}
%  \left(\mB_{\rho}^m - \lambda I \right) B(\beta) =\frac{1}{\sin \beta} \frac{d}{d \beta }\left(\sin \beta \frac{d B(\beta)}{d \beta} \right)  + \left [ -\left(\frac{1}{D_{44}} i r \cos \beta + \frac{\lambda}{D_{44}}\right) - \frac{m^2}{\sin^2 \beta} \right] B(\beta) &= 0, \\
%  \frac{1}{\sin \beta} \frac{d}{d \beta }\left(\sin \beta \frac{d B(\beta)}{d \beta} \right)  + \left [ -\left( i \rho \cos \beta + \tlambda \right) - \frac{m^2}{\sin^2 \beta} \right] B(\beta) &= 0.
% \end{align}
After applying the transformation $x = \cos \beta$, $y(x) = B(\beta)$ to Eq. (\ref{eq:CCrewritten}), we get:

\begin{align}\label{eq:maineqCC}
(1-x^2) \frac{d^2 y(x)}{dx^2} - 2x \frac{d y(x)}{dx} +  \left[ -\rho i x - \tilde{\lambda}_\rho - \frac{m^2}{1-x^2} \right]y(x) = 0, \qquad -1 \leq x \leq 1.
\end{align}
This equation now has the form of a specific case of the generalized spheroidal wave equation (GSWE) \cite[Sec. 30.12]{leaver_solutions_1986,olver_nist_2010}.

\begin{remark} In literature \cite{leaver_solutions_1986,figueiredo_generalized_2007}, the generalized spheroidal wave equation also appears in the following form:
\begin{equation}
x(x-x_0)\frac{d^2 y(x)}{dx^2} + (C_1 + C_2x) \frac{dy(x)}{dx} + [\omega^2 x(x - x_0) - 2\eta \omega(x- x_0) + C_3]y(x) = 0.
\end{equation}
In this equation $C_i$, $\omega$, $\eta$ are constants and the equation has singularities at $x= 0$ and $x = x_0$. With an appropriate choice of constants, the GSWE can be brought to the form of Equation (\ref{eq:maineqCC}). However, this does require taking the limit $\omega \rightarrow 0$,  such that $2 \eta \omega$ stays bounded and nonzero. According to \cite{figueiredo_generalized_2007}, this type of limit is considered by both Whittaker and Ince. We are only interested in the solution on the interval $[0, x_0]$. In \cite{leaver_solutions_1986,figueiredo_generalized_2007}, solutions are provided in various series expansions, but not all of them converge on $[0,x_0]$ and not all of them allow for taking the limit $\omega \rightarrow 0$ as above. Analogous to SWE case, we derive solutions that are series of associated Legendre functions. This is different from the solutions in \cite{leaver_solutions_1986,figueiredo_generalized_2007}, but our series is also suited for the case $\omega = 0$.
\end{remark}

We refer to \ref{ap:GSWF} for the derivation of the eigenvalues $\tilde{\lambda}_\rho^{l,m}$ of the GSWE and the corresponding eigenfunctions that we denote with $GS^{l,m}_\rho$. Here we just state that the eigenfunctions of $\mB_{\omega}^{2}$ are given by

\begin{equation}\label{important}
\Psilmom(\bn^{\bomega} (\beta,\gamma)) = GS^{l,m}_\rho(\cos \beta) \frac{e^{i m \gamma}}{\sqrt{2 \pi}}, \quad l \in \mathbb{N}_0, \; m \in \mathbb{Z}, \; |m| \leq l, \rho = \frac{||\bomega||}{D_{44}},
\end{equation}
in which we used the same substitution $x = \cos \beta$ as before. The functions $GS_\rho^{l,m}$ for certain $\rho$, $l$ and $m$ are shown in Fig. \ref{fig:gsweigenfunctions}. Recall \eqref{eq:relationrholambdarho} for the relation $\lambda_r^{l,m} = -D_{44} \tilde{\lambda}_\rho^{l,m}$ between the eigenvalues corresponding to $\Psilmom$ and $GS_\rho^{l,m}$, respectively.

\begin{figure}[t!]
   \centering
   \includegraphics[width=0.9\textwidth]{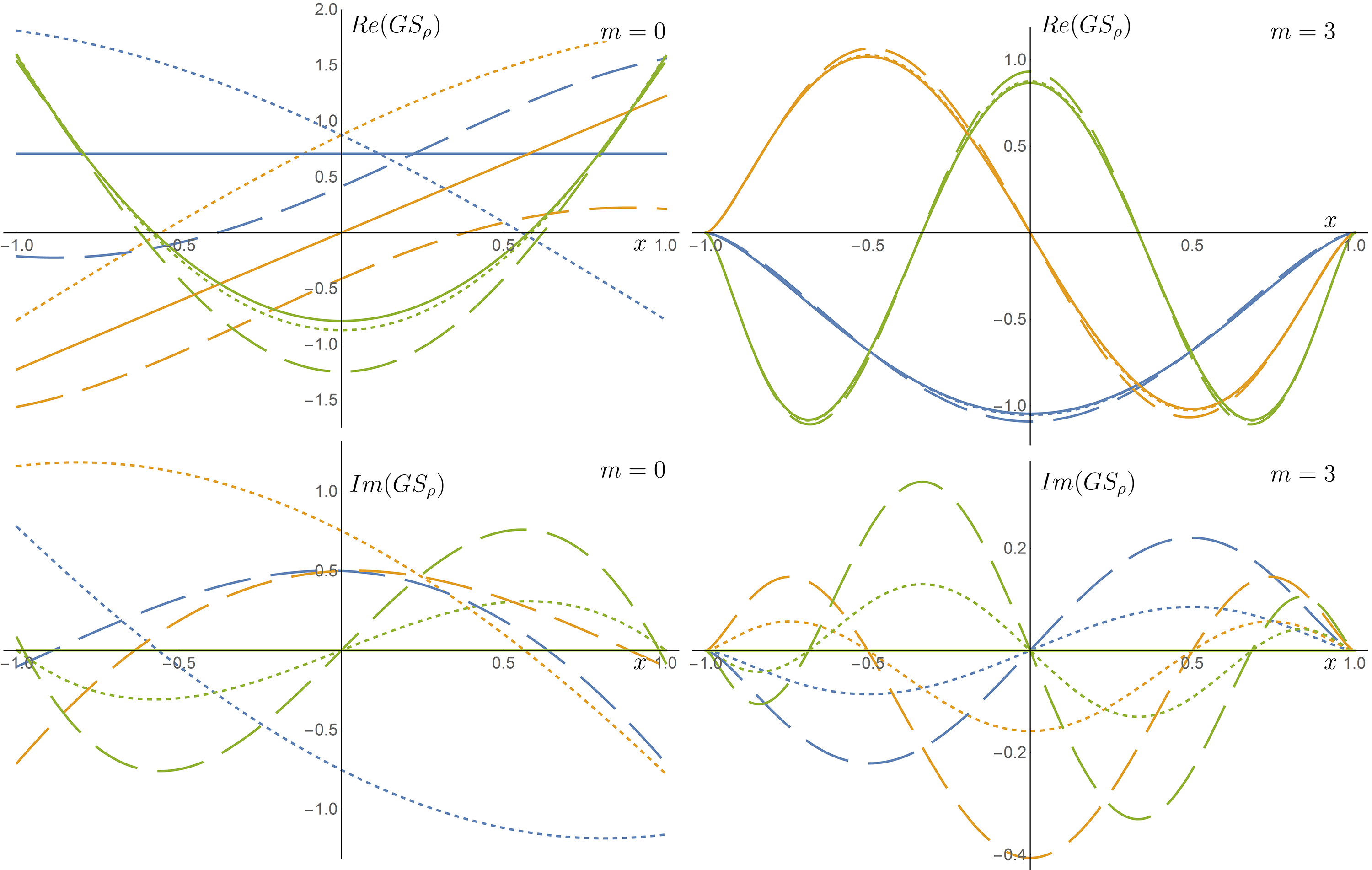}
 \caption{Plot of the real (top) and imaginary (bottom) part of the generalized spheroidal wave functions $GS^{l,m}_{\rho}(x)$ for $m = 0$ (left) and $m = 3$ (right), with $\rho = 0, 2, 5$ (indicated with solid, small dashed and long dashed lines, respectively) and $l = m, m+1, m+2$ (indicated with blue, yellow and green, respectively).}\label{fig:gsweigenfunctions}
 \end{figure}

From property (\ref{eq:adjointproperty}) the following can be derived:

\begin{align}
\lambdalmr (\Psilmom, \overline{\Psilpmpom})_{\mathbb{L}_2(S^2)} = (\mB^2_{\bomega} \Psilmom, \overline{\Psilpmpom})_{\mathbb{L}_2(S^2)} \overset{(\ref{eq:adjointproperty})} = (\Psilmom, \overline{\mB^2_{\bomega} \Psilpmpom})_{\mathbb{L}_2(S^2)} \nonumber \\= (\Psilmom, \overline{\lambdalpmpr} \overline{\Psilpmpom})_{\mathbb{L}_2(S^2)} = \lambdalpmpr (\Psilmom, \overline{\Psilpmpom})_{\mathbb{L}_2(S^2)}.
\end{align}
This implies that

\begin{equation}\label{eq:orthogonalityofpsi}
\lambdalmr = \lambdalpmpr \; \vee \; (\Psilmom, \overline{\Psilpmpom})_{\mathbb{L}_2(S^2)} = 0.
\end{equation}
As a result, we see that if $\{ \Psilmom \}$ is complete and it admits a reciprocal basis $\{\Psi_\bomega^{l,m}\}$ such that \\ $(\Psi_\bomega^{l,m},\Psi_{l',m'}^\bomega) = \delta_{l'}^l \delta_{m'}^m$, then for the reciprocal basis functions we have $\Psi^{l,m}_\bomega = \overline{\Psi_{l,m}^\bomega}$. In the next section, this completeness is further discussed.

\subsection{The Time-Integrated Process}

% It can be seen directly from the differential equation that when $y(x)$ is an eigenfunction if and only if $\overline{(y(-x))}$ is an eigenfunction.

The kernel for the time-integrated process in the spatial Fourier domain corresponds to:

\begin{equation}\label{eq:resolventseriesCC}
\hat{R}_\alpha^2(\bomega, \bn) := -\alpha \left((\mB_\bomega^2 - \alpha I)^{-1}\delta_{e_z} \right)(\bn) =  \alpha \sumlm \frac{1}{\alpha - \lambdalmr} \frac{\overline{\Psilmom(\be_z)} \overline {\Psilmom(\bn)}}{(\Psi^{\bomega}_{l,m}, \overline{\Psi^{\bomega}_{l,m}})} .
\end{equation}
However, there are conditions on the convergence of this series expression. Since operator $\mB_\bomega^2$, in contrast to $\mB_\bomega^1$, is no longer self-adjoint, the standard Sturm-Liouville theory, that ensures completeness of the eigenfunctions with negative, real eigenvalues, cannot be applied. In the following we formulate a lemma, on the eigenvalues of $\mB_\rho^m$, and a theorem, on the eigenfunctions, that combined imply that the convergence holds almost everywhere. Only for particular radii $||\bomega|| = D_{44} \rho^m_n$, for some $\rho^m_n$ in the frequency domain there is no convergence, but it can be shown that for any $m$ this happens only on a countable set of $\{\rho^m_n\}_{n \in \mathbb{N}_0}$ that has no accumulation point. As a result, the series in (\ref{eq:resolventseriesCC}) converges almost everywhere in the Fourier domain, and thereby the inverse Fourier transform, similar to Eq. (\ref{eq:solutionCEviainvFourier}) in Theorem \ref{th:maintheoremCE}, is still well-defined.

In the next lemma we prove properties of the eigenvalues $\lambdalmr$ of the operator $\mB_\bomega^2$ that are necessary to have convergence of the series in (\ref{eq:resolventseriesCC}). For this we also need to consider the operator $\mB_\rho^m$ for fixed $m$, recall the definition in (\ref{def:operatorBrhom}).

% Before proceeding to the lemma, we first remark that the operator $\mB_\bomega^2$ can be written as the sum of a self-adjoint operator and a perturbation term: 

% \begin{equation}
% \mB^2_\bomega = D_{44} \Delta_{S^2} - i r \cos \beta = D_{44} \Delta_{S^2} - i r \mM, \quad \mM := \cos \beta,
% \end{equation}
% with $\mB^2_\bomega$ self-adjoint and unbounded, and $i r \mM$ bounded. So in particular, as before, the operator $\mB^2_{\mathbf{0}} = D_{44} \Delta_{S^2}$, with the spherical harmonics as eigenfunctions. 

% In the following lemma, we derive properties for the eigenvalues of the operator $\mB_\bomega^2$ by considering the operator as the sum of a self-adjoint operator and a perturbation that depends on $r$.

\begin{lemma}[Eigenvalues of the operator $\mB_\bomega^2$]\label{le:eigenvalues} We have the following properties for the eigenvalues of $\mB_\bomega^2$:

\begin{enumerate} 
\item Let $m \in \mathbb{Z}$, then there exists a $\rho^m_*>0$ such that $\mB_\rho^m$ has real eigenvalues. Moreover, there is at most a countable set $\{\rho^m_n\}_{n=1}^\infty$ where two eigenvalues collide and branch into a complex conjugate pair of eigenvalues.
\item For all $r \geq 0$ the real part of the eigenvalues of $\mB_\bomega^2$ is negative.
\end{enumerate}
\begin{proof}
We prove the two points subsequently:
\begin{enumerate}
\item For $\rho = 0$, the operator $\mB_0^m$, recall (\ref{def:operatorBrhom}), is self-adjoint, negative semi-definite and therefore all eigenvalues $\tlambda_0^{l,m} = -l(l+1)$, $l \geq|m|$ are real, negative and simple. From the spectral inclusion theorem \cite{veselic_spectral_2007} it follows for the spectrum $\sigma(\mB_\rho^m)$ that:

\begin{equation}
\sigma(\mB^m_\rho) \subset \{\lambda \in \mathbb{C} \; \vline \; \text{dist}(\lambda, \sigma(\mB_{0}^m)) \leq ||i \rho \mM||\}.
\end{equation}
The operator norm is $||i \rho \mM|| =  \rho$ and for $\mB_0^m$ the minimal distance between two eigenvalues is the distance between the two smallest eigenvalues, when $l = |m|$ and $l = |m|+1$, resulting in $|\tlambda^{|m|+1,|m|}_{{0}} - \tlambda^{|m|,|m|}_{0}| = 2(|m|+1) > 0$. Therefore, we choose $\rho_*^m = (|m|+1)$ to guarantee that $\tlambda^{|m|+1,|m|}_\rho \neq \tlambda^{|m|,|m|}_\rho$ for all $\rho < \rho_*^m$. It can be observed from Eq. (\ref{eq:maineqCC}) that when $y(x)$ is an eigenfunction for $\lambda$, that $\overline{y(-x)}$ is an eigenfunction for $\bar{\lambda}$. It cannot happen that branching of eigenvalues occurs without two eigenvalues colliding, since the multiplicity of $\lambda(\rho)$ depends continuously on $\rho$. Since we have shown that no eigenvalues can collide for $\rho < \rho^m_*$, we are guaranteed to have real eigenvalues in this case. 

%Moreover, since the eigenvalues for $\mB_0^m$ increase quadratically in $l$, there difference increases linearly in $l$. The values for $\rho^n_m$ for which branching of eigenvalues occurs thereby increase at least linearly in $n$, and therefore these points are countable and cannot have an accumulation point.

Now according to \cite{meixner_grundlagen_1954}, there exists a nonzero analytical function $F(\lambda,\rho)$, such that the equation $F(\lambda,\rho) = 0$ defines the eigenvalues $\lambda_i^m$, $m$ fixed, as functions of $\rho$. We define $\rho^m_n$ to be those values for $\rho$, in increasing order, for which $\lambda_i^m(\rho) = \lambda_j^m(\rho)$ for some $i \neq j$. Due to the analyticity of $F$, the set $\{ \rho_n^m\}_{n = 0}^\infty$ is countable and cannot have an accumulation point \cite{meixner_grundlagen_1954}. We specify the analytic function whose zeros provides for given $m \in \mathbb{Z}$  the values $(\rho^{m}_n)_{n=0}^{\infty}$ later (in (Eq.~\ref{eq:analyticfunc})).

%Now according to \cite{meixner_grundlagen_1954}, there exists a nonzero analytical function $F(\rho)$ \footnote{In \cite{meixner_spharoidfunktionen_1954} the function $F(\rho)$ relates to the Floquet exponent, that arises by outward analytic extension of $y(x)$ in (\ref{eq:maineqCC}).} such that the eigenvalues of the operator $\mB_\rho^m$ form a subset of the roots of this analytical function. From this we conclude that the set of values for $\rho$ where branching occurs is countable and cannot have an accumulation point. 

\item To show that all eigenvalues have a negative real part, it is sufficient to show that the symmetric part of the operator $\mB_\bomega^2$ is negative definite. Indeed for all $f \in \mathbb{H}_2(S^2)$ (dense in $\mathbb{L}_2(S^2)$), we have:

\begin{multline}
\left(\frac{\mB^2_\bomega + (\mB^2_{\bomega})^*}{2} f, f\right) = \frac12 \left( (D_{44} \Delta_{S^2} -r i \mM)f + (D_{44} \Delta_{S^2} + r i \mM)f , f \right) = (D_{44} \Delta_{S^2} f,f) \leq 0.
\end{multline}
\end{enumerate}
\end{proof}

\end{lemma}
The dependency of the eigenvalues on $\rho$ is displayed for two different values of $m$ in Fig. \ref{fig:branchingofeigenvalues}. In the figure, the points $\rho_n^m$ for which $\mB_\rho^m$ has two colliding eigenvalues are indicated with red dots. The points $\rho_n^m$ are in fact zeros of the analytic function

\begin{equation}\label{eq:analyticfunc}
\rho \mapsto (\Psilmom,\overline{\Psilmom}) = \int_{S^2} (\Psilmom(\bn)^2) \; \rmd \sigma(\bn),
\end{equation}
where the right hand side only depends on $\rho = D_{44}\|\bomega\|$. Moreover, for the behavior of the eigenvalues we have

\begin{equation}
\mathrm{Im}(\lambda_\bomega^{l,m}) = \frac12 \dfrac{((i\mB_\bomega^2 - i (\mB_\bomega^2)^*)\Psilmom,\Psilmom)}{(\Psilmom,\Psilmom)} = r \frac{(\mathcal{M}\Psilmom, \Psilmom)}{(\Psilmom,\Psilmom)} = \mathcal{O}(r).
\end{equation}

\begin{figure}[t!]
   \centering
   \includegraphics[width=0.9\textwidth]{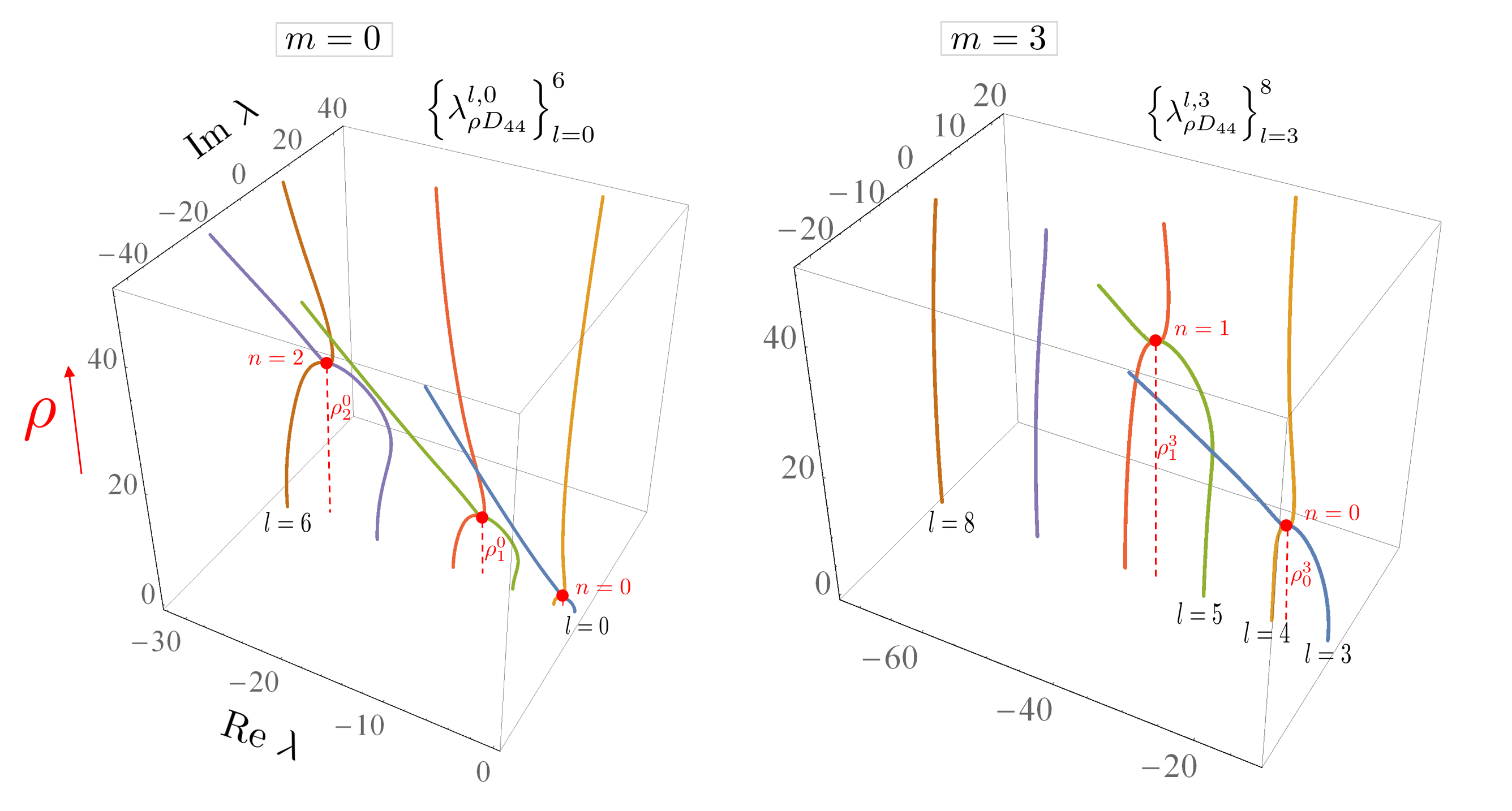}
 \caption{Plot of the real and imaginary parts of the first 6 eigenvalues of $\mB_\rho^m$ for $m = 0$ (left) and $m = 3$ (right). Note that all eigenvalues are real for sufficiently small $\rho$. When $\rho$ increases, each time two eigenvalues collide and branch into two complex conjugate eigenvalue pairs. Comparing the left and right figure, we note that the higher $m$, the higher the values for $\rho$ where this branching occurs. Moreover, we have that $\mathrm{Im}(\lambda_{\rho D_{44}}) \sim \mathcal{O}(\rho D_{44})$.}\label{fig:branchingofeigenvalues}
 \end{figure}
In particular, for $m$ fixed and $\rho \leq \rho_0^m$, all eigenvalues are real, and hence $(\mM \Psilmom, \Psilmom) = 0$. We use Lemma \ref{le:eigenvalues} in the next theorem, which proves the solution of the time-integrated differential equation is unique.

\begin{theorem}\label{th:resolventseries} Let $\bomega \in \mathbb{R}^3$ be given. Let $\alpha >0$. Then
\begin{enumerate}
\item (Existence of the resolvent) the resolvent operator $(\mB^2_\bomega - \alpha I)^{-1}$ exists, i.e., the unbounded operator $(\mB^2_\bomega - \alpha I): \mathbb{H}_{2}(S^{2}) \rightarrow \mathbb{L}_{2}(S^{2})$ is invertible.

\item (Completeness) there exists a complete basis of generalized eigenfunctions of the operator $\mB^2_\bomega$. If $||\bomega|| \neq D_{44} \rho^m_n$, these generalized eigenfunctions are true eigenfunctions and coincide with the eigenfunctions $\Psilmom$ derived above. So for $||\bomega|| \neq D_{44} \rho_n^m$ the resolvent operator $(\mB_\bomega^2 - \alpha I)^{-1}$ is diagonalizable.

\end{enumerate}
\end{theorem}

\begin{proof}
\begin{enumerate}
\item (Existence of the resolvent) Let $\bomega \in \mathbb{R}^3$ and $\alpha > 0$ given, $r = ||\bomega||$. Injectivity of $(\mB^2_\bomega - \alpha I)$ follows from the fact that the resolvent operator is bounded from below. For $f \in \mathbb{H}_2(S^2)$, we have:

\begin{multline}
((\mB^2_\bomega - \alpha I)f, (\mB^2_\bomega - \alpha I)f) = ((D_{44} \Delta_{S^2} + i r \mM)f, (D_{44} \Delta_{S^2} + i r \mM)f) - 2 D_{44} \alpha (\Delta_{S^2}f,f) + \alpha^2 (f,f) \\
\geq \alpha^2 (f,f) \geq 0.
\end{multline}

Hence $(\mB_\bomega^2 - \alpha I)f = 0 \implies f = 0$. 

To show surjectivity, we start by noting that in general, $(\mathcal{R}(\mB^2_\bomega - \alpha I))^{\bot} = \mathcal{N} ((\mB^2_\bomega)^* - \alpha I)$. Now let $f \in (\mathcal{R}(\mB^2_\bomega - \alpha I))^{\bot}$ and $f \in \mathbb{H}_{2}(S^2)$, then 

\begin{equation}
\begin{aligned}
(\mB_\bomega^2)^* f - \alpha f = 0 \; \Longleftrightarrow \; \overline{\mB^2_\bomega \bar{f}} - \alpha f = 0 \; \Longleftrightarrow \; \mB^2_\bomega \bar{f} = \alpha \bar{f} \; \Longleftrightarrow \; (\mB^2_\bomega - \alpha I)\bar{f} = 0, 
\end{aligned}
\end{equation}
but injectivity of $\mB_\bomega^2 - \alpha I$ then implies that $\bar{f} = 0 = f$. It follows that $(\mathcal{R}(\mB^2_\bomega - \alpha I))^{\bot}$ equals $\{0\}$ and because of closedness of both $\mB^2_\bomega$ and $I$, the surjectivity follows from the closed range theorem \cite{yosida_functional_1980}. Hence $(\mB_\bomega^2 - \alpha I)^{-1}$ exists.

\item (Completeness) We first consider the operator $\mB_\rho^m$. By direct computation it can be shown that the multiplication operator $\rho \mM$ is bounded, with $\rho$ as a bound for the operator norm. For $\rho = 0$, $\mB_0^m$ has simple eigenvalues. Thereby, according to Kato \cite[Ch. V, Sect. 5, Th. 4.15a]{kato_operators_1976} there exists a complete basis of generalized eigenfunctions. Then $\mB_\rho^m$ is closed with compact resolvent $(\mB_\rho^m - \alpha I)^{-1}$.

In the case of $\rho \neq \rho_n^m$, we still need to show that the basis of generalized eigenfunctions correspond to the actual eigenfunctions $\Psilmom$ as computed above. For $\|\bomega\| < D_{44}\rho^*$ this is clear. For $\rho \geq \rho^*$, it follows by analytic extension in $\rho$, which is exactly what happens when we write down the eigenfunctions as a series of Legendre functions as in (\ref{eq:GSlmrho}), as this boils down to a Taylor series in $\rho$.

Furthermore, the functions $GS^{l,m}_\rho$ are uniformly bounded on $[-1,1]$ and thereby form a Riesz basis \cite{makin_summability_2012}, which makes the reciprocal basis unique. The reciprocal basis has the property that

\begin{equation}
(\Psi_\bomega^{l,m} , \Psi_{l',m'}^{\bomega}) = \delta_{l'}^l \delta_{m'}^m, \qquad l,l' \in \mathbb{N}_0, \; m, m' \in \mathbb{Z}, \; |m|\leq l, \; |m'| \leq l'.
\end{equation}
In fact we have $\Psi_\bomega^{l,m} = \mS^{-1}\Psi^\bomega_{l,m}$, where $\mS : \mathbb{L}_2(S^2) \rightarrow \mathbb{L}_2(S^2)$ denotes the frame operator, given by

\begin{equation}
\mS f = \sumlm (f, \Psi^{\bomega}_{l,m})\Psi_{l,m}^{\bomega}.
\end{equation}

From the properties (\ref{eq:adjointproperty}) and (\ref{eq:orthogonalityofpsi}) it follows that the reciprocal basis $\Psi_\bomega^{l,m}$ of $\Psilmom$ is linearly proportional to the conjugate basis, which implies that the reciprocal basis $\{\Psi_\bomega^{l,m} \}$ is complete. Therefore, for any $f \in \mathbb{L}^2(S^2)$, there is the convergent series representation

\begin{equation}
(\mB_\bomega^2 - \alpha I)^{-1}f = \sumlm \frac{1}{\lambda_r^{l,m} - \alpha} \frac{(f, \overline{\Psi^{\bomega}_{l,m}}) \Psi^{\bomega}_{l,m}}{(\Psi^{\bomega}_{l,m}, \overline{\Psi^{\bomega}_{l,m}})}, \qquad D_{44}||\bomega|| \neq \rho^m_n.  
\end{equation} 

Hence for $\rho \neq \rho_n^m$, $(\mB_\bomega^2 - \alpha I)^{-1}$ is diagonalizable with eigenfunctions $\Psi_{l,m}^\bomega$ and eigenvalues $1/(\lambda_r^{l,m} - \alpha)$ with strictly negative real part.   

\end{enumerate}
\end{proof}

\subsubsection{Main Theorem}
The following theorem summarizes the result regarding eigenfunctions of the operator $\mB^2_{\bomega}$ corresponding to the generator $Q_2$:

\begin{theorem}\label{th:maintheoremCC}
The eigenfunctions $\Psilmom$ of the operator $\mB^2_{\bomega}$ in (\ref{eq:operatorBwCC}) are given by:

\begin{equation}
\Psilmom(\beta,\gamma) = GS_\rho^{l,m}(\cos \beta) C_m(\gamma).
\end{equation}

Here $GS_\rho^{l,m}$ is the eigenfunction, given by \eqref{eq:GSlmrho}, of the generalized spheroidal wave equation. For almost every $\bomega \in \mathbb{R}^3$, these eigenfunctions form a complete bi-orthogonal system. Therefore the solution of the convection-diffusion equation on $\quot$ is given by:

% \begin{equation}
% W(\by,\bn,t) = \left( \mFR^{-1}\left( \bomega \mapsto \sumlm  \overline{\Psilmom(\be_z)} \overline{\Psilmom(\cdot)} e^{- \lambdalmrho t} \right) \ast_{\quot} U \right) (\by,\bn),
% \end{equation}

\begin{align}
P^2_\alpha(\by,\bn) = \left( R^2_\alpha \ast_{\quot} U \right) (\by,\bn),
\end{align}
with
\begin{equation}
R^2_\alpha(\by,\bn) = \mFR^{-1}\left( \bomega \mapsto \sumlm  \frac{\alpha}{\alpha - \lambdalmr}\frac{\overline{\Psilmom(\be_z)} \overline{\Psilmom(\bn)}}{(\Psilmom,\overline{\Psilmom})} \right)(\by),
\end{equation}
where the series converges in $\mathbb{L}_2(\quot)$-sense. With $\lambdalmr = -D_{44} l(l+1) + \mathcal{O}(r)$ we denote the countable eigenvalues of $\mB^2_{\bomega}$, with $\|\bomega \| = r = \rho D_{44}$. The eigenvalues are disjoint for $\rho \neq 0$, and $\rho \neq \rho^m_n$.
\end{theorem}

\begin{remark}
\begin{itemize}
\item For all $\bomega \in \mathbb{R}^3$, including the cases $||\bomega|| = D_{44} \rho_n^m$, we can decompose the resolvent into a complete basis of generalized eigenfunctions. In fact the projection onto the generalized eigenspace $E_{\tilde{\lambda}_{\rho}^{l,m}}$ is given by

\begin{equation}
P_{E_{\tlambdalmrho}} = \frac{1}{2\pi i} \oint_J (z I - \mB_\rho^m)^{-1} \rmd z, 
\end{equation} 
with $J$ a Jordan curve enclosing only $\tlambdalmrho$ in positive direction, where we recall definition (\ref{def:operatorBrhom}). 

\item Now if $\rho \neq \rho_m^n$, we have one-dimensional eigenspaces, spanned by $GS_\rho^{l,m}$:

\begin{equation}
E_{\tlambdalmrho} = \mathrm{span} \left\{ GS_\rho^{l,m} \right \}  = \mathcal{N}(\mB_\rho^m - \tlambdalmrho I).
\end{equation}
In the case that $\rho = \rho_n^m$, we have instead:

\begin{equation}
E_{\tlambdalmrho} = \mathrm{span} \left\{ \mathcal{N} (\mB_\rho^m - \tlambdalmrho I), \mathcal{N}(\mB_\rho^m - \tlambdalmrho I)^2 \right\} .
\end{equation}
At $\rho = \rho_n^m$, the algebraic multiplicity of $\tlambdalmrho$ in $\mB_\rho^m |_{E_{\tlambdalmrho}}$ is 2, whereas the geometric multiplicity is 1.
\end{itemize}
\end{remark}

\subsubsection{Time integration with Gamma-distributed traveling times}
The kernel for the time-integrated process with exponentially distributed traveling times has a singularity at the origin $(\by,\bn) = (\mathbf{0},\be_z)$. It is possible to derive a relation between the resolvent kernel and the process with Gamma-distributed traveling times, that does not suffer from this singularity. Thanks to the exact solution representation of the kernels for both processes, we obtain the following refinement and generalization of \cite[Thm. 12]{duits_morphological_2012}.

\begin{theorem}\label{thm:gammadist}
The kernel of the time-integrated diffusion $(i=1)$ and convection-diffusion $(i=2)$ process with the assumption of $\Gamma$-distributed evolution time $T$, i.e., $T \sim \Gamma(k,\alpha)$, $k \in \mathbb{N}$, with $E(T)= \frac{k}{\alpha}$, is related to repeated convolutions of the resolvent kernel as follows:

\begin{equation}
\int \limits_{0}^{\infty} K_t^i \; \Gamma(t; k, \alpha) \rmd t = R^i_\alpha \ast_{\quot}^{(k-1)} R^i_\alpha, 
\end{equation}
where $\Gamma(t; k, \alpha)$ is the pdf of $T$. For the case $i = 1$, this kernel does not have a singularity in the origin when $k \geq 2$. For the case $i = 2$, this holds when $k \geq 4$.
\end{theorem}

\begin{proof}
We refer to \ref{app:proof} for the proof.
\end{proof}

% \begin{theorem}
% The solution of the time-integrated direction process with Gamma-distributed traveling times is related to repeated convolutions of the resolvent kernel as follows:

% \begin{equation}
% R_\alpha^2 \ast_{\quot}^{(k-1)} R_\alpha^2 = \int \limits_{0}^{\infty} K_t^2 \; \Gamma(t; k, \alpha) \rmd t, 
% \end{equation}
% with $\Gamma(t;k,\alpha)$ the Gamma distribution with $E(T) = k/\alpha$. For $k \geq 4$ this convolution of resolvent kernels does not have a singularity in the origin.
% \end{theorem}

% \begin{proof}
% See \cite[Thm. 12]{duits_morphological_2012}.
% \end{proof}

When comparing kernels for varying $k$, while keeping the expected value of the traveling time $T$ fixed, using a $k$ higher than the bounds given in the theorem above gives better shaped kernels in practice, with more outward mass and dampened singularities. We use this idea in the visualization of a time-integrated contour completion kernel in Section \ref{se:Fourierr3s2}, Fig. \ref{fig:exactCCkernel}.

% The approach we follow here requires a transformation of the form:

% \begin{equation}
% .
% \end{equation}

% A straightforward series expansion around $x_0$ does:

% \begin{equation}
% GS_{m,n}(x) = \sum_{i=0}^\infty b_n
% \end{equation}
\section{Matrix Representation of the Evolution and Resolvent in a Fourier Basis}\label{se:Fourierr3s2}
% Application of the Fourier Transform on \texorpdfstring{$\quot$}{R3xS2}}\label{se:DSFT}
% \section{Application of the Fourier Transform on \texorpdfstring{$\quot$}{R3xS2}}\label{se:DSFT}

In \cite{duits_explicit_2008} a connection between the exact solutions and a numerical algorithm proposed by August and Zucker in \cite{august_curve_2001,august_sketches_2003} was established for the $SE(2) = \mathbb{R}^2 \rtimes S^1$ case. In this algorithm the Fourier transform on $\mathbb{L}_2(\mathbb{R}^2)$ and $\mathbb{L}_2(S^1)$ are applied subsequently. Here we again establish such a connection between the exact solutions and such a numerical algorithm in the 3D case. 

We have seen in the previous sections that the eigenfunctions $\Philmom$ and $\Psilmom$ are closely related to the spherical harmonic functions. Using the Fourier transform on $\mathbb{L}_2(S^2)$, we naturally obtain an ordinary differential equation in terms of spherical harmonic coefficients. In Sections \ref{se:diffusionnumerics} and \ref{se:convdiffusionnumerics} we derive this ODE for diffusion and convection-diffusion, respectively. At the end of Section \ref{se:diffusionnumerics} we include two remarks: one that shows that in deriving the ODEs we encounter a matrix representation of the evolution operator $\mB_\bomega^i$ and its resolvent, and one that makes the connection with the Fourier transform on SE(3) as presented in \ref{ap:Fourierse3}. Finally, we show in Section \ref{se:implementation} that with the procedure presented in this section it is straightforward to compute a numerical solution to the PDEs, since it only requires truncation of the order of the spherical harmonics.

% In \cite{duits_explicit_2008} a connection between the exact solutions and the August algorithm \cite{august_curve_2001} was established for the $SE(2) = \mathbb{R}^2 \rtimes S^1$ case. In this algorithm, the Fourier transform on $\mathbb{L}_2(\mathbb{R}^2 \times S^1)$ is applied, where $\mathbb{R}^2 \times S^1$ is the normal Cartesian product. This means that first the Fourier transform on $\mathbb{L}_2(\mathbb{R}^2)$ is applied, and subsequently the Fourier transform on $\mathbb{L}_2(S^1)$. Here we follow the same approach: we first apply a Fourier transform on $\mathbb{L}_2(\mathbb{R}^3)$, followed by a Fourier transform on $\mathbb{L}_2(S^2)$. This allows us to directly compute a discrete version of the exact solutions derived above. 

\subsection{Hypo-Elliptic Diffusion}\label{se:diffusionnumerics}
Recall that after a Fourier transform in the spatial coordinates, we have the following system:

\begin{equation}\label{eq:augusteqCE}
\left\{ \begin{aligned}
\frac{\partial}{\partial t} \hW(\bomega,\bn,t)&= (D_{44} \Delta_{S^2} - D_{33}(\bomega \cdot \bn)^2) \; \hW(\bomega,\bn,t), \qquad t \geq 0,\\
\hW(\bomega,\bn,0) &= \hU(\bomega,\bn).
\end{aligned} \right.
\end{equation}
Since the spherical harmonics form a basis for $\mathbb{L}_2(S^2)$, we can expand $\hW(\bomega,\bn,t)$ for fixed $\bomega$ and $t$ in the spherical harmonic basis. Instead of doing this with standard spherical harmonics, we use a specific type of reoriented spherical harmonics, that we define as follows:

\begin{equation}\label{def:reorientedSH}
Y_{\bomega}^{l,m}(\bn) := Y^{l,m}_{\mathbf{0}}(\bR^T_{\bomega r^{-1}} \bn) := Y^{l,m}(\beta,\gamma)= \frac{\varepsilon^m}{\sqrt{2\pi}} P_l^m(\cos \beta) e^{i m \gamma}, \text{ with } \bn = \bn^{\bomega} (\beta,\gamma), 
\end{equation}
and $m \in \mathbb{Z}$, $l \in \mathbb{N}_0$, $|m| \leq l$. Recall Fig. \ref{fig:figureParametrization} and Eq. (\ref{def:coordinatechoice}). The rotation $\bR_{r^{-1} \bomega }$ in (\ref{def:reorientedSH}) is defined through its matrix

\begin{equation}
 \bR_{r^{-1} \bomega } = \begin{pmatrix}
\frac{(\bomega \times \be_z) \times \bomega}{||(\bomega \times \be_z) \times \bomega||} & \vline & \frac{\bomega \times \be_z}{||\bomega \times \be_z||} & \vline & r^{-1}\bomega
\end{pmatrix}.
\end{equation}
This rotation maps $\be_z$ onto $r^{-1} \bomega $ and $\be_y$ onto $\bomega \times \be_z/||\bomega \times \be_z||$, such that for every $\beta,\gamma$ we have that

\begin{equation}
\begin{aligned}
\bn^{\bomega} (\beta,\gamma) &= \bR_{ r^{-1} \bomega } \bn^{\be_z}(\beta,\gamma)\\
\bR_{r^{-1} \bomega }^T \bR_{\bomega r^{-1},\gamma} \bR_{\frac{\bomega \times \be_z}{||\bomega \times \be_z||}, \beta} \bR_{r^{-1} \bomega } \; \be_z &= \bR_{\be_z,\gamma} \bR_{\be_y,\beta} \be_z.
\end{aligned}
\end{equation}
Now for fixed $\bomega$ and $t$ we develop $\hW(\bomega,\cdot,t)$ in the basis $\{ \ylmom \,|\,l \in \mathbb{N}_0, m \in \mathbb{Z}, |m| \leq l \}$ as follows:

\begin{equation}\label{eq:hWseries}
\hW(\bomega,\bn,t) = \sumlm \hW^{l,m}(\bomega,t) \, \ylmom (\bn),
\end{equation}
where we define $\hW^{l,m}(\bomega,t)$ (and similarly $\hU^{l,m}(\bomega)$):

\begin{equation}\label{def:hWlm}
\hW^{l,m}(\bomega,t) := \int_{S^2} \hW(\bomega,\bn,t) \overline{\ylmom(\bn)} \, \rmd \sigma(\bn).
\end{equation}
Our goal is the recursion in Eq. (\ref{eq:hWlm}) for $\hW^{l,m}(\bomega,t)$, that we can use to obtain a solution for these coefficients. To this end we start by substituting (\ref{def:hWlm}) into our differential equation (\ref{eq:augusteqCE}):

\begin{equation}\label{eq:intermediateWlm}
\begin{aligned}
\sumlm \ylmom(\bn) \, \partial_t \hW^{l,m}(\bomega,t) &= \sumlm (D_{44} \Delta_{S^2} - D_{33}r^2 \cos^2 \beta) \ylmom(\bn) \, \hW^{l,m}(\bomega,t) \\
&= \sumlm (-D_{44} l(l+1) - D_{33} r^2 \cos^2 \beta) \ylmom(\bn) \, \hW^{l,m} (\bomega,t).
\end{aligned}
\end{equation}
We aim at rewriting the term $\cos^2 \beta \, Y^{l,m}(\beta,\gamma)$, see (\ref{eq:cos2betaYlm}) below. In \cite{olver_nist_2010} the following identity for Legendre functions is given: 

% \begin{equation}\label{eq:relationSH}
% (l - m + 1)P_{l +1}^{m}(x) - (2l +1) x P_{l}^{m}(x) + (l + m)P_{l - 1}^m(x) = 0, \qquad m \geq 0. 
% \end{equation}

\begin{equation}\label{eq:relationSH}
 x P_{l}^{m}(x)  = \frac{l - m + 1}{2l +1}P_{l +1}^{m}(x) + \frac{l + m}{2l+1} P_{l - 1}^m(x) =: \xi^{l,m}P_{l+1}^{m}(x) + \nu^{l,m}P_{l-1}^m(x), \qquad m \geq 0. 
\end{equation}
Using this identity twice, we get:

\begin{equation}
\begin{aligned}
x^2 P_{l}^{m}(x) &= \frac{(l-m+1)(l-m+2)}{(2l+3)(2l+1)} P_{l+2}^{m}(x) + \frac{(2l(l+1)-2m^2 -1)}{4l(l+1)-3}P_{l}^{m}(x) + \frac{(l+m-1)(l+m)}{(2l-1)(2l+1)}P_{l-2}^{m}(x) \\
&=: \zeta^{l,m} P_{l+2}^{m}(x) + \eta^{l,m} P_{l}^{m}(x) + \alpha^{l,m} P_{l-2}^m(x), \qquad m \geq 0.
\end{aligned}
\end{equation}
However, this identity only holds for the Legendre functions where the normalization factor $N^{l,m}$ is not included in the definition. The identity needs to be adapted, which is done as follows:

\begin{equation}\label{eq:relationPlm}
x^2  \; P_l^m(x) = \frac{N_{l,m}}{N_{l+2,m}}\zeta^{l,m} P_{l+2}^m(x) + \eta^{l,m} P_l^m(x) + \frac{N_{l,m}}{N_{l-2,m}} \alpha^{l,m} P_{l-2}^m(x). 
\end{equation}

This can directly be applied to the term $\cos^2 \beta Y^{l,m}(\beta,\gamma)$. We define the tridiagonal matrix $\bM_1^m$ as follows:

\begin{equation}\label{eq:definitionM1}
\begin{aligned}
&\left((\bM_1^m)^T \right)_{l,l'= |m|,|m+1|,\dots} := \\
\bN^m &\begin{pmatrix}
\eta^{|m|,|m|} & 0 & \zeta^{|m|,|m|} &  &  &  &   \\
0 & \eta^{|m|+1,|m|} & 0 &  \zeta^{|m|+1,|m|} &  & \text{\textbf{\large{O}}} &  \\
\alpha^{|m|+2,|m|} & 0 & \eta^{|m|+2,|m|} & 0 & \zeta^{|m|+2,|m|} &  &   \\
  & \alpha^{|m|+3,|m|} & 0 & \eta^{|m|+3,|m|} & 0 & \zeta^{|m|+3,|m|} &   \\
 & \text{\textbf{\large{O}}} & \ddots & 0 & \ddots & 0 & \ddots 
\end{pmatrix} (\bN^m)^{-1},
\end{aligned}
\end{equation}
with 

\begin{equation}\label{eq:normalizationMatrix}
\bN^m = \text{diag}(N^{|m|,m}, N^{|m+1|,m, \dots}).
\end{equation} 
Then we can write the relation for spherical harmonics in the form:

\begin{equation}\label{eq:cos2betaYlm}
\cos^2 \beta \; \Ylm (\beta,\gamma) = \sum_{l'=|m|}^\infty ((\bM_1^m)^T)_{l l'} \Ylpm(\beta,\gamma).
\end{equation}
% \begin{equation}
% M^m = \begin{pmatrix}
% b^{m,m} & 0 & c^{m,m} & 0 & 0 & 0 & 0 & \dots \\
% 0 & b^{m+1,m} & 0 & c^{m+1,m} & 0 & \mathbf{O} & 0 & \dots \\
% a^{m+2,m} & 0 & b^{m+2,m} & 0 & c^{m+2,m} & 0 & 0 & \dots \\
% 0  & a^{m+3,m} & 0 & b^{m+3,m} & 0 & c^{m+3,m} & 0 & \dots \\
% 0 & 0 & \ddots & 0 & \ddots & 0 & \ddots & 0  
% \end{pmatrix}
% \end{equation}
Because matrix $\bM_1^m$ has only three nonzero diagonals, at most three terms appear in the sum on the right. Since we consider these matrices for fixed $m$, it is more natural to change the order of summation in Eq. (\ref{eq:intermediateWlm}). With (\ref{eq:cos2betaYlm}) we can rewrite (\ref{eq:intermediateWlm}) as:

\begin{equation}\label{eq:diffCEcoeff}
\begin{aligned}
\summl  \ylmom(\bn) \partial_t \hW^{l,m}(\bomega,t) =& \summl - D_{44} l(l+1) \ylmom(\bn) \hW^{l,m}(\bomega,t) \\ &- D_{33} r^2 \summl \left(\sum_{l'=|m|}^\infty ((\bM_1^m)^T)_{l l'} Y_{\bomega}^{l',m}(\bn) \hW ^{l,m}(\bomega,t) \right).
\end{aligned}
\end{equation}
Equating the coefficients in this equation, the functions $\hW^{l,m}$ can be expressed recursively according to:

\begin{equation}\label{eq:hWlm}
\begin{aligned}
\frac{\partial}{\partial t}\hW^{l,m} (\bomega,t) &= -D_{44} l(l+1)\hW^{l,m}(\bomega,t) \\ & \qquad - D_{33} r^2\left((\bM_1^{m})^T_{l-2,l} \hW^{l-2,m}(\bomega,t) + (\bM_1^{m})^T_{l,l} \hW^{l,m}(\bomega,t) + (\bM_1^{m})^T_{l+2,l} \hW^{l+2,m}(\bomega,t) \right) \\
&= - D_{44} l(l+1)\hW^{l,m}(\bomega,t) \\ & \qquad - D_{33} r^2\left(\zeta^{l-2,m} \hW^{l-2,m}(\bomega,t) + \eta^{l,m} \hW^{l,m}(\bomega,t) + \alpha^{l+2,m} \hW^{l+2,m}(\bomega,t) \right)
\end{aligned}
\end{equation}
%\\&= -\tl(\tl+1)\hW^{\tl,\tm} +\\ &- \rho^2 \left( \frac{(\tl-\tm-1)(\tl-\tm)}{(2\tl-1)(2\tl-3)} \hW^{\tl-2,\tm} + \frac{(2\tl(\tl+1)-2\tm^2 -1)}{(2\tl -1)(2\tl+3)} \hW^{\tl,\tm} + \frac{(\tl+\tm+1)(\tl+\tm+2)}{(2\tl+3)(2\tl+5)} \hW^{\tl+2,\tm} \right).
In other words, for $m$ fixed, we need to solve the following system:

\begin{equation}\label{eq:matrixeqWm}
\left\{\begin{aligned}
\frac{\partial}{\partial t}\bhw^m(\bomega,t) &= - D_{33} r^2 \bM_1^m \bhw^m(\bomega,t) - D_{44} \mathbf{\Lambda}^m \bhw^m(\bomega,t), \qquad \bhw^m(\bomega,t) = (\hW^{l,m}(\bomega,t))_{l=|m|}^{\infty}, \\
\bhw^m(\bomega,0) &= \hat{\mathbf{u}}^m(\bomega), \qquad \hat{\mathbf{u}}^m(\bomega,t) = (\hU^{l,m}(\bomega,t))_{l=|m|}^{\infty}
\end{aligned}\right. 
\end{equation}
with $\mathbf{\Lambda}^m = \text{diag}(|m|(|m|+1), (|m|+1)(|m|+2), \dots)$.
The solution of this system, being a matrix-vector representation of (\ref{eq:augusteqCE}), is given by

\begin{equation}\label{eq:whatexp}
\bhw^m(\bomega,t) = \exp \{( -D_{33} r^2 \bM_1^m - D_{44} \mathbf{\Lambda}^m)t \} \hat{\mathbf{u}}^m(\bomega).
\end{equation}

\begin{remark}
There is a direct connection between the exact solution presented in Theorem~\ref{th:maintheoremCE}/Corollary~\ref{corr:relationSHSWE} and the solution found via Eq. (\ref{eq:whatexp}). In fact, in case of the exact solutions, the generator and thereby also the evolution operator in (\ref{eq:whatexp}), are diagonalized by the solutions of Eq.~\!(\ref{eq:finaleigvproblem}) in \ref{ap:SWF}. To clarify this observation, We note that Eq.~\!(\ref{eq:finaleigvproblem}) is the eigenvector problem corresponding to the eigenfunction problem in Eq.~\!(\ref{eq:eigenfunctions}), while restricting operator $\mB_1^{\bomega}$ to the span of spherical harmonics $Y^{l,m}$, $l \geq |m|$, for $m \in \mathbb{Z}$ fixed.
\end{remark}

\begin{remark}
Eq. (\ref{eq:hWlm}) follows from Eq. (\ref{eq:CEdiff}) by application of the operator:

\begin{equation}
(\mFS \otimes \mathbbm{1}_{\mathbb{L}_2(\mathbb{R}^3})) \circ (\mathbbm{1}_{\mathbb{L}_2(\mathbb{R}^3)} \otimes \mathcal{U}_{\bR_{\bomega r^{-1}}}) \circ (\mFR \otimes \mathbbm{1}_{\mathbb{L}_2(S^2})),
\end{equation}
that can be roughly formulated as $\mFS \circ \mathcal{U}_{\bR_{\bomega r^{-1}}} \circ \mFR$, with $\mathcal{U}_{\bR_{\bomega r^{-1}}}$ the left-regular representation $\mathcal{U}_{\bR} \phi(\bn) = \phi(\bR^T \bn)$. There is a close connection between this operator and the Fourier transform and irreducible representations on SE(3), see \ref{ap:Fourierse3}.
\end{remark}

% \begin{remark}
% Eq. (\ref{eq:hWlm}) follows from Eq. (\ref{eq:CEdiff}) by application of the operator:

% \begin{equation}
% (\mFS \otimes \mathbbm{1}_{\mathbb{R}^3}) \circ (\mathbbm{1}_{\mathbb{R}^3} \otimes \mathcal{U}_{\bR_{\bomega}}) \circ (\mFR \otimes \mathbbm{1}_{S^2}),
% \end{equation}
% or roughly formulated by $\mFS \circ \mathcal{U}_{\bR_{\bomega}} \circ \mFR$, with $\mathcal{U}_{\bR_{\bomega}}$ the left-regular representation $\mathcal{U}_{\bR} \phi(\bn) = \phi(\bR^T \bn)$. One may also consider the application of the Fourier transform $\mF_{SE(3)}$ to Eq. (\ref{eq:CEdiff}). For this, one can consider the irreducible representations on SE(3):

% \begin{equation}
% U^{\rho,s} : SE(3) \rightarrow B(\mathbb{L}_2(S^2)), \qquad g \mapsto U_g^{\rho,s}.
% \end{equation}
% Here the bounded operator $U_g^{\rho,s}$ acts as:

% \begin{equation}
% (U_g^{\rho,s} \phi)(N) = e^{-i N \cdot \bx} \; \phi(\bR^{-1}N) \; \chi_s (\bR_N^{-1}\bR \bR_N), \quad g = (\bx, \bR).
% \end{equation}  
% With the substitution $\rho = ||\bx||$ and $N = \rho \bR_{\rho^{-1} \bx} \bn$ and the Fourier transform on SE(3), a system can be obtained in terms of irreducible representations, that eventually leads to an equation identical to (\ref{eq:whatexp}).

% \end{remark}

\subsection{Resolvent of the Convection-Diffusion}\label{se:convdiffusionnumerics}
The same idea of substituting directly a series of spherical harmonics into the equation (as done in the previous section for the pure diffusion case) can be used for the resolvent case of the convection-diffusion equation. After the Fourier transform in the spatial coordinates, this equation reads

\begin{equation}\label{eq:augusteqCC}
(\alpha I - \mB^2_\bomega) \hW(\bomega,\bn) = (\alpha I - (D_{44} \Delta_{S^2} - i (\bomega \cdot \bn)) \hW(\bomega,\bn) = \alpha \hat{U}(\bomega,\bn).
\end{equation}
Again, we express $\hW$ in terms of spherical harmonics as

\begin{equation}
\hW(\bomega, \bn) = \sumlm \hW^{l,m}(\bomega) \ylmom(\bn).
\end{equation}
Substituting this into Eq. (\ref{eq:augusteqCC}), we obtain:

\begin{equation}\label{eq:augusteqCC2}
\begin{aligned}
\alpha \sumlm \hW^{l,m}(\bomega)\ylmom(\bn) \; &+  \\+  \sumlm & D_{44} l(l+1) + i r \cos \beta \hW^{l,m}(\bomega) \ylmom(\bn) &= \alpha \sumlm \hU^{l,m}(\bomega) \ylmom(\bn).
\end{aligned}
\end{equation}
We use the relation in Eq. (\ref{eq:relationSH}) once, and we get

\begin{equation}
\cos \beta \; Y^{l,m}(\beta,\gamma) = \sum_{l' = |m|}^\infty ((\bM^m_2)^T)_{l l'} Y^{l',m}(\beta,\gamma),
\end{equation}
where matrix $\bM_2^m$ only has non-zero elements on the upper and lower diagonal:

\begin{equation}\label{eq:matrixM2}
(\bM_2^m)^T := \bN^m \begin{pmatrix}
0 & \xi^{|m|,|m|} &  & \text{\textbf{\large{O}}} &  \\
\nu^{|m|+1,|m|} & 0 & \xi^{|m|+1,|m|} &  &  \\
 & \nu^{|m|+2,|m|} & 0 & \xi^{|m|+2,|m|} & \\
 \text{\textbf{\large{O}}} & & \ddots &  & \ddots 
\end{pmatrix} (\bN^m)^{-1},
\end{equation}
with $\xi^{l,m}$ and $\nu^{l,m}$ as in (\ref{eq:relationSH}) and $\bN^m$ as in (\ref{eq:normalizationMatrix}). Using this matrix and proceeding analogously to the diffusion case, the solution to Eq. (\ref{eq:augusteqCC2}) is found by solving:

% \begin{equation}
% (\alpha + D_{44} l(l+1) )\hW^{l,m}(\bomega) + i r \left(\frac{l-m}{2l-1} \hW^{l-1,m}(\bomega) + \frac{m+l+1}{2l+3}\hW^{l+1,m}(\bomega)\right) = \alpha \hU^{l,m}(\bomega).
% \end{equation}
% Or in matrix form:

\begin{equation}\label{eq:matrixeqWmCC}
(\alpha I + D_{44} \mathbf{\Lambda}^m  + i r \bM_2^m) \bhw^m(\bomega) = \alpha \mathbf{\hat{u}}^m(\bomega), \qquad \bhw^m(\bomega) = (\hW^{l,m}(\bomega))_{l=|m|}^{\infty}.
\end{equation}

\subsection{Numerical Implementation of the Discrete Spherical Transform}\label{se:implementation}
When using the above procedure to compute the kernel, there are two places where the numerics differ from the exact solution: the series of spherical harmonics is truncated and the Fourier transform is carried out discretely. We introduce a parameter $l_{max}$ to indicate the maximal order of spherical harmonics that is taken into account. Furthermore, we take discrete values for $\bomega$ on an equidistant cubic grid, say $\bomega_{ijk}$, such that for each $\bomega_{ijk} \in \mathbb{R}^3$ the component $\bhw^{m}$ of the solution requires, for the pure-diffusion case, solving the ODE:

% , Eq. (\ref{eq:diffCEcoeff}) becomes

% \begin{equation}
% \sum_{l=0}^{l_{max}} \sum_{m=-l}^l  \ylmom(\bn) \partial_t \hW^{l,m}(\bomega,t) = \sum_{l=0}^{l_{max}} \sum_{m=-l}^l -l(l+1) \ylmom(\bn) \hW^{l,m}(\bomega,t) - \sum_{l=0}^{l_{max}} \sum_{m=-l}^l (\rho^2 \sum_{l'} M_{l l'}^m Y_{\bomega}^{l',m} \hW ^{l,m}(\bomega,t)).
% \end{equation}

\begin{equation}\label{eq:numericaldifferentialeq}
\left\{ \begin{aligned}
&\partial_t \bhw^m(\bomega_{ijk},t) = -D_{33} r^2 \bM^m_{1,l_{max}} \bhw^m(\bomega_{ijk},t) - D_{44} \mathbf{\Lambda}^m_{l_{max}} \bhw^m(\bomega_{ijk},t), \qquad \bhw^m = (\hW^{l,m})_{l=|m|}^{l_{max}},\\
&\bhw^m(\bomega_{ijk},0) = \hat{\mathbf{u}}^m(\bomega_{ijk}) := \sum_{l=0}^{l_{max}} Y^{l,m}_{\bomega_{ijk}}(\be_z) Y^{l,m}_{\bomega_{ijk}} (\bn)
\end{aligned}\right. ,
\end{equation}
where $\bM^m_{1,l_{max}}, \mathbf{\Lambda}^m_{l_{max}} \in \mathbb{R}^{(l_{max}-|m|+1) \times (l_{max}-|m|+1)}$. In the convection-diffusion, resolvent case, it comes down to solving:

\begin{equation}
(\alpha I + D_{44} \mathbf{\Lambda}_{l_{max}}^m  + i r \bM_{2,l_{max}}^m) \bhw^m(\bomega) = \alpha \mathbf{\hat{u}}^m(\bomega), \qquad \bhw^m(\bomega,t) = (\hW^{l,m}(\bomega,t))_{l=|m|}^{l_{max}},
\end{equation}
for all $\bomega = \bomega_{ijk}$ on an equidistant grid:

\begin{equation}
\bomega_{ijk} = \left(\frac{i \eta \pi}{N}, \frac{j \eta \pi}{N}, \frac{k \eta \pi}{N} \right), \qquad i,j,k \in \{-N, \dots, N \}, \; \eta \in \mathbb{N},
\end{equation}
where $N$ denotes the number of samples. We use a discrete centered inverse Fourier transform to go back to the $\quot$ domain. 

\begin{remark}
In our implementation of the algorithm for Eq. (\ref{eq:numericaldifferentialeq}), we work with the basis $Y_{\mathbf{0}}^{l,m} = Y^{l,m}$. We use Wigner D-matrices to bring the coefficients with respect this basis to coefficients with respect to the rotated spherical harmonics $Y_\bomega^{l,m}$ and vice versa. This is done for all $m$ simultaneously to transform the coefficients of the initial condition. Then Eq. (\ref{eq:whatexp}) is solved for each $m$, and the solution is transformed back, again for all $m$ at once.
\end{remark}

The result for the diffusion kernel $K_t^1(\bomega,\bn)$ with $t = 2$, $D_{33} = 1$, $D_{44} = 0.1$, $\eta = 8$, $N = 65$ and $l_{max} = 12$ (resulting in 169 spherical harmonic coefficients) is shown in Fig. \ref{fig:exactCEkernels}. The convection-diffusion kernel $K_t^2(\bomega, \bn)$ for $k = 1$, $D_{44} = 0.5$, $\eta = 4$, $N = 65$, $l_{max} =12$ is given on the left in Fig. \ref{fig:exactCCkernel}. Numerically integrating these kernels for different $t$, using a $\Gamma$-distribution with $k = 4$ and $\alpha = 0.25$ gives the result as displayed on the right in Fig. \ref{fig:exactCCkernel}.

\begin{figure}[t!]
   \centering
   \includegraphics[width = 0.9\textwidth]{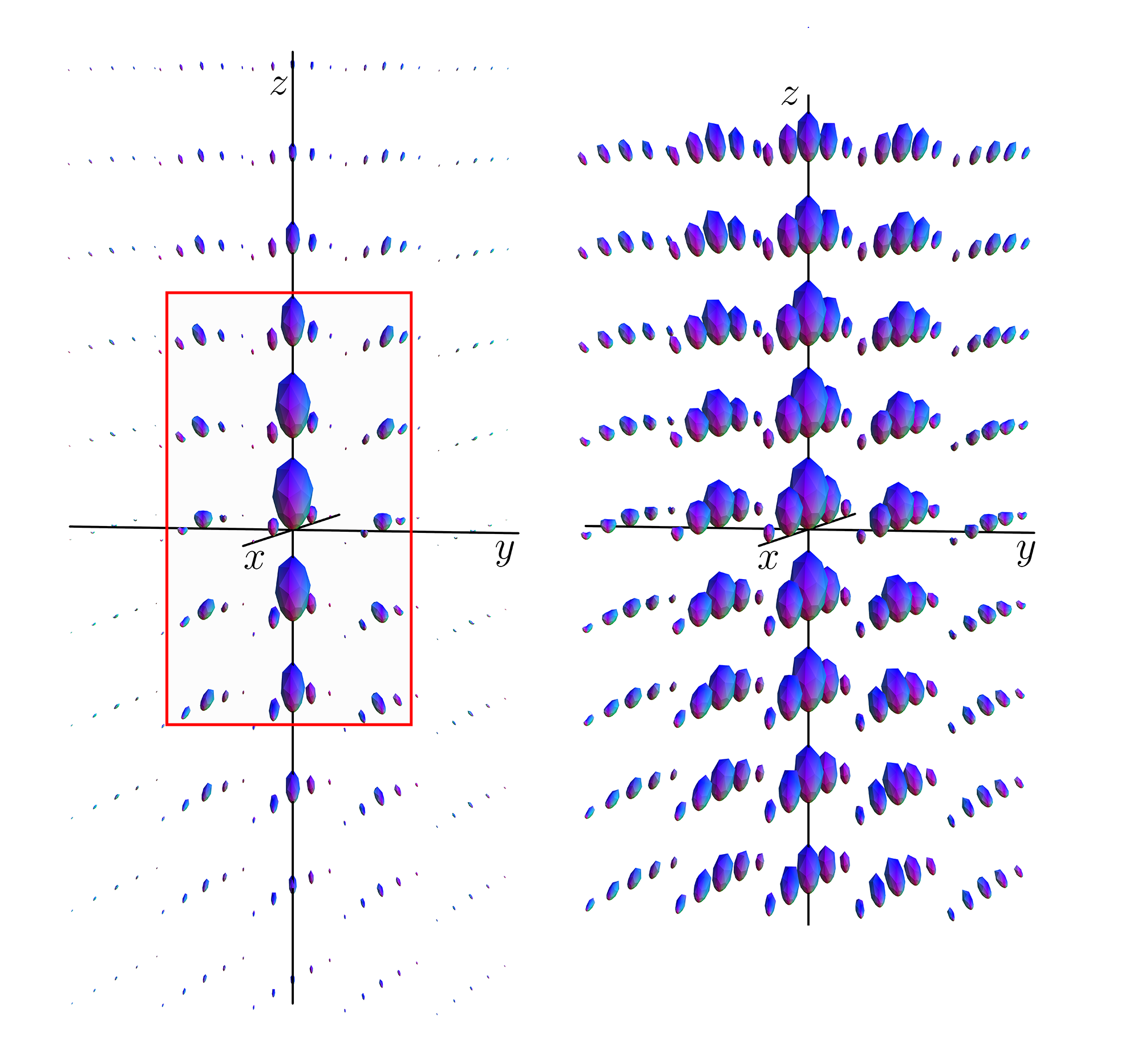}
 \caption{Glyph field visualization (as explained in Section \ref{se:stopro}) of the kernel $K^1_{t=2}(\bomega,\bn)$, with a higher resolution on the right. For this kernel, $D_{33} = 1$, $D_{44} = 0.1$.}\label{fig:exactCEkernels}
 \end{figure}

 % \begin{figure}[t!]
 %   \centering
 %   \includegraphics[width = 0.9\textwidth]{Figures/kernelCC_Fourierrealim.png}
 % \caption{Glyph field visualization of the real (left) and imaginary (right) part of the kernels $\hat{R}^2_{\alpha = 0.5}$ on the central part of an equidistant grid in the Fourier domain. Here $D_{44} = 0.1$.}\label{fig:exactCCkernelFourier}
 % \end{figure}

  \begin{figure}[t!]
   \centering
   \includegraphics[width = 0.9\textwidth]{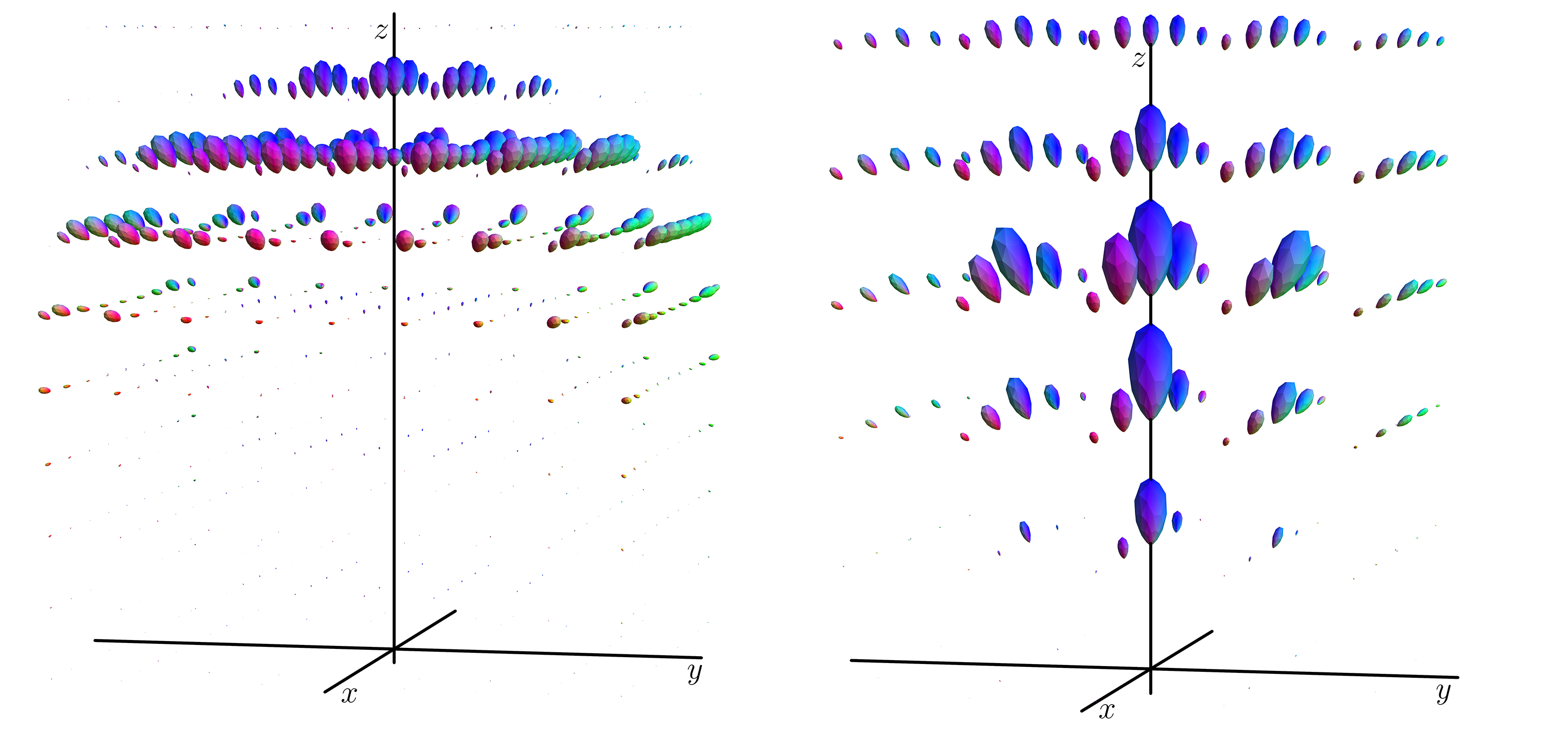}
 \caption{Glyph field visualization of the time dependent kernel $K_{t=1}^2$ for the convection-diffusion case, with $D_{44} = 0.5$ (left) and the time-integrated kernel, where we used a Gamma-distribution with $k = 4$ and $\alpha = .25$. }\label{fig:exactCCkernel}
 \end{figure}

\section{Analytical Approximations Using Logarithm on SE(3)}\label{se:approximationkernel}
For performing the enhancement of an image $U : \mathbb{R}^3 \rtimes S^2$ via the PDE framework as before, it is beneficial to have a good approximation of the convolution kernel. We have derived the exact solutions and related them to Fourier based algorithms, but they are not feasible for fast computation of the enhancement of the image. Numerical approximations of the kernels exist \cite{duits_left-invariant_2010}, for which improvements have been proposed in our conference paper \cite{portegies_new_2015}. A fast implementation for neuroimaging purposes was made by \cite{meesters_fast_2016}. In the following lemma we show three properties of the solution operator and the kernel of the pure diffusion case on the group $SE(3)$, from which we derive later a number of properties for the kernel on the group quotient $\quot$.

\begin{lemma}\label{le:lemma}
Let $\tUps:\mathbb{L}_2^R(SE(3)) \rightarrow \mathbb{L}_2^R(SE(3))$ be a linear and legal kernel operator, and assume it maps $\mathbb{L}_2(SE(3))$ into $\mathbb{L}_{\infty}(SE(3))$. Then we have:
\begin{enumerate}
\item identity: $(\tUps (\tU))(g) = \int \limits_{SE(3)} \tilde{k}(g,q)\,\tU(q)\, {\rm d} q = (\check{\tK} \ast_{SE(3)} \tU)(g),
\textrm{ with } \\ \tilde{k}(g,q) = \tilde{k}(e,(g^{-1}q)) =:\check{\tK}(q^{-1}g)$ and with $\tilde{K}(g):=\tilde{k}(g,e)=\check{\tilde{K}}(g^{-1})$.\mbox{}\vspace{0.2cm}\mbox{}
\item symmetry: $
\tilde{k}(gh, qh') = \tilde{k}(g,q)\textrm{ for all } h, h' \in H$ and all $g \in SE(3)$.
%\end{equation}
\item preservation of correspondences:
$(U \leftrightarrow \tU) \Rightarrow \Upsilon(U) \leftrightarrow \tUps(\tU)$ with
\begin{equation}\label{int}
\begin{array}{lll}
(\Upsilon (U))(\by,\bn) &= \int \limits_{\mathbb{R}^3 \times S^2} k(\by,\bn,\by',\bn') U(\by',\bn') {\rm d} \by'{\rm d} \sigma(\bn'), \qquad \textrm{ with } \\
k(\by,\bn,\by',\bn') &= \frac{1}{2\pi}\tilde{k}((\by,\bR_{\bn}),(\by',\bR_{\bn'})) \\
 &= \frac{1}{2\pi} \tilde{K} ((\by',\bR_{\bn}')^{-1}(\by,\bR_{\bn})) = K(\bR_{\bn'}^{-1}(\by-\by'),\bR_{\bn'}^{-1}\bn).
\end{array}
\end{equation}
\end{enumerate}
\end{lemma}
The proof of this lemma was given in \cite{portegies_new_2015}. Now we consider the special case where $K$ represents the solution kernel $K_t^1$ of the diffusion process with generator $Q_1$.

\begin{corollary}\label{corr:1}
The exact kernels $\tilde{K}^1_{t}:SE(3) \to \mathbb{R}^{+}$ and $K^1_{t}:\mathbb{R}^{3}\rtimes S^{2} \to \mathbb{R}^+$ satisfy the following symmetries
\begin{equation}\label{eq:symm1}
K^1_{t}(\ul{y},\ul{n})= K^1_{t}(\bR_{\ul{e}_{z},\alpha}\ul{y},\bR_{\ul{e}_{z},\alpha}\ul{n}) \textrm{ and } \tilde{K}^1_{t}(h^{-1}g)=\tilde{K}^1_{t}(g)=\tilde{K}^1_{t}(gh')
\end{equation}
for all $t>0$, $\alpha>0$, $(\ul{y},\ul{n}) \in \mathbb{R}^{3}\rtimes S^{2}$, $g \in SE(3)$, $h,h' \in H$. Moreover,
\begin{equation}\label{eq:symm2}
 K^1_t(\by,\bn) = K^1_t(-\bR_{\bn}^T \by, \bR_{\bn}^T \be_z) \textrm{ and } \tK^1_t(g) = \tK^1_t(g^{-1}).
\end{equation}
\end{corollary}
\begin{proof}
The symmetry (\ref{eq:symm1}) is due to Lemma \ref{le:lemma}. The second symmetry (\ref{eq:symm2}) is due to reflectional invariance $(\mA_3,\mA_4,\mA_5) \mapsto (-\mA_3,-\mA_4,-\mA_5)$ in the diffusion (\ref{eq:CEdiff}) and reflection on the Lie algebra corresponds to inversion on the group.
\end{proof}
\begin{remark}\label{remark:symmetry}
Note that (\ref{eq:symm2}) in terms of $k$ would be: $k(\by,\bn,\mathbf{0},\be_z) = k(\mathbf{0},\be_z,\by,\bn)$, by the relation $k(\by,\bn,\by',\bn') = K^1_t(\bR_{\bn'}^{T}(\by-\by'),\bR_{\bn'}^{T}\bn)$.
\end{remark}
Now the Gaussian kernel approximation, based on the logarithm on $SE(3)$ and the theory of coercive weighted operators on Lie groups \cite{ter_elst_weighted_1998}, presented in earlier work \cite[Ch:8,Thm.11]{duits_morphological_2012} does not satisfy this symmetry. Next we will improve it by a new practical analytic approximation which does satisfy the property.

\subsection{Comparison with Earlier Kernel Approximations for the Diffusion Kernel}\label{se:kernelapprox}
We first present two existing approximations for $K^1_t$ and $\tK^1_t$, that do not automatically carry over the properties we have shown in the previous section. A possible approximation kernel for $K^1_t$ is based on a direct product of two $SE(2)$-kernels, see \cite[Eq.(10), Eq.(11)]{rodrigues_accelerated_2010}, to which we refer as $K_t^{1,\text{App}\;1}$, with App being short for approximation. This approximation is easy to use since it is defined in terms of the Euler angles of the corresponding orientations. However, the symmetries described before are not preserved by $K_t^{1,\text{App}\;1}$ and errors tend to be larger when $D_{44}$ and $t$ increase. Therefore we move to an approximation for the $SE(3)$-kernel $\tK^1_t$, for which we show that it can be adapted such that it has the important symmetries. We need the logarithm on $SE(3)$ for this approximation, so first consider an exponential curve in $SE(3)$, given by $\tilde{\gamma}(t)= g_0 \exp \left(t \sum \limits_{i=1}^{6} c^{i}A_i\right)$. The logarithm $\log_{SE(3)}: SE(3) \to T_{e}(SE(3))$ is bijective, and it is given by
\begin{equation}\label{eq:log}
\log_{SE(3)}(g)=\log_{SE(3)}\left(\exp\left(\sum \limits_{i=1}^{6}c^i(g)A_i\right)\right) =\sum \limits_{i=1}^{6}c^i(g)A_i,
\end{equation}
and we can relate to this the vector of coefficients $\bc(g)=(\bc^{(1)}(g),\bc^{(2)}(g))=(c^{1}(g),\cdots, c^{6}(g))^{T}$.
We define matrix $\OmegaBF$ as follows:
\begin{equation}
\OmegaBF := \left(\begin{array}{ccc}
0 & -c^6 & c^5 \\
c^6 & 0 & -c^4 \\
-c^5 & c^4 & 0
\end{array}
\right), \qquad \bR = e^{t \OmegaBF}, \qquad \bR \in SO(3).
\end{equation}
We can write $\bR$ in terms of Euler angles, $\bR = \bR_{\be_z,\gamma} \bR_{\be_y,\beta} \bR_{\be_z,\alpha}$. Let the matrix $\OmegaBF_{\gamma,\beta,\alpha}$ be such that
$\exp(\OmegaBF_{\gamma,\beta,\alpha})=\bR_{\be_z,\gamma} \bR_{\be_y,\beta} \bR_{\be_z,\alpha}$. The spatial coefficients $\bc^{(1)} = \bc^{(1)}_{\bx,\gamma,\beta,\alpha}$ are given by the following equation:
\begin{equation}\label{eq:c1}
\bc^{(1)}  = \left(I - \frac12 \OmegaBF_{\gamma,\beta,\alpha} + q^{-2}_{\gamma,\beta,\alpha} \left(1 - \frac{q_{\gamma,\beta,\alpha}}{2} \cot \left(\frac{q_{\gamma,\beta,\alpha}}{2} \right) \right)(\OmegaBF_{\gamma,\beta,\alpha})^2 \right) \bx,
\end{equation}
where $q_{\gamma,\beta,\alpha}$ is the (Euclidean) norm of $\bc^{(2)}$. Then an approximation for the kernel $\tK^1_t$ is, up to a rescaling of time $t \mapsto c t$ with $c\geq 1$ due to an equivalence of two norms\footnote{The weighted logarithmic modulus and the (sub-Riemannian) distance in $SE(3)$, cf.~\!\cite{duits_sub-riemannian_2013}.} \cite[Prop.6.1]{ter_elst_weighted_1998}, given by \cite{duits_left-invariant_2010}:
\begin{equation}
  \tK^{1,\log}_t(\bc(g)) = \frac{1}{(4\pi t^2 D_{33} D_{44})^2} e^{-\frac{|\log_{SE(3)}(g)|^2_{D_{33},D_{44}}}{4t}},
\end{equation}
  with the smoothed variant of the weighted modulus, \cite[Eq.78,79]{duits_morphological_2012}, given by
\begin{equation}
    |\log_{SE(3)}(\cdot)|_{D_{33},D_{44}} = \sqrt[4]{\frac{1}{\xi}\, \frac{|c^1|^2+|c^2|^2}{D_{33}D_{44}}+\frac{|c^6|^2}{D_{44}}+\left(\frac{(c^3)^2}{D_{33}}+\frac{|c^4|^2 + |c^5|^2}{D_{44}}\right)^2},
\end{equation}
where we included a scaling with $\xi>0$ in the commutator term. In view of computations in \cite[ch.5.2]{duits_left-invariant_2010-2} we set $\xi=16$. First numerical investigations indicate this is needed to get Gaussian kernels close to the exact kernels (compare the result in Figure~\ref{fig:logarithmicapp} to the exact result in Figure~\ref{fig:exactCEkernels}).

Now the difficulty lies in the fact that the function $U$ is defined on the quotient for elements $(\ul{y}, \bR_{\bn})$, where the choice for $\alpha$ in the rotation matrix $\bR_{\bn}$ (mapping $\ul{e}_{z}$ onto $\ul{n} \in S^2$) is not of importance. However, the logarithm is only well-defined on the group $SE(3)$, not on the quotient $\mathbf{R}^{3}\rtimes S^{2}$, and explicitly depends on the choice of $\alpha$. It is therefore not straightforward to use this approximation kernel such that the invariance properties in Corollary~\ref{corr:1} are preserved. In view of Corollary~\ref{corr:1}, we need both left-invariance and right-invariance for $\tilde{K}^1_t(g)$ under the action of the subgroup $H$. As right-invariance is naturally included via $\tilde{K}^1_{t}(\ul{y},\bR)=K^1_{t}(\ul{y},\bR \ul{e}_{z})$, left-invariance is equivalent to inversion invariance. In previous work the choice of $\alpha=0$ is taken, giving rise to the approximation
\begin{equation}
K_t^{1,\mathrm{App}\,2}(\by,\bn(\beta,\gamma)) = \tK_t^{1,\log}(\bc(\by,\bR_{\be_z,\gamma}\bR_{\be_y,\beta})).
\end{equation}
However, this section is not invariant under inversion. In contrast, we propose to take the section $\alpha=-\gamma$, which is invariant under inversion (since
$(\ul{R}_{\ul{e}_{z},\gamma}\ul{R}_{\ul{e}_{y},\beta}\ul{R}_{\ul{e}_{z},-\gamma})^{-1}=
 \ul{R}_{\ul{e}_{z},\gamma}\ul{R}_{\ul{e}_{y},-\beta}\ul{R}_{\ul{e}_{z},-\gamma}$).
Moreover, this choice for estimating the kernels in the group is natural, as it provides the
weakest upper bound kernel since by direct computation one has \mbox{$\alpha=-\gamma \Rightarrow c_{6}=0$}.
Finally, this choice indeed provides us the correct symmetry for the Gaussian approximation of $K^1_t(\ul{y},\ul{n})$ as stated in the following theorem, that is proven in \cite{portegies_new_2015}.

% \begin{figure}[t!]
%   \centering
%   \includegraphics[width=\textwidth]{symmanalysis4_conv.eps}
%   \caption{\emph{Top}: 2 positions and orientations, rotated around the $z$-axis. The graphs display $p_t(\bR_{\be_z,\alpha'}\by,\bR_{\be_z,\alpha'}\bn)$, for $p^{prev,1}_t$ in green, $p_t^{prev,2}$ in orange and $p_t^{new}$ in blue. Parameters $D_{33} = 1, D_{44}=0.02, t = 4$. \emph{Bottom}: $k(\by,\bn,\mathbf{0},\be_z)$ should be equal to $k(\mathbf{0},\be_z,\by,\bn)$, see Remark \ref{remark:symmetry}. The absolute sum of the difference of the two evaluations on a $5\times5\times5$ grid with at each point 42 uniformly distributed orientations is shown for $k^{prev,1}$, $k^{prev,2}$ and $k^{new}$, corresponding to $p^{prev,1}_t$, $p_t^{prev,2}$ and $p_t^{new}$, respectively, with coloring as above. Parameter $D_{33}= 1$ and $t, D_{44}$ are as shown.}
% \label{fig:symmanalysis}
% \end{figure}

\begin{theorem}
When the approximate kernel $K_t^{1,\mathrm{new}}$ on the quotient is related to the approximate kernel on the group $\tK_t^{1,\log}$ by
\begin{equation}\label{eq:defappkernel}
K_t^{1,\mathrm{new}}(\by,\bn(\beta,\gamma)) := \tK_t^{1,\log}(\bc(\by,\bR_{\be_z,\gamma}\bR_{\be_y,\beta}\bR_{\be_z,-\gamma})),
\end{equation}
i.e. we make the choice $\alpha = -\gamma$, we have the desired $\alpha$-left-invariant property
\begin{equation}
\tK_t^{1,\mathrm{new}}(\by,\bn) = K_t^{1,\mathrm{new}}(\bR_{\be_z,\alpha'}\by, \bR_{\be_z,\alpha'}\bn), \qquad \alpha' \in [0,2\pi].
\end{equation}
and the symmetry property
\begin{equation}\label{symm}
K_t^{1,\mathrm{new}}(\by,\bn) = K_t^{1,\mathrm{new}}(-\bR^T_{\bn}\by,\bR_{\bn}^T\be_{z})
\end{equation}
\end{theorem}

\subsection{Comparison Logarithmic Approximation to Exact Solutions}
Qualitatively, the approximation kernel can be used effectively in applications. Here $D_{44} > 0$ tunes the angular diffusion, $t>0$ regulates the size and $D_{33}$ tunes the spatial diffusion. However, from the general theory \cite{ter_elst_weighted_1998} we can only expect locally good approximations. In comparison to the exact kernels, these approximation kernels do not accurately solve the PDEs for these parameters. This is mainly due to the fact that the non-commutative structure of $SE(3)$ is only encoded in the logarithm (\ref{eq:log}), which is not sufficient. On top of this, we recall the required rescaling of time, which requires a smaller time in the Gaussian approximations than in the exact solutions. A visualization of the approximation kernel can be found in
Fig.~\!\ref{fig:logarithmicapp} which is qualitatively comparable to the exact kernel depicted in  Fig.~\!\ref{fig:exactCEkernels}.

  \begin{figure}[t!]
   \centering
   \includegraphics[width = 0.4\textwidth]{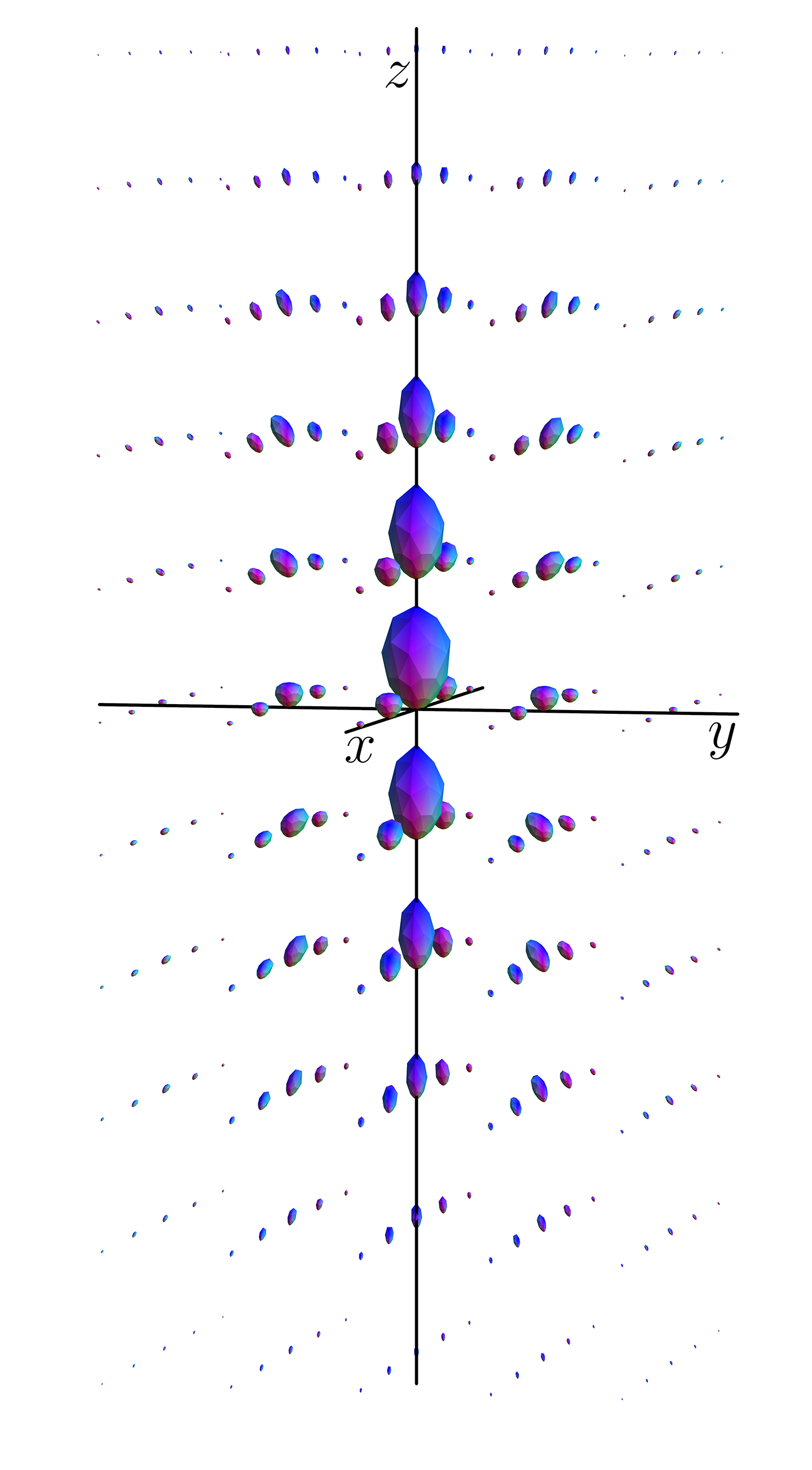}
 \caption{Glyph field visualization of the logarithmic approximation kernel $K_t^{1,\mathrm{new}}$, with $D_{33} = 1.$, $D_{44} = 0.24$, $t = 0.7$, $\xi = 16$. Although the parameters have to be chosen differently, the result is qualitatively comparable to the kernel in Fig. \ref{fig:exactCEkernels}.}\label{fig:logarithmicapp}
 \end{figure}

% \section{Application in dMRI}\label{se:applications}
% Convolution with the kernel of the contour enhancement PDE. Add references to older work for diffusion kernel and aODF paper by Reisert for convection-diffusion kernel.

\section{Conclusion}\label{se:conclusion}
We have provided the explicit solutions for the hypo-elliptic diffusion process for (restricted) Brownian motion on SE(3) and for the convection-diffusion process on SE(3), that is an extension of Mumford's 3D-direction process \cite{mumford_elastica_1994}. The solutions were derived by applying a Fourier transform in the spatial variables and a particular (frequency dependent) choice of coordinates, yielding for both processes a separable second order differential operator. Using the spectral decomposition of this operator, we have obtained a series expression for the solution kernel of the evolution equation and the kernel for the time-integrated process, related to the resolvent of this operator.

In the diffusion case, the eigenfunctions encountered in the spectral decomposition are similar to spherical harmonics, but require spheroidal wave functions instead of Legendre functions. Convergence of the series expressions in terms of the eigenfunctions was shown using Sturm-Liouville theory. The final expression for the exact solution is given in Theorem \ref{th:maintheoremCE}.

For the convection-diffusion case, generalized spheroidal wave functions are needed, that we derive applying a non-standard expansion in Legendre functions. In this case, in contrast to the diffusion case, the considered operator is no longer self-adjoint and as a result, standard Sturm-Liouville cannot be used for proving convergence of the series. Instead we prove completeness of the eigenfunctions and specific properties of the spectrum using perturbation theory of linear (self-adjoint) operators. The exact solution for the resolvent case is given in Theorem \ref{th:maintheoremCC}. 

We have also established a numerical algorithm on SE(3) that generalizes the algorithm in \cite{august_curve_2001,august_sketches_2003} for Mumford's 2D direction process to the 3D case. Using this method, approximate solutions can be found by rewriting the equation into an ODE in terms of expansion coefficients w.r.t. a rotated spherical harmonic basis. We connect this numerical algorithm to the exact solutions, showing that it in essence just diagonalizes (the matrix of) the resolvent operator. Truncation of the exact series representation of the kernels yield the same result. Results of the truncated kernels are shown in Figs. \ref{fig:exactCEkernels} and \ref{fig:exactCCkernel}.

Building on previous work on approximations of the kernels, we have included Gaussian estimates for the hypo-elliptic diffusion kernel and we have compared them qualitatively to the exact solution. The new approximation has the same symmetries as the exact kernel and by making use of a scaling factor, the approximation is good for reasonable parameter settings in practice.

Finally, we have algebraically derived the same exact solutions via harmonic analysis, using the Fourier transform on SE(3), see Theorem \ref{th:Fourierse3}.

Future work will include extensive analysis of the Monte-Carlo simulations, briefly described in \ref{app:stochastics}. Furthermore, we will put connections of hypo-elliptic diffusion (and Brownian bridges) on SE(3) and recently derived sub-Riemannian geodesics in SE(3) \cite{duits_sub-riemannian_2013}.

\section{Acknowledgements}

We gratefully acknowledge A.J.E.M. Janssen for all his useful comments on preliminary versions of this manuscript. Furthermore, we thank J.W. Portegies for his helpful input regarding the completeness of the (generalized) eigenfunctions of the generator of the convection-diffusion case. The research leading to the results of this paper has received funding from the European Research Council under the European Community's Seventh Framework Programme (FP7/2007-2013) / ERC grant \emph{Lie Analysis}, agr.~nr.~335555. \\
\includegraphics[width=0.2\hsize]{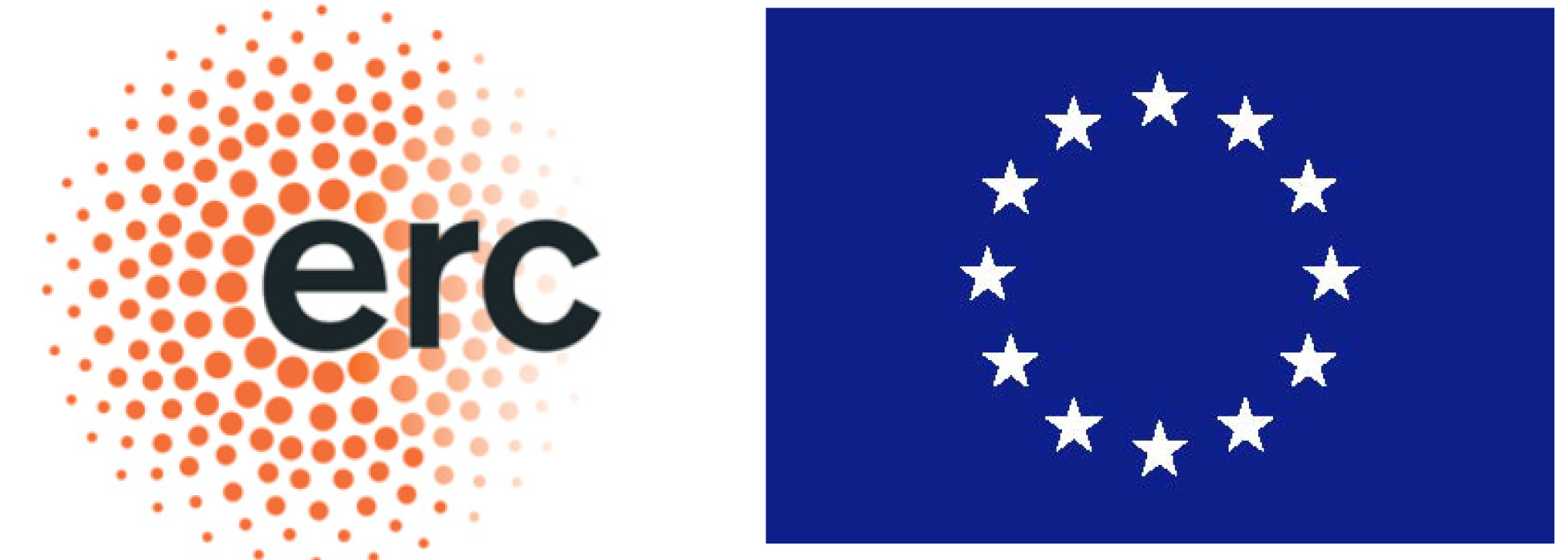}

%\section*{References}
\bibliographystyle{elsarticle-num}
\bibliography{exactsolutions}

\appendix

\section{Expansion of Spheroidal Wave Functions}\label{ap:SWF}
In this section we show how to obtain the eigenfunctions of the spheroidal wave equation \eqref{eq:spheroidalwave}. First note that when $\rho = r = 0$, corresponding to $\bomega = \mathbf{0}$, the spheroidal wave equation reduces to the Legendre differential equation, that has eigenvalues $-l(l+1)$ and eigenfunctions $P_l^m$, associated Legendre functions with $l \in \mathbb{N}_0$, $m \in \mathbb{Z}$, $|m| \leq l$. We immediately include a normalization factor in the definition of these functions:

\begin{equation}\label{def:plm}
P_l^{m}(x) = N^{l,m} (1-x^2)^{|m|/2} \frac{d^{|m|} P_l(x)}{dx^{|m|}}, \quad P_l(x) = \frac{1}{2^l l!} \frac{d^l}{dx^l}\left[(x^2 -1)^l \right], \quad -1 \leq x \leq 1.
\end{equation}
The normalization factor $N^{l,m}$ is given by

\begin{equation}
N^{l,m} = \sqrt{\frac{(2l+1)(l-|m|)!}{2(l+|m|)!}}.
\end{equation}
When $\rho > 0$ in Eq. (\ref{eq:spheroidalwave}), several series solutions are possible, but it is customary to use a series of associated Legendre polynomials \cite{flammer_spheroidal_1957,olver_nist_2010}:

 \begin{equation}\label{eq:expansionSlm}
 \Slmrho(x) = \sum_{j=0}^{\infty} d_{j}^{l,m} P_{m+j}^m(x), \qquad m \geq 0,
 \end{equation}
 where $d_j^{l,m} = d_j^{l,m}(\rho)$. For the case $m < 0$ it follows immediately from the property that $P_l^m (x)= P_l^{-m}(x)$ that:

 \begin{equation}
 d_j^{l,m} = d_j^{l,-m}, \qquad S^{l,m}_\rho(x) = S^{l,-m}_\rho(x) = \sum_{j=0}^{\infty} d_j^{l,|m|} P_{|m|+j}^{m}(x), \qquad m \in \mathbb{Z}.
 \end{equation}
 We use the following identity to normalize the $\Slm_\rho$:

 \begin{equation}\label{eq:normSrholm}
 \int_{-1}^1 \Slmrho(x) S^{l',m}_\rho(x) \; \rmd x = \delta_{l l'} \; \sum_{j=0}^{\infty} |d_j^{l,m}|^2 =: \delta_{l l'} \; ||\bd^{l,m}||^2, \qquad \bd^{l,m} := (d_j^{l,m})_{j=0}^{\infty}.
 \end{equation}

Our solutions are of the form $S_\rho^m(x) = \sum_{j=0}^\infty d_j^m(\rho) P_{m+j}^m(x)$, where in this appendix we only consider $m \geq 0$. For shortness, we will omit the dependence on $\rho$ of the coefficients. It will follow later that for every $m$ a countable number of solutions exist for $S_\rho^m$. For $m$ fixed, substitution of the series in the differential equation (\ref{eq:spheroidalwave}) gives the following identity:

\begin{equation}\label{eq:derivrec1}
 \sum_{j=0}^{\infty} (-\rho^2 x^2 + \tlambda_\rho - (m+j)(m+j+1)) \; d_j^m P_{m+j}^m (x) = 0, \textrm{ i.e.,}
\end{equation}

\begin{equation}\label{eq:derivrec2}
 -\rho^2 \sum_{j = 0}^{\infty} d_j^m x^2 P_{m+j}^m(x) + \sum_{j= 0}^{\infty} (\tlambda_\rho - (m+j)(m+j+1)) d_j^m P_{m+j}^m(x) = 0.
\end{equation}
By applying the identity given in (\ref{eq:relationPlm}), we can expand the $x^2 P_{m+j}^m$ term in Legendre polynomials $P_{m+j-2}^m$, $P^m_{m+j}$ and $P^m_{m+j+2}$. By substituting Eq. (\ref{eq:relationPlm}) in Eq. (\ref{eq:derivrec2}) and by equating coefficients of $P_{m+j}^m(x)$, the following relation for the $d$'s can be found:

%  \begin{multline}
%  -\rho^2 \sum_{j = 0}^{\infty} d_j^m (\alpha^{m+j,m} P_{m+j-2}^m(x) + \eta^{m+j,m} P_{m+j}^m(x)+ \zeta^{m+j,m} P_{m+j+2}^m(x)) \\ + \sum_{j= 0}^{\infty} (\tlambda_\rho - (m+j)(m+j+1)) d_j^m P_{m+j}^m(x) = 0,
%  \end{multline}
%  with constants $\alpha,\beta$ and $\gamma$ defined as

%  \begin{equation}
% \begin{aligned}
% \alpha^{l,m} &= \frac{(l+m-1)(l+m)}{(2l-1)(2l+1)}, \quad \eta^{l,m} = \frac{2l(l+1)-2m^2-1}{(2l-1)(2l+3)}, \quad \zeta^{l,m} = \frac{(l-m+1)(l-m+2)}{(2l+3)(2l+1)}.
% \end{aligned}
% \end{equation}

\begin{equation}
\rho^2 \frac{N^{m+j+2,m}}{N^{m+j,m}} \alpha^{m+j+2,m} d_{j+2}^m + (\rho^2 \eta^{m+j,m} + (m+j)(m+r+1)) d_j^m + \rho^2 \frac{N^{m+j-2,m}}{N^{m+j,m}} \zeta^{m+j-2,m}d_{j-2}^m = \tlambda_\rho d_{j}^m.
\end{equation}
In matrix form, this equation can be written as

\begin{equation}\label{eq:matrixeqCE}
(\rho^2 \bM_1^m + \mathbf{\Lambda}^m) \bd^m = \tlambda_\rho \bd^m, \qquad \bd^m = (d^m_0, d^m_{1}, \dots)^T,
\end{equation}
with $\bM_1^m$ as in (\ref{eq:definitionM1}) and $\mathbf{\Lambda}^m$ as in (\ref{eq:matrixeqWm}).

% \begin{equation}\label{eq:matrixeqCE}
% \begin{aligned}
% (\rho^2 \bM_1^m + \mathbf{\Lambda}^m) \underline{d}^m &= \tlambda_\rho \underline{d}^m, \qquad \text{with}\\
% \bM_1^m &= \begin{pmatrix}
% \eta^{m,m} & 0 & \alpha^{m+2,m} &  & \text{\textbf{\large{O}}} & \\
%  & \eta^{m+1,m} & 0 & \alpha^{m+3,m} &  & \\
% \zeta^{m,m} & 0 & \eta^{m+2,m} & 0 & \alpha^{m+4,m} & \\
%  & \zeta^{m+1,m} & 0 & \eta^{m+3,m} & 0 & \alpha^{m+5,m} \\
%  \text{\textbf{\large{O}}} & & \ddots & & \ddots & & \ddots
% \end{pmatrix}, \qquad \text{and} \\
% \mathbf{\Lambda}^m &= \begin{pmatrix}
% m(m+1) & &\text{\textbf{\large{O}}} &\\
% & (m+1)(m+2) & &\\
% & & (m+2)(m+3)& \\
% \text{\textbf{\large{O}}} & & & \ddots 
% \end{pmatrix}.
% \end{aligned}
% \end{equation}

Then the eigenvalues of the matrix on the left are spheroidal eigenvalues, that we denote with $\tlambda_\rho^{l,m}$, $l \geq |m|$, such that $\tlambda_\rho^{m,m} < \tlambda_\rho^{m+1,m} < \tlambda_\rho^{m+2,m} < \dots$. The corresponding eigenvectors form the constant vectors $\bd^{l,m}$ and thereby the functions $\Slmrho(x)$ are determined up to normalization. It follows from the form of the matrix in (\ref{eq:matrixeqCE}) that either the even or the odd coefficients of $\bd^{l,m}$ are 0, resulting in only even and odd functions $\Slmrho(x)$. Recall that in our case $\rho = \sqrt{\frac{D_{33}}{D_{44}}} r >0$ and the eigenvalues corresponding to the eigenfunctions $\Philmom$ are $\lambdalmr = - D_{44} \tlambda_\rho^{l,m}$, and thereby

\begin{equation}\label{eq:finaleigvproblem}
(r^2 D_{33} \bM^m_1 + D_{44}\mathbf{\Lambda}^m)\bd^{l,m} = -\lambda_r^{l,m}\bd^{l,m}.
\end{equation}

We conclude that the spheroidal wave functions are given by

\begin{equation}\label{def:spheroidalwavefunctions}
S_\rho^{l,m}(x) = \sum_{j=0}^{\infty} d_j^{l,|m|}(\rho) \;P_{|m|+j}^m(x), \qquad l \in \mathbb{N}_0, \; m \in \mathbb{Z}, \; |m| \leq l,
\end{equation}
where the vectors $\bd^{l,m}(\rho)$ are the solutions to Eq. (\ref{eq:finaleigvproblem}).

\section{Generalized Spheroidal Wave Functions}\label{ap:GSWF}
Similar to the case of the spheroidal wave functions, we can derive eigenfunctions of the generalized spheroidal wave equation by substituting $GS_\rho^m(x) = \sum_{j=0}^{\infty} c_j^{m}(\rho) P_{m+j}^m(x)$ in Eq. (\ref{eq:maineqCC}). For shortness, we omit the dependence on $\rho$ of the coefficients and we only consider the case $m \geq 0$. This yields

 \begin{equation}\label{eq:derivrec3}
 -i \rho \sum_{j = 0}^{\infty} c_j^m x P_{m+j}^m(x) + \sum_{j= 0}^{\infty} (\tlambda_\rho - (m+j)(m+j+1)) c_j^m P_{m+j}^m(x) = 0.
 \end{equation}
In our case $\rho = \frac{r}{D_{44}} > 0$ and the eigenvalues $\lambda_r^{l,m}$ of the eigenfunctions $\Psilmom$ are given by $\lambdalmr = -D_{44} \tilde{\lambda}_\rho^{l,m}$. Now applying (\ref{eq:relationSH}) once to rewrite the term $x P_{m+j}^m(x)$, we get

\begin{multline}
-i \rho \sum_{j=0}^\infty c_j^m \frac{N^{m+j}}{N^{m+j-1}} \nu^{m+j,m} P_{m+j-1}^m(x) + \frac{N^{m+j}}{N^{m+j+1}} c_j^m \xi^{m+j,m} P_{m+j+1}^m(x) \\ + \sum_{j=0}^\infty (\tilde{\lambda}_\rho - (m+j)(m+j+1))\; c_j^m P_{m+j}^m(x) = 0.
\end{multline}
Again, equating coefficients of $P_{m+j}^m$, we get

\begin{equation}
i \rho \nu^{m+j+1,m} c_{j+1}^m + i \rho \xi^{m+j-1,m} c_{j-1}^m + (m+j)(m+j+1)c_j^m = \tilde{\lambda}_\rho c_j^m, \qquad j \in \mathbb{N}_0.
\end{equation}
In matrix form:

\begin{equation}\label{eq:finaleigvproblemc}
\begin{aligned}
(i \rho \bM_2^m + \mathbf{\Lambda}^m) \bc^m &= \tilde{\lambda}_\rho \bc^m, \qquad \bc^m = (c_0^m, c_1^m \dots)^T,
\end{aligned}
\end{equation}
with $\bM^m_2$ as in (\ref{eq:matrixM2}) and $\mathbf{\Lambda}^m$ as in (\ref{eq:matrixeqWm}).

% \begin{equation}
% \begin{aligned}
% (i \rho \bM_2^m + \mathbf{\Lambda}^m) \underline{c}^m &= \tilde{\lambda}_\rho \underline{c}^m, \qquad \text{with}\\
% \bM_2^m &= \begin{pmatrix}
% 0 & \alpha^{m+1,m} &  & \text{\textbf{\large{O}}} &  \\
% \zeta^{m,m} & 0 & \alpha^{m+2,m} &  &  \\
%  & \zeta^{m+1,m} & 0 & \alpha^{m+3,m} & \\
%  \text{\textbf{\large{O}}} & & \ddots &  & \ddots 
% \end{pmatrix},
% \end{aligned}
% \end{equation}
For each fixed $m$, eigenvalues $\tilde{\lambda}_\rho$ can again be numbered as $\tilde{\lambda}_\rho^{l,m}$, $l \geq |m|$, with eigenvectors $\bc^m$. This determines the eigenfunctions $GS_\rho^{l,m}$ up to a normalization constant:

\begin{equation}\label{eq:GSlmrho}
GS_\rho^{l,m}(x) = \sum_{j=0}^{\infty} c_j^{l,|m|}(\rho) \;P_{|m|+j}^m(x), \qquad l \in \mathbb{N}_0, \; m \in \mathbb{Z}, \; |m| \leq l,
\end{equation}
where the vectors $\bc^{l,m}(\rho)$ are the solutions to Eq. (\ref{eq:finaleigvproblemc}).

\section{Proof of Theorem \ref{thm:gammadist}}\label{app:proof}
We now give a proof for Theorem \ref{thm:gammadist}, that stated the relation between time-integration of the diffusion and convection-diffusion kernels and the repeated convolution of resolvent kernels.
\begin{proof}
When $T \sim \Gamma(k,\alpha)$, we can also write $T = \sum_{j=1}^k T_k$ with i.i.d. $T_k \sim Exp(\alpha)$ and corresponding pdf $P_\alpha(t) := \alpha e^{-\alpha t}$. When $k = 1$, we have seen that $R_\alpha^i = - \alpha (Q_i - \alpha I)^{-1} \delta_{(\mathbf{0},\be_z)}$. When $k > 1$, we have 

\begin{equation}
\begin{aligned}
\int_0^\infty K_t^i \; \Gamma(t;k,\alpha) \; \rmd t &= \int_0^\infty e^{t Q_i}\;\delta_{(\mathbf{0},\be_z)} \;(P_\alpha \ast_{\mathbb{R}^+}^{(k-1)} P_\alpha) (t) \; \rmd t \\
&= \left[\alpha \left[\mathcal{L} (t \mapsto e^{t Q_i} \delta_{(\mathbf{0},\be_z)})\right](\alpha)\right]^k  \\
&= \alpha^k \left[(\alpha I - Q_i)^{-k} \right]\delta_{(\mathbf{0},\be_z)} = R_\alpha^i \ast_{\quot}^{(k-1)} R_\alpha^i,
\end{aligned}
\end{equation}
with $\mathcal{L}$ the Laplace transform on $\mathbb{R}^+$ and $\ast_{\mathbb{R}^+}$ the convolution on $\mathbb{R}^+$. In the second step we used the property that $\mathcal{L}(f \ast_{\mathbb{R}^+} g) = (\mathcal{L}f) (\mathcal{L}g)$.

Regarding the singularities of these kernels, we recall that

\begin{equation}
\begin{aligned}
\hK_t^1(\bomega,\bn)  &= \sumlm \overline{\Philmom(\be_z)} \Philmom(\bn) \;e^{\lambda_r^{l,m}t}, \qquad \lambda_r^{l,m} = - D_{44} l (l+1) + \mathcal{O}(r^2), \\
\hK_t^2(\bomega,\bn)  &= \sumlm \overline{\Psilmom(\be_z) \Psilmom(\bn)} \;e^{\lambda_r^{l,m}t}, \qquad \lambda_r^{l,m} = - D_{44} l (l+1) + \mathcal{O}(r).
\end{aligned}
\end{equation}
All eigenfunctions $\Philmom$, $\Psilmom$ are bounded on the compact set $S^2$. Consider the case $i = 1$, then \\
 $\mF_{\mathbb{R}^3} \left[ R_\alpha^1 \ast_{\quot}^{(k-1)} R_\alpha^1 \right] (\bomega,\bn)$  is $\mathbb{L}_1$-integrable w.r.t $\bomega$ if

\begin{equation}
\int_0^\infty \left(\frac{\alpha}{l(l+1)D_{44} + \mathcal{O}(r^2)+ \alpha} \right)^k r^2 \rmd r < \infty.
\end{equation}
This implies $k > \frac{3}{2}$. For the case $i = 2$, it follows analogously that $k >3$. The fact that this implies that no singularities occur when these conditions for $k$ are satisfied follows from \cite[Thm. 7.5]{rudin_functional_2006}.

\end{proof}

\section{Equivalent Solutions via the Fourier transform on SE(3)\label{ap:Fourierse3}}
 In this appendix, we show how to obtain the solutions of the diffusion and convection-diffusion equations using the Fourier transform on SE(3), rather than a Fourier transform in only the spatial coordinates. We write $G = SE(3)$ for shortness. Recall from Section \ref{sec:introse3} that we denote the rigid body motions with $g=(\ul{x},\ul{R}) \in SE(3)=\R^{3} \rtimes SO(3)$. Now $G$ is a unimodular Lie group (of type I) with (left- and right-invariant) Haar measure ${\rmd}g = {\rm d}\bx \rmd \mu_{SO(3)}(\bR)$ being the product
  of the Lebesgue measure on $\mathbb{R}^{3}$ and the Haar-Measure $\mu_{SO(3)}$ on $SO(3)$. We denote the unitary dual space of $G$ with $\hG$. Its elements are equivalence classes of unitary, irreducible group representations
  $\sigma: G \rightarrow B(H_{\sigma})$, where $B(H_{\sigma})= \{A: H_\sigma \rightarrow H_\sigma \; | \; A \text{ linear and trace}(A^*A)< \infty \}$ denotes the space of bounded linear trace-class operators on a Hilbert space $H_{\sigma}$ (on which each operator $\sigma_{g}$ acts). In the dual space $\hG$ we identify elements via the following equivalence relation:
 \[
 \sigma_1 \sim \sigma_2 \Leftrightarrow \textrm{ there exists a unitary operator $U$, s.t. } \sigma_1 = U \circ \sigma_2 \circ U^{-1}.
 \]
 Then for all $f \in \mathbb{L}_1(G) \cap \mathbb{L}_2(G)$ the Fourier transform $\mathcal{F}_{G}f$ is given by
 \begin{equation}
\hat{f}(\sigma) = (\mathcal{F}_{G}(f))(\sigma) = \int \limits_G f(g)\; \sigma_{g^{-1}} \rmd g \; \in B(H_\sigma), \textrm{ for all } \sigma \in \hat{G},
 \end{equation}
where  $\mathcal{F}_{G}$ maps $\mathbb{L}_2(G)$ unitarily onto the direct integral space $\int_{\hat{G}}^\oplus B(H_\sigma) \; \rmd \nu(\sigma)$ with $\nu$ the Plancherel measure on $\hat{G}$. For more details see \cite{fuhr_abstract_2005,sugiura_unitary_1990,folland_course_1994}.
One has the inversion formula:
 \begin{equation}
f(g) = (\mF^{-1}_G \mF_G f)(g) = \int_{\hat{G}} \text{trace}\{(\mF_G f)(\sigma)\; \sigma_g \}\; \rmd \nu(\sigma).
 \end{equation}
 In our Lie group case of $SE(3)$ we identify all irreducible representations $\sigma^{p,s}$ having non-zero dual measure with the pair $(p,s) \in \mathbb{R}^{+} \times \mathbb{Z}$. This identification is commonly applied, see e.g.\cite{chirikjian_engineering_2000}. All unitary irreducible representations (UIR's) of $G$, up to equivalence, with non-zero Plancherel measure are given by \cite{mackey_imprimitivity_1949,sugiura_unitary_1990}:
 %\footnote{The unitary dual space $\hat{G}$ may be identified with the spectrum of the $C^*$-algebra associated to $G$, \cite[Thm 3.9]{williams_lecture_2011}.} all irreducible representations $\sigma^{p,s}$ having non-zero dual measure with the pair $(p,s) \in \mathbb{R}^{+} \times \mathbb{Z}$.
%\footnote{That they are, up to unitary equivalence, all UIR's of $SE(3)$ with nonzero Plancherel measure follows by results in \cite{mackey_imprimitivity_1949,sugiura_unitary_1990}.}   
%\footnote{UIR's of $SE(3)$ with zero Plancherel measure are e.g. given by $\mathcal{\sigma}_{(\mathbf{x},\mathbf{R})}= \mathcal{U}_{\ul{R}}$ with $\mathcal{U}$ a UIR of $SO(3)$.} 

 \begin{equation} \label{inter}
 \begin{array}{l}
\sigma = \sigma^{p,s}: SE(3) \to B(\mathbb{L}_{2}(p \, S^2)), \qquad p > 0, s \in \mathbb{Z}, \textrm{ given by }\\[7pt]
 (\sigma^{p,s}_{(\bx,\bR)} \phi)(\bu) = e^{- i \bu \cdot \bx} \phi(\bR^{-1}\bu) \; \Delta_s \left(\bR^{-1}_{\frac{\bu}{p}} \bR \bR_{\frac{\bR^{-1}\bu}{p}}\right), \ \ \ \bu \in p S^{2}, \phi \in \mathbb{L}_{2}(p S^2),
 \end{array}
 \end{equation}
where $\Delta_s$ is a representation of $SO(2)$ (or rather of the stabilizing subgroup $\text{stab}(\be_z) \subset SO(3)$ isomorphic to $SO(2)$) producing a scalar. In (\ref{inter}) $\bR_{\frac{\bu}{p}}$ denotes a rotation that maps $\ul{e}_{z}$ onto $\frac{\bu}{p}$.
So direct computation
\[
\bR^{-1}_{\frac{\bu}{p}} \bR \bR_{\frac{\bR^{-1}\bu}{p}}\ul{e}_{z}= \bR^{-1}_{\frac{\bu}{p}} \bR \bR^{-1}\left(\frac{\bu}{p}\right)= \ul{e}_{z},
\]
shows us that it is a rotation around the $z$-axis, say about angle $\phi$. This yields character $\Delta_s (\bR^{-1}_{\frac{\bu}{p}} \bR \bR_{\frac{\bR^{-1}\bu}{p}})=e^{- i s \phi}$, for details see \cite[ch.10.6]{chirikjian_engineering_2000}.

Furthermore, the dual measure $\nu$ can be identified with the Lebesgue measure on $p S^{2}$ and we have
\[
d \nu (\sigma^{p,s}) \equiv p^2 \rmd p \textrm{ for all }p>0, s \in \mathbb{Z}.
\]
%\begin{remark} \label{rem:ok}
%%Note that for $s \neq 0$ the choice of rotation $\bR_{\frac{\bu}{||\bu||}}$
%%makes a difference, whereas only for $s=0$ the choice does not matter, since
%%$\Delta_s ((\bR_{\frac{\bu}{||\bu||}}\, \bR_{\ul{e}_{z},\alpha} )^{-1} \bR \bR_{\frac{\bR^{-1}\bu}{||\bu||}} \bR_{\ul{e}_{z},\alpha})= e^{- i s (\phi-\alpha)}$. In view of our homogeneous space restriction $\mathbb{R}^{3} \rtimes S^{2}:=SE(3)/(\{\ul{0}\} \times SO(2))$ we must restrict ourselves to the case $s=0$.
%\end{remark}
%As a result one has
%\begin{equation}
%(\mF_G f) (\sigma^{p,s}) = \int_G f(g) \sigma_{g^{-1}}^{p,s} \rmd g,
%\end{equation}
%and,
The matrix elements of $\mF_G f$ w.r.t. an orthonormal basis of modified spherical harmonics $\{Y^{l,m}_s(p^{-1}\cdot)\}$, with $|m|, |s| \leq l$, for $\mathbb{L}_{2}(p \, S^2)$ are given by
\begin{equation} \label{matrixone}
\hat{f}^{p,s}_{l,m,l',m'} := \int \limits_G f(g) \; \;(\, Y^{l,m}_s(p^{-1} \cdot)\, , \,\sigma_{g^{-1}}^{p,s} Y^{l',m'}_s(p^{-1}\cdot)\, )_{\mathbb{L}_{2}(p \, S^2)}\; \, \rmd g.
\end{equation}
For an explicit formula for the modified spherical harmonics, see \cite{chirikjian_engineering_2000}, where they are denoted with $h_{m,s}^{l}$. Moreover, we have inversion formula (\cite[Eq.10.46]{chirikjian_engineering_2000}):
\begin{equation} \label{C5}
f(g) = \frac{1}{2 \pi^2} \sum_{s \in \mathbb{Z}} \sum_{l' = |s|}^{\infty} \sum_{l = |s|}^\infty \sum_{m'=-l}^l \sum_{m=-l}^l \int \limits_0^\infty \hat{f}^{p,s}_{l,m,l',m'}\, (\sigma^{p,s}_g)_{l',m',l,m} \; p^2 \rmd p,
\end{equation}
with matrix coefficients (independent of $f$) given by
\begin{equation} \label{matrix}
(\sigma^{p,s}_g)_{l',m',l,m}= (\, \sigma^{p,s}_g Y^{l,m}_s(p^{-1} \cdot), Y^{l',m'}_{s}(p^{-1} \cdot) \,)_{\mathbb{L}_{2}(p \, S^2)}.
\end{equation}
Note that $\sigma^{p,s}$ is a UIR so we have
\begin{equation}\label{property}
(\sigma^{p,s}_{g^{-1}})_{l',m',l,m}= \overline{(\sigma^{p,s}_g)_{l,m,l',m'}}\,.
\end{equation}
%\subsection{Fourier transform on $\mathbb{L}_{2}(\mathbb{R}^{3}\rtimes %S^2)=\mathbb{L}_{2}(SE(3)/(\{\mathbf{0}\}\times SO(2)))$.}

Now in our position and orientation analysis we must constrain ourselves to a special type of functions $f=\tU \in \mathbb{L}_{2}(G)$, namely the ones that have the property that $f(\bx,\bR)=\tU(\bx,\bR) = U(\bx, \bR \be_z)$. Then it follows by \cite[eq.10.35]{chirikjian_engineering_2000} and by (\ref{property}) that the only non-zero matrix elements (\ref{matrix}) have the property that both $m=0$ and $m'=0$. This can also be observed in the analytical examples in \cite[ch:10.10]{chirikjian_engineering_2000}.
This reduces the five-fold sum in (\ref{C5}) to a three-fold sum, and the modified spherical harmonics become standard spherical harmonics.

% \todo[inline]{I included the following definition, but what is $H$ here? Also, is it necessary to introduce new notation for a general Lie algebra element and group representations? Or can we directly use this definition in the form of the first line in (C.8)}

Next we introduce some notation for generators of group representations of SE(3). 

\begin{definition}\label{def:generatorrepr}
Let $A \in T_{e}(G)$ denote an element of the Lie algebra of $G = SE(3)$, with unity element $e=(\ul{0},I) \in G$. Then given a group representation $\mathcal{U}: G \to B(H)$, where $B(H)$ denotes the space of bounded linear operators on some Hilbert space $H$, we define its generator $\rmd \mathcal{U}(A) : \mathcal{D}_H \rightarrow H$ by
\[
{\rm d}\mathcal{U}(A)f := \lim \limits_{t \to 0} t^{-1}\, (\mathcal{U}_{e^{tA}}-I)f, \qquad \text{for all } f \in \mathcal{D}_H,
\]
where $\mathcal{D}_H \subset H$ is a domain of sufficiently regular $f \in H$ such that limit ${\rm d}\mathcal{U}(A)f$ exists and is in $H$.
\end{definition}

Now let $\mathcal{R}: G \to B(\mathbb{L}_{2}(G))$ denote the right-regular representation given by $(\mathcal{R}_{h}f)(g)=f(g h)$, $f \in \mathbb{L}_{2}(G)$ and $g,h \in G$. 
Let $\left.\mathcal{A}_{i}\right|_{g}= (L_{g})_{*} A_{i}$ denote the left-invariant vector fields on $G$ relative to basis $\{A_{i}\}$ in the Lie-algebra $T_{e}(G)$. Then the left-invariant vector fields $\{ \mA_i\}$ are obtained by the generators of the right-regular representation:
\begin{equation}
\mA_i f = \rmd \mathcal{R}(A_i)f := \lim_{t \rightarrow 0} t^{-1} \left(\mathcal{R}_{e^{t A_i}} - I \right)f, \qquad \text{for all } f \in \mathbb{H}(SE(3)).
\end{equation}
Here we note that due to our choice $H=\mathbb{L}_{2}(SE(3))$ in Definition~\ref{def:generatorrepr} the domain $\mathcal{D}_{H}$ of the generator $\rmd \mathcal{R}(A_i)$ becomes the first order Sobolev space $\mathbb{H}(SE(3))$.
% The right-regular representation commutes with the left-regular representation, and as a result the left-invariant vector fields $\{\mathcal{A}_{i}\}$ are obtained by the generators $\{{\rm d}\mathcal{R}(A_i)\}$ of the right-regular representation $\mathcal{R}$, which is the first identity.
In the lemma below we use this identity, together with Definition \ref{def:generatorrepr} applied to the UIR's $\sigma^{p,s}$. In fact the lemma allows us to map unbounded generators acting on sufficiently regular $f \in \mathbb{L}_{2}(G)$
to unbounded generators acting on sufficiently regular $\mathcal{F}_{G}f \in \mathbb{L}_{2}(\hat{G})$, which will be the key to our final theorem.

\begin{lemma}\label{lemma:appC2}
The following relation holds for the Fourier transform of derivatives of functions on $G$:
\begin{equation}
(\mF_{G} \circ \mA_i f)(\sigma^{p,s}) =  \rmd \sigma^{p,s}(A_i) \mF_{G}f(\sigma^{p,s}) 
% := \lim \limits_{h \rightarrow 0} \frac{\left(\sigma^{p,s}_{e^{hA_i}}-I\right)}{h} \mF_{G}f(\sigma^{p,s}).
\end{equation}
\end{lemma}
\begin{proof}
As a first step, we note the following relation for the composition of the right-regular representation and the Fourier transform:

\[
(\mF_G \circ \mathcal{R}_h f)(\sigma^{p,s}) = \int_G f(gh)\, \sigma^{p,s}_{g^{-1}} \rmd g = \int_G f(g') \, \sigma^{p,s}_{(g' h^{-1})^{-1}} \rmd g' = \sigma_h^{p,s} (\mF_G f)(\sigma^{p,s}) = \sigma_h^{p,s}(\mF_G f) (\sigma^{p,s}).
\]
The desired identity then follows from linearity and continuity of the Fourier transform $\mathcal{F}_{G}$. In more detail, we have:
\[
\begin{array}{ll}
{\rm d}\sigma^{p,s}(A_i) \mathcal{F}_{G}f (\sigma^{p,s})  &=
\lim \limits_{t \to 0} \frac{\sigma^{p,s}_{e^{t A_i}}- I}{t} \; \mathcal{F}_{G}f (\sigma^{p,s}) \\
&= \lim \limits_{t \to 0} \frac{(\mathcal{F}_{G} \circ \mathcal{R}_{e^{t A_i}}f ) (\sigma^{p,s})-\mathcal{F}_{G}f(\sigma^{p,s}) }{t} \\
&= \left(\mathcal{F}_{G} \circ \lim \limits_{t \to 0} \frac{\mathcal{R}_{e^{tA_i}}f - f}{t}\right)(\sigma^{p,s}) = (\mF_{G} \circ \mA_i f)(\sigma^{p,s}),
\end{array}
\]
where subsequently we applied the definition of generator ${\rm d}\sigma^{p,s}(A_i)$ 
(recall Definition \ref{def:generatorrepr}), we used the second identity in (C.13) for the special case $h=e^{A_i}$, we used continuity of the Fourier transform and finally identity we used the definition of  ${\rm d}\mathcal{R} (A_i)$.
\end{proof}

We conclude the appendix with the following theorem, that shows the correspondence between the Fourier transform on $SE(3)$ and the formulation used earlier in the paper based on the Fourier transform on the Cartesian product space $\mathbb{R}^{3} \times S^2$.

\begin{theorem}\label{th:Fourierse3}
Let $j \in \{1,2\}$ label the diffusion PDE ($j=1$) and the convection-diffusion PDE ($j=2$).
The SE(3)-Fourier transform $\hat{\tK}_t^{j}=\mathcal{F}_{G}\tK_{t}^j$ of $\tK_t^j: SE(3) \rightarrow \mathbb{R}^+$ satisfies
\begin{equation}
\left\{ \begin{aligned}
&\partial_t \hat{\tK}_t^1 (\sigma^{p,s}) = \left[ D_{33} (\rmd \sigma^{p,s} (A_3)^2) + D_{44} \left( (\rmd \sigma^{p,s}(A_4))^2 + (\rmd \sigma^{p,s}(A_5))^2 \right)\right] \hat{\tK}_t^1 (\sigma^{p,s}), \\
& \hat{\tK}_0^1= 2\pi \, {\rm id},
\end{aligned}\right.
\end{equation}
whereas for $j=2$ we have
\begin{equation}
\left\{ \begin{aligned}
&\partial_t \hat{\tK}_t^2 (\sigma^{p,s}) = \left[ -\rmd \sigma^{p,s} (A_3) + D_{44} \left( (\rmd \sigma^{p,s}(A_4))^2 + (\rmd \sigma^{p,s}(A_5))^2 \right)\right] \hat{\tK}_t^2 (\sigma^{p,s}), \\
& \hat{\tK}_0^2= 2\pi \, {\rm id}.
\end{aligned}\right.
\end{equation}
When we set
\begin{equation} \label{ident}
\boxed{
p = r=||\bomega|| \textrm{ and } \bu = -||\bomega|| \; \bR_{\frac{\bomega}{||\bomega||}}^{T} \bn,
}
\end{equation}
this system is for $j=1$ fully equivalent to Eq. (\ref{eq:augusteqCE}), and for $j=2$ (up to temporal Laplace transform) fully equivalent to (\ref{eq:augusteqCC})
with initial condition $\hat{W}(\bomega,\ul{n},0)=\delta_{\ul{e}_{z}}$.
In fact the solution of the corresponding PDE-evolutions (\ref{eq:differentialQ}) can be rewritten as
\begin{equation} \label{fun}
\tilde{W}(\ul{y},\ul{n},t)=\mathcal{F}_{SE(3)}^{-1}\left[\, \mathcal{F}_{SE(3)}\tilde{K}_{t}^{j}\; \mathcal{F}_{SE(3)}\tilde{W}(\cdot,0)\, \right](\ul{y},\ul{R}_{\ul{n}}).
\end{equation}
\end{theorem}
\begin{proof}
Follows by application of Lemma~\ref{lemma:appC2} and application of the Fourier transform on $SE(3)$ to the PDE-systems for contour enhancement and contour completion.
Regarding the statement of obtaining equivalent systems under our identification (\ref{ident}) we note that regarding the spatial generator that
\begin{equation} \label{C6}
\begin{array}{lll}
\left(\rmd \sigma^{p,s} (A_3) \phi \right)(\bu) &= \lim \limits_{h \downarrow 0} \frac{(\sigma^{p,s}_{(h \ul{e}_{z},I)}\phi)(\bu)- \phi(\bu)}{h} &=  - i p \left(\be_z \cdot \frac{\bu}{\|\bu\|} \right) \phi(\bu), \\[6pt]
(\rmd \mathcal{R}(A_3) \tilde{W})(\mathbf{y}, \mathbf{R}_{\mathbf{n}},t)&= \lim \limits_{h \downarrow 0}  \frac{(\mathcal{R}_{(h \ul{e}_{z},I)}\tilde{W})(\mathbf{y}, \mathbf{R}_{\mathbf{n}},t)- \tilde{W}(\mathbf{y}, \mathbf{R}_{\mathbf{n}},t)}{h} &=\mathcal{A}_{3} \tilde{W}(\mathbf{y}, \mathbf{R}_{\mathbf{n}},t),   \\
\mathcal{F}_{\mathbb{R}^{3}}[\rmd \mathcal{R}(A_3) \tilde{W}(\cdot, \mathbf{R}_{\mathbf{n}},t)](\bomega),
 &=(\mathcal{F}_{\mathbb{R}^{3}}\ul{n} \cdot \nabla W(\cdot,\mathbf{n},t))(\bomega) &=  i r \left(\ul{n} \cdot \frac{\bomega}{\|\bomega\|} \right) [\mathcal{F}_{\mathbb{R}^{3}}W(\cdot, \mathbf{n},t)](\bomega).
\end{array}
\end{equation}
So in order to match the multiplier operators at the right hand side of (\ref{C6}) we indeed need (\ref{ident}), where we note that
$\mathbf{R}_{\frac{\bomega}{\|\bomega\|}} \ul{e}_{z} =\frac{\bomega}{\|\bomega\|}$.

Regarding the angular generators we note that the following generator is independent of $s$:
\[
\begin{array}{ll}
({\rm d}\sigma^{p,s}(A_4))^2 + ({\rm d}\sigma^{p,s}(A_5))^2 &= \mathcal{F}_{G}\circ \left( ({\rm d}\mathcal{R}(A_4))^2 + ({\rm d}\mathcal{R}(A_5))^2 \right) \circ \mathcal{F}_{G}^{-1} \\ &=\mathcal{F}_{G} \circ \left( (\mathcal{A}_{4})^2 + (\mathcal{A}_{5})^2 \right) \circ \mathcal{F}_{G}^{-1},  \\
%\\[6pt]
%({\rm d}\bar{\sigma}^{p}(A_4))^2 + ({\rm d}\bar{\sigma}^{p}(A_5))^2  &=\mathcal{F}_{\mathbb{R}^{3}\rtimes S^{2}} \circ \Delta_{S^2}  \circ \mathcal{F}_{\mathbb{R}^{3}\rtimes S^{2}}^{-1}\ .
\end{array}
\]
and recall that
\[
\left(\left((\mathcal{A}_{4})^2 + (\mathcal{A}_{5})^2\right) \tilde{W}\right) (\ul{y},\ul{R}_{\ul{n}},t)=
\left(\left((\mathcal{A}_{4})^2 + (\mathcal{A}_{5})^2 + (\mathcal{A}_{6})^2\right) \tilde{W}\right) (\ul{y},\ul{R}_{\ul{n}},t)= (\Delta_{S^{2}} W)(\ul{y},\ul{n},t).
\]
As a result, the generators of the linear evolutions relate via conjugation with the Fourier transform on $SE(3)$
%and $\mathbb{R}^{3}\rtimes S^2$
and thereby the same applies to the evolutions themselves.
Finally, regarding (\ref{fun}), the expression is invariant under the choice of rotation $\ul{R}_{\ul{n}}$ mapping $\ul{e}_{z}$ onto $\ul{n}$ (since only the matrix coefficients with $m=m'=0$ contribute to the series (\ref{C5}) of the inverse Fourier transformation).
%that the kernels $K_{t}^j$, $j=1,2$, indeed satisfy the constraint needed for the final identity in (\ref{C8}) to hold, recall Corollary~\ref{corr:1}.
\end{proof}
\begin{remark}
Akin to the $SE(2)$-case (where the substitution showing the equivalence was similar, cf.~\cite[App.A, Thm A.2]{duits_explicit_2008}) we conclude that the tool of Fourier transform on $SE(3)$ is not strictly needed to find an explicit analytic description of the heat-kernels. Nevertheless, it is clearly present in our specific choice of coordinates (\ref{def:coordinatechoice}) in Lemma~\ref{le:laplacebeltrami} and Fig.~\ref{fig:figureParametrization}, that was crucial in our derivation of the exact solutions. In fact, application of the Fourier transform $\mathcal{F}_{SE(3)}$ on the Lie group $SE(3)$
yields isomorphic equations to the ones obtained by application of the operator
\[
({\rm id}_{\mathbb{L}_{2}(\mathbb{R}^{3})} \otimes \mathcal{F}_{S^{2}}) \circ \Xi \circ (\mathcal{F}_{\mathbb{R}^{3}} \otimes {\rm id}_{\mathbb{L}_{2}(S^2)})\ ,
\]
with operator $\Xi: \mathbb{L}_{2}(\mathbb{R}^{3} \times S^{2}) \to \mathbb{L}_{2}(\mathbb{R}^{3} \times S^{2})$ given by
\[
(\Xi\hat{U})(\bomega,\bn)= \hat{U}(\bomega,\mathbf{R}_{\frac{\bomega}{\|\bomega\|}}^{T}\bn),
\]
which relates to our substitution (\ref{ident}), and which is the actual net operator
we applied in our classical analytical approach pursued in the body of the article (Lemma~\ref{le:laplacebeltrami} and Figure~\ref{fig:figureParametrization}).
We conclude that in retrospect, in our classical analytic PDE-approach, we did solve the PDE's via Fourier transform on $SE(3)$.
\end{remark}

%-------------------------------------------------

\section{Monte Carlo Simulations of the (Convection-)Diffusion Kernels}\label{app:stochastics}
The PDEs for diffusion and convection-diffusion correspond to Fokker-Planck equations or forward Kolmogorov equations of certain stochastic processes. The solution kernels can thereby also be approximated with a Monte Carlo simulation, where the end points of random walks are accumulated. We briefly discuss how this simulation can be done.

% Before introducing the stochastic process, we first need to introduce some parameters. Let $t>0$ be the evolution time of the PDE or equivalently the traveling time of a particle and let $N \in \mathbb{N}$ be the number of steps taken by the particle. Let $D_{33}$>0 and $D_{44} = D_{55}$ as before. Let $(\varepsilon_k^i)_{k=1}^N$ be $N$ identically distributed random variables with $Var(\varepsilon_k^i) = 1$ and zero mean. Although the specific distribution of $(\varepsilon_k^i)$ is not crucial, we will assume that they are standard normally distributed: $\varepsilon_k^i \sim N(0,1)$.

Let $t > 0$, $D_{33} > 0$, $D_{44} >0$. We simulate $M$ random walks, with $N$ steps. We denote the position after $k$ steps with position $\bY_{k} \in \mathbb{R}^3$ and the orientation with $\bN_{k} \in S^2$. For the diffusion case, each random walk is given by:

\begin{equation}\label{eq:stochasticsenh}
\left\{ \begin{aligned}
\bY_{k+1} &= \bY_k + \varepsilon_{k+1} \sqrt{\frac{2 t D_{33}}{N}} \bN_{k}, \qquad \varepsilon_k \sim \mathcal{N}(0,1),\\
\bN_{k+1} &= \bR_{\bN_k} \bR_{\be_z,\gamma_{k+1}} \bR_{\be_y,\beta_{k+1}\sqrt{\frac{2tD_{44}}{N}}} \bR_{\bN_k}^T \bN_k,\qquad \gamma_k \sim \text{UNIF}(0,\pi), \qquad \beta_k \sim \mathcal{N}(0,1).
\end{aligned} \right.
\end{equation}
Here $\bR_{\bn}$ is any rotation matrix that maps $\be_z$ onto $\bn$ and $\bR_{\bn,\phi}$ denotes a counter-clockwise rotation of angle $\phi$ about axis $\bn$. Intuitively, every step in the random walk consists of a small normally distributed step in the current orientation and a small change in orientation. For the convection-diffusion case, we have

\begin{equation}\label{eq:stochasticscompl}
\left\{ \begin{aligned}
\bY_{k+1} &= \bY_k + \sqrt{\frac{t }{N}} \bN_{k}, \\
\bN_{k+1} &= \bR_{\bN_k} \bR_{\be_z,\gamma_{k+1}} \bR_{\be_y,\beta_{k+1}\sqrt{\frac{2tD_{44}}{N}}} \bR_{\bN_k}^T \bN_k,\qquad \gamma_k \sim \text{UNIF}(0,\pi), \qquad \beta_k \sim \mathcal{N}(0,1).
\end{aligned} \right.
\end{equation}
Each step in the random walk now consists of a small step of fixed size in the current orientation and a small change in orientation. The distribution is then computed simply by binning the endpoints of the random walks. We divide $\mathbb{R}^3$ in bins with size $\Delta x \times \Delta y \times \Delta z$. We divide $S^2$ into bins using the orientations of a three times refined icosahedron, where a correction factor is needed to take into account the difference in surface area of the triangles in the tessellation. The distributions shown in Fig. \ref{fig:kernelvis} are found with such a simulation. When $M,N \rightarrow \infty$ and the bin sizes tend to zero, this distribution converges to the exact kernels $K^i_t$ as found in Sections \ref{se:exactCE} and \ref{se:exactCC}. The proof of this, along with other results concerning the relation between the PDEs and their corresponding stochastic processes are left for future work.

\begin{remark}
An alternative stochastic formulation of \eqref{eq:stochasticsenh} expressed in the SE(3) Lie-algebra $\{A_i\}$, recall the notation in Section \ref{se:leftinvariantoperators}, is:

\begin{equation}
\left\{ \begin{aligned}
& (\bY_{k},\bN_k) = G_k \odot (\mathbf{0},\be_z) \\
& G_{k+1} = e^{\sqrt{\frac{2t}{N}} \sum \limits_{i \in \{3,4,5\}}\sqrt{D_{ii}}\varepsilon^i_{k+1}A_i}G_k, \qquad k = 0,\dots, N-1,
\end{aligned} \right.
\end{equation}
with $\varepsilon^i_{k+1} \sim \mathcal{N}(0,1)$ independent (with $\varepsilon^3_{k+1} = \varepsilon_{k+1}$ as in \eqref{eq:stochasticsenh}). Recall \eqref{odot} for the definition of the group action $\odot$. This stochastic formulation, where random paths move along exponential curves, yields the same limiting distribution. This can be seen by applying the theory of Brownian motion on an isotropic manifold \cite{pinsky_isotropic_1976} to the angular $S^2$-part.
\end{remark}
\end{document}